\documentclass[twoside]{article}
\usepackage[letterpaper, left=3.5cm, right=3.5cm, top=4cm, bottom=2.5cm]{geometry}
\usepackage{graphicx}
\usepackage{amsmath,amssymb,amsthm}
\usepackage{authblk}
\usepackage[utf8]{inputenc}
\usepackage{microtype}
\usepackage{mathrsfs} 
\usepackage{enumitem}
\usepackage{calc}
\usepackage{esint}
\usepackage{fancyhdr}
\usepackage[dvipsnames]{xcolor} 
\usepackage[labelfont = bf]{caption}
\usepackage{doi}
\usepackage{subfigure}
\usepackage[numbers]{natbib}
\usepackage{float}
\usepackage{stmaryrd}
\usepackage{mathtools}
\usepackage{booktabs}
\usepackage{colortbl}
\usepackage{tabulary}
\usepackage{diagbox}
\usepackage{caption}
\usepackage{cleveref}
\usepackage{commath}
\usepackage{multirow}
\usepackage{algorithm}
\usepackage{algpseudocode}
\usepackage{scalerel}

\usepackage{amsmath}
\usepackage{tikz}
\usepackage{mathdots}
\usepackage{yhmath}
\usepackage{cancel}
\usepackage{color}
\usepackage{siunitx}
\usepackage{array}
\usepackage{multirow}
\usepackage{amssymb}
\usepackage{tabularx}
\usepackage{extarrows}
\usepackage{booktabs}
\usetikzlibrary{fadings}
\usetikzlibrary{patterns}
\usetikzlibrary{shadows.blur}
\usetikzlibrary{shapes}
\usetikzlibrary{arrows.meta}

\newtheorem{theorem}{Theorem}[section]
\numberwithin{equation}{section}

\newtheorem{definition}[theorem]{Definition}

\newtheorem{remark}[theorem]{Remark}
\newtheorem{lemma}[theorem]{Lemma}

\newtheorem{assumption}[theorem]{Assumption}

\usepackage{titlesec}
\titleformat{\section}{\normalfont\scshape\centering}{\thesection.}{0.5em}{}
\titleformat*{\subsection}{\itshape}
\titleformat*{\subsubsection}{\itshape}

\providecommand{\keywords}[1]
{
	{\small\emph{Keywords:} #1}
}

\providecommand{\MSC}[1]
{
	{\small\emph{AMS MSC (2020):~~} #1}
}

\definecolor{denim}{rgb}{0.08, 0.38, 0.74}
\definecolor{byzantium}{rgb}{0.44, 0.16, 0.39} 
\definecolor{shamrockgreen}{rgb}{0.0, 0.62, 0.38} 
\definecolor{red}{rgb}{0.68, 0.05, 0.0}

\usepackage{hyperref}
\hypersetup{
	colorlinks=true,
	linkcolor=denim,
	citecolor = shamrockgreen,
	filecolor=magenta, 
	urlcolor=byzantium,
}

\usepackage[textsize=small]{todonotes}
\begin{document}
	\setlength{\abovedisplayskip}{5.5pt}
	\setlength{\belowdisplayskip}{5.5pt}
	\setlength{\abovedisplayshortskip}{5.5pt}
	\setlength{\belowdisplayshortskip}{5.5pt}

	\title{\vspace{-1cm}Pulsatile Flows for Simplified Smart Fluids with Variable \textcolor{black}{Exponent} Power-Law: Analysis and Numerics}
	\author[1]{Luigi C. Berselli\thanks{Email: \url{luigi.carlo.berselli@unipi.it}}\thanks{funded by  INdAM GNAMPA and MIUR (Italian Ministry of Education, University and Research)
			Excellence, Department of Mathematics,
  University of Pisa CUP I57G22000700001.}}
	\author[2]{Alex Kaltenbach\thanks{Email: \url{kaltenbach@math.tu-berlin.de}}}
	\date{July 31, 2025}
	\affil[1]{\small{Department of Mathematics, University of Pisa, Via F. Buonarroti 1/c, 56127~Pisa,~{Italy}}}
	\affil[2]{\small{Institute of Mathematics, Technical University of Berlin, Straße des 17. Juni 136, 10623~Berlin,~{Germany}}}
	\maketitle

	\pagestyle{fancy}
	\fancyhf{}
	\fancyheadoffset{0cm}
	\addtolength{\headheight}{-0.25cm}
	\renewcommand{\headrulewidth}{0pt} 
	\renewcommand{\footrulewidth}{0pt}
	\fancyhead[CO]{{Pulsatile Flows for Simplified Smart Fluids}}
	\fancyhead[CE]{{L.~C. Berselli and A. Kaltenbach}}
	\fancyhead[R]{\thepage}
	\fancyfoot[R]{}
	
	\begin{abstract}
        We study the fully-developed, time-periodic motion of a shear-dependent non-Newtonian fluid with variable exponent rheology through an infinite pipe $\Omega\coloneqq \mathbb{R}\times \Sigma\subseteq \mathbb{R}^d$, $d\in \{2,3\}$, of arbitrary cross-section $\Sigma\subseteq \mathbb{R}^{d-1}$. The focus is on a generalized $p(\cdot)$-fluid model, where the power-law index is position-dependent (with respect to $\Sigma$),~\textit{i.e.},~a~function~${p\colon \Sigma\to (1,+\infty)}$. We prove the existence of time-periodic solutions with either assigned time-periodic~\mbox{flow-rate} or pressure-drop, generalizing known results for the Navier--Stokes~and~for~$p$-fluid~\mbox{equations}. 
        
        In addition, we identify 
        explicit solutions, relevant as benchmark cases, especially for electro-rheological fluids or, more generally, \textit{`smart fluids'}. To support~\mbox{practical}~\mbox{applications}, we present a fully-constructive existence proof for variational solutions by means of a fully-discrete finite-difference/-element  discretization, consistent with our numerical experiments. Our approach, which unifies the treatment of all values of $p(\overline{x})\in (1,+\infty)$, $\overline{x}\in \Sigma$, without requiring an auxiliary Newtonian term, provides new insights even in the constant~\mbox{exponent} case. 
        The theoretical findings are illustrated by means of numerical experiments. 
	\end{abstract}
	
	
	\MSC{35Q35, 76A05, 76D05, 65M60, 35D30, 35K55, 35B10, 76M10, 35Q30}
	
	\section{Introduction}
    \thispagestyle{empty}\enlargethispage{15mm} 

\hspace{5mm}Our first aim was to identify exact solutions for equations of motion of unsteady complex~\mbox{fluids}, to be used as natural  
benchmarks for approximate solutions obtained by numerical experiments. To this end, we started considering a simplified setting and, in the present paper, we study the unsteady motion of certain \emph{`smart'} (non-Newtonian) incompressible fluids~in~infinite~straight~pipes. 

A \textit{`smart fluid'} is a fluid whose rheological properties --such as viscosity or flow behavior-- can be rapidly 
altered by external stimuli like electric or magnetic fields, concentrations~of~\mbox{chemical} molecules, pH, or  temperature, making them attractive for an application~in~fields~like~aerospace, automotive, \hspace{-0.1mm}heavy \hspace{-0.1mm}machinery, \hspace{-0.1mm}electronic, \hspace{-0.1mm}and \hspace{-0.1mm}biomedical \hspace{-0.1mm}industries \hspace{-0.1mm}(\textit{cf}.\ \hspace{-0.1mm}\cite[\hspace{-0.1mm}Chap.~\hspace{-0.1mm}6]{smart_fluids},~\hspace{-0.1mm}for~\hspace{-0.1mm}an~\hspace{-0.1mm}overview).

The unsteady motion in straight pipes of infinite length, when the velocity is directed along the axis and depends only on the variables in the orthogonal directions, leads~to~a~class~of~\mbox{fully-develop}\-ed solutions, like the classical Hagen--Poiseuille solutions (\textit{cf}.~\cite{Hagen1839,Poiseuille1846}) (in the case of a circular~cross-section) for the steady Navier--Stokes equations. The same time-dependent~problem~can~be~exactly integrated in the case of a given time-periodic pressure drop by means of special~(Bessel)~functions, as in the work of Womersley in 1955 (\textit{cf}.~\cite{Womersley1995}). The time-dependent case, in the~presence~of~a~given pressure drop or flow rate is also at the basis of one of~the~so-called~Leray's~problems.

To fix the problem, let $\Omega\hspace{-0.15em}\coloneqq\hspace{-0.15em} \mathbb{R}\times \Sigma$ be a $d$-dimensional (with $d\hspace{-0.15em}\in\hspace{-0.15em}\{2,3\}$) cylindrical pipe~of~infinite\linebreak length occupied by a simplified~\emph{`smart~fluid'}. We
choose the coordinate system~in~such~a~way~that the cross-section $\Sigma$ lies in the $\{0\}\times \mathbb{R}^{d-1}$-plane (\textit{i.e.},~for~the~sake~of~\mbox{simplicity},~we~write~${\Sigma\subseteq \mathbb{R}^{d-1}}$). The
generic $L$-time-periodic motion of the 
fluid, denoting by $I\coloneqq(0,L)$, $L\in (0,+\infty)$,~the~time interval, is then characterized~by~a~velocity field $\mathbf{v}\colon I\times \Omega\to \mathbb{R}^d$ and a kinematic pressure field  $\pi\colon I\times \Omega\to \mathbb{R}$ jointly satisfying the following system of equations:\footnote{Note that $\partial\Omega=\mathbb R\times\partial\Sigma$. \textcolor{black}{In particular, $\Sigma$ denotes a cross-section of the pipe and is not a part of $\partial\Omega$. Thus, \eqref{eq:periodic_pNSE.c} prescribes the flow rate through each cross-section, whereas \eqref{eq:periodic_pNSE.d} is the no-slip condition on the lateral~wall.}}\enlargethispage{2.5mm}
\begin{subequations}\label{eq:periodic_pNSE}
\begin{alignat}{2}
    \partial_t \mathbf{v}- \textup{div}\,\mathbf{S}(\cdot,\mathbf{D} \mathbf{v})+\textup{div}(\mathbf{v}\otimes \mathbf{v})+\nabla \pi&=\mathbf{0}_d&&\quad \text{ in }I\times \Omega\,,\label{eq:periodic_pNSE.a}\\
    \textup{div}\,\mathbf{v}&=0&&\quad\text{ in } I\times\Omega\,,\label{eq:periodic_pNSE.b}\\
    (\mathbf{v},\mathbf{n}_{\Sigma})_{\Sigma}&=\alpha&&\quad\text{ in } I\,,\label{eq:periodic_pNSE.c}\\
    \mathbf{v}&=\mathbf{0}_d &&\quad\text{ on } I\times\partial\Omega\,,\label{eq:periodic_pNSE.d}\\
    \pi(0)=\pi(L)\,,\;\mathbf{v}(0)&=\mathbf{v}(L)&&\quad\text{ in }\Omega\,,\label{eq:periodic_pNSE.e}
\end{alignat}
\end{subequations}
where \hspace{-0.1mm}$\mathbf{n}_{\Sigma}\colon\hspace{-0.175em}\Sigma\hspace{-0.175em}\to\hspace{-0.175em} \mathbb{S}^{d-1}$ \hspace{-0.1mm}is \hspace{-0.1mm}a \hspace{-0.1mm}unit-length  \hspace{-0.1mm}vector \hspace{-0.1mm}field \hspace{-0.1mm}orthogonal  \hspace{-0.1mm}to \hspace{-0.1mm}$\Sigma$, ${\mathbf{D}\mathbf{v}\hspace{-0.175em}\coloneqq\hspace{-0.175em} \frac{1}{2}(\nabla \mathbf{v}\hspace{-0.175em}+\hspace{-0.175em}\nabla\mathbf{v}^\top)\colon \hspace{-0.175em}I\hspace{-0.175em}\times\hspace{-0.175em}\Omega \hspace{-0.175em}\to\hspace{-0.175em}\mathbb{R}^{d\times d}_{\textup{sym}}}$ the \textit{strain-rate tensor},  and $\alpha\colon I\to \mathbb{R}$ a prescribed $L$-time-periodic \textit{flow rate}. Moreover, the \textit{stress tensor} $\mathbf{S}(\cdot,\mathbf{D} \mathbf{v})\colon I\times \Omega\to \mathbb{R}^{d\times d}_{\textup{sym}}$, in the Navier--Stokes case, is the product of the \textit{kinematic viscosity} $\nu_0>0$ and the strain-rate~tensor; however, more general stress tensors~will~be~studied~here.
\begin{remark}\label{rem:intro}
     Problem \eqref{eq:periodic_pNSE} is often called the \emph{`\textcolor{black}{assigned flux}'} problem, as opposed to the \emph{`\textcolor{black}{assigned pressure}'} problem, where the pressure gradient $\nabla \pi\colon I\times \Omega\to \mathbb{R}^d$ is given and the problem is a standard  parabolic one for the single unknown velocity field $\mathbf{v}\colon I\times \Omega\to \mathbb{R}^d$~(\textit{cf}.~Section~\ref{sec:pressure}).
\end{remark}
According to the definition of a \textit{fully-developed flow}, the velocity profile has~to~be~invariant under translations along the axis\footnote{Throughout the paper, for $i=1,\ldots,d$, $d\in \{2,3\}$, by $\mathbf{e}_i\in \mathbb{S}^{d-1}$, we denote the $i$-th unit vector.} $\mathbb{R}\textsf{a}$ $||$ $\mathbb{R}\textsf{e}_1$ of the pipe $\Omega$ and directed along
it,~while~the~\mbox{pressure} gradient  is parallel to the axis $\mathbb{R}\textsf{a}$ and may depend only on time.  As a consequence, there exist $L$-time-periodic~functions $v\colon I\times \Sigma\to \mathbb{R}$ and $\Gamma\colon I\to \mathbb{R}$  such that for every $(t,x)=(t,x_1,\overline{x})\in I\times \Omega$, where 
\begin{align}\label{def:x_bar}
    \overline{x}\coloneqq \begin{cases}
        (x_2,x_3)&\text{ if }d=3\\
        x_2&\text{ if }d=2
    \end{cases}\Bigg\}\in \Sigma\,,
\end{align}
we have that 
\begin{equation}\label{eq:restrictions}
    \mathbf{v}(t,x)\coloneqq v(t,\overline{x})\mathbf{e}_1\qquad\text{and}\qquad 
    \pi(t,x)\coloneqq \Gamma(t)x_1\,.
\end{equation}  
The configuration described above is illustrated in Figure \ref{fig:enter-label.1} for a specific case of the cross-section~$\Sigma$. However, we emphasize that throughout this paper, we do not impose any assumptions on the regularity or shape of the cross-section $\Sigma$, except that it is a bounded polygonal domain if  $d=3$ and an interval if $d=2$.
%
%
%
\begin{figure}[H]
    \centering

 
\tikzset{
pattern size/.store in=\mcSize, 
pattern size = 5pt,
pattern thickness/.store in=\mcThickness, 
pattern thickness = 0.3pt,
pattern radius/.store in=\mcRadius, 
pattern radius = 1pt}
\makeatletter
\pgfutil@ifundefined{pgf@pattern@name@_xgx43540n}{
\pgfdeclarepatternformonly[\mcThickness,\mcSize]{_xgx43540n}
{\pgfqpoint{0pt}{0pt}}
{\pgfpoint{\mcSize+\mcThickness}{\mcSize+\mcThickness}}
{\pgfpoint{\mcSize}{\mcSize}}
{
\pgfsetcolor{\tikz@pattern@color}
\pgfsetlinewidth{\mcThickness}
\pgfpathmoveto{\pgfqpoint{0pt}{0pt}}
\pgfpathlineto{\pgfpoint{\mcSize+\mcThickness}{\mcSize+\mcThickness}}
\pgfusepath{stroke}
}}
\makeatother

 
\tikzset{
pattern size/.store in=\mcSize, 
pattern size = 5pt,
pattern thickness/.store in=\mcThickness, 
pattern thickness = 0.3pt,
pattern radius/.store in=\mcRadius, 
pattern radius = 1pt}
\makeatletter
\pgfutil@ifundefined{pgf@pattern@name@_6ptzhly6o}{
\pgfdeclarepatternformonly[\mcThickness,\mcSize]{_6ptzhly6o}
{\pgfqpoint{0pt}{0pt}}
{\pgfpoint{\mcSize+\mcThickness}{\mcSize+\mcThickness}}
{\pgfpoint{\mcSize}{\mcSize}}
{
\pgfsetcolor{\tikz@pattern@color}
\pgfsetlinewidth{\mcThickness}
\pgfpathmoveto{\pgfqpoint{0pt}{0pt}}
\pgfpathlineto{\pgfpoint{\mcSize+\mcThickness}{\mcSize+\mcThickness}}
\pgfusepath{stroke}
}}
\makeatother

  
\tikzset {_vqe0ka0ee/.code = {\pgfsetadditionalshadetransform{ \pgftransformshift{\pgfpoint{89.1 bp } { -108.9 bp }  }  \pgftransformscale{1.32 }  }}}
\pgfdeclareradialshading{_1rdukeo9l}{\pgfpoint{-72bp}{88bp}}{rgb(0bp)=(1,1,1);
rgb(0bp)=(1,1,1);
rgb(25bp)=(0,0,0);
rgb(400bp)=(0,0,0)}

  
\tikzset {_q2wx3gikc/.code = {\pgfsetadditionalshadetransform{ \pgftransformshift{\pgfpoint{89.1 bp } { -108.9 bp }  }  \pgftransformscale{1.32 }  }}}
\pgfdeclareradialshading{_j679kea5w}{\pgfpoint{-72bp}{88bp}}{rgb(0bp)=(1,1,1);
rgb(0bp)=(1,1,1);
rgb(25bp)=(0,0,0);
rgb(400bp)=(0,0,0)}
\tikzset{every picture/.style={line width=0.75pt}} 

\begin{tikzpicture}[x=1.025pt,y=1.025pt,yscale=-1,xscale=1]

\draw [color={Red}  ,draw opacity=1 ]   (30.73,56) -- (58,56.18) ;
\draw [shift={(61,56.2)}, rotate = 180.37] [fill={Red}  ,fill opacity=1 ][line width=0.08]  [draw opacity=0] (5.36,-2.57) -- (0,0) -- (5.36,2.57) -- (3.56,0) -- cycle    ;
\draw [color={green}  ,draw opacity=1 ]   (30.73,56) -- (30.61,30.4) ;
\draw [shift={(30.6,27.4)}, rotate = 89.75] [fill={green}  ,fill opacity=1 ][line width=0.08]  [draw opacity=0] (5.36,-2.57) -- (0,0) -- (5.36,2.57) -- (3.56,0) -- cycle    ;
\draw [color={Goldenrod}  ,draw opacity=1 ]   (30.73,56) -- (17.52,69.27) ;
\draw [shift={(15.4,71.4)}, rotate = 314.87] [fill={Goldenrod}  ,fill opacity=1 ][line width=0.08]  [draw opacity=0] (5.36,-2.57) -- (0,0) -- (5.36,2.57) -- (3.56,0) -- cycle    ;
\draw  [fill={rgb, 255:red, 0; green, 0; blue, 0 }  ,fill opacity=1 ] (31.52,56) .. controls (31.52,55.56) and (31.17,55.2) .. (30.73,55.2) .. controls (30.29,55.2) and (29.93,55.56) .. (29.93,56) .. controls (29.93,56.44) and (30.29,56.8) .. (30.73,56.8) .. controls (31.17,56.8) and (31.52,56.44) .. (31.52,56) -- cycle ;
\draw  [fill={denim}  ,fill opacity=0.1 ] (370.01,85.74) -- (77.04,86.15) .. controls (72.07,86.16) and (68.02,72.73) .. (68,56.16) .. controls (67.98,39.59) and (71.99,26.16) .. (76.96,26.15) -- (369.93,25.74) .. controls (374.9,25.73) and (378.94,39.16) .. (378.97,55.73) .. controls (378.99,72.3) and (374.98,85.73) .. (370.01,85.74) .. controls (365.04,85.75) and (360.99,72.32) .. (360.97,55.75) .. controls (360.94,39.19) and (364.96,25.75) .. (369.93,25.74) ;
\draw  [draw opacity=0][fill={byzantium}  ,fill opacity=0.7 ] (111.78,55.87) .. controls (111.78,72.37) and (107.76,85.75) .. (102.79,85.75) .. controls (97.82,85.75) and (93.78,72.38) .. (93.78,55.88) .. controls (93.77,39.38) and (97.8,26) .. (102.77,26) .. controls (107.74,26) and (111.77,39.37) .. (111.78,55.87) -- cycle (102.79,85.75) -- (102.79,26) .. controls (97.82,26) and (93.79,39.38) .. (93.8,55.88) .. controls (93.8,72.35) and (97.83,85.71) .. (102.79,85.75) -- cycle  ;
\draw  [draw opacity=0][fill={byzantium}  ,fill opacity=0.6 ] (151.53,55.87) .. controls (151.53,72.37) and (147.51,85.75) .. (142.54,85.75) .. controls (137.57,85.75) and (133.53,72.38) .. (133.53,55.88) .. controls (133.52,39.38) and (137.55,26) .. (142.52,26) .. controls (147.49,26) and (151.52,39.37) .. (151.53,55.87) -- cycle (142.54,85.75) -- (142.54,26) .. controls (137.57,26) and (133.54,39.38) .. (133.55,55.88) .. controls (133.55,72.35) and (137.58,85.71) .. (142.54,85.75) -- cycle  ;
\draw  [draw opacity=0][fill={byzantium}  ,fill opacity=0.5 ] (191.78,55.87) .. controls (191.78,72.37) and (187.76,85.75) .. (182.79,85.75) .. controls (177.82,85.75) and (173.78,72.38) .. (173.78,55.88) .. controls (173.77,39.38) and (177.8,26) .. (182.77,26) .. controls (187.74,26) and (191.77,39.37) .. (191.78,55.87) -- cycle (182.79,85.75) -- (182.79,26) .. controls (177.82,26) and (173.79,39.38) .. (173.8,55.88) .. controls (173.8,72.35) and (177.83,85.71) .. (182.79,85.75) -- cycle  ;
\draw  [draw opacity=0][fill={byzantium}  ,fill opacity=0.4 ] (231.53,55.87) .. controls (231.53,72.37) and (227.51,85.75) .. (222.54,85.75) .. controls (217.57,85.75) and (213.53,72.38) .. (213.53,55.88) .. controls (213.52,39.38) and (217.55,26) .. (222.52,26) .. controls (227.49,26) and (231.52,39.37) .. (231.53,55.87) -- cycle (222.54,85.75) -- (222.54,26) .. controls (217.57,26) and (213.54,39.38) .. (213.55,55.88) .. controls (213.55,72.35) and (217.58,85.71) .. (222.54,85.75) -- cycle  ;
\draw  [draw opacity=0][fill={byzantium}  ,fill opacity=0.3 ] (271.53,55.87) .. controls (271.53,72.37) and (267.51,85.75) .. (262.54,85.75) .. controls (257.57,85.75) and (253.53,72.38) .. (253.53,55.88) .. controls (253.52,39.38) and (257.55,26) .. (262.52,26) .. controls (267.49,26) and (271.52,39.37) .. (271.53,55.87) -- cycle (262.54,85.75) -- (262.54,26) .. controls (257.57,26) and (253.54,39.38) .. (253.55,55.88) .. controls (253.55,72.35) and (257.58,85.71) .. (262.54,85.75) -- cycle  ;
\draw  [draw opacity=0][fill={byzantium}  ,fill opacity=0.2 ] (311.78,55.62) .. controls (311.78,72.12) and (307.76,85.5) .. (302.79,85.5) .. controls (297.82,85.5) and (293.78,72.13) .. (293.78,55.63) .. controls (293.77,39.13) and (297.8,25.75) .. (302.77,25.75) .. controls (307.74,25.75) and (311.77,39.12) .. (311.78,55.62) -- cycle (302.79,85.5) -- (302.79,25.75) .. controls (297.82,25.75) and (293.79,39.13) .. (293.8,55.63) .. controls (293.8,72.1) and (297.83,85.46) .. (302.79,85.5) -- cycle  ;
\draw  [draw opacity=0][fill={byzantium}  ,fill opacity=0.1 ] (351.78,56.12) .. controls (351.78,72.62) and (347.76,86) .. (342.79,86) .. controls (337.82,86) and (333.78,72.63) .. (333.78,56.13) .. controls (333.77,39.63) and (337.8,26) .. (342.77,26) .. controls (347.74,26) and (351.77,39.62) .. (351.78,56.12) -- cycle (342.79,86) -- (342.79,26) .. controls (337.82,26) and (333.79,39.63) .. (333.8,56.13) .. controls (333.8,72.6) and (337.83,85.96) .. (342.79,86) -- cycle  ;
\draw [color={denim}  ,draw opacity=1 ][line width=1.5]    (79.25,55.72) -- (361.07,55.88) ;
\draw [color={denim}  ,draw opacity=0.8 ][line width=0.75]    (84.75,49.05) -- (361,48.75) ;
\draw [color={denim}  ,draw opacity=0.6 ][line width=0.75]    (85,42.8) -- (361.25,42.5) ;
\draw [color={denim}  ,draw opacity=0.25 ][line width=0.75]    (85.25,81.8) -- (365.5,81.5) ;
\draw [color={denim}  ,draw opacity=0.4 ][line width=0.75]    (85,75.55) -- (362.75,75.25) ;
\draw [color={denim}  ,draw opacity=0.6 ][line width=0.75]    (85.25,69.05) -- (361.5,68.75) ;
\draw [color={denim}  ,draw opacity=0.8 ][line width=0.75]    (85.25,62.8) -- (361.5,62.5) ;
\draw [color={denim}  ,draw opacity=0.2 ][line width=0.75]    (84.64,30.02) -- (365.88,29.75) ;
\draw [color={denim}  ,draw opacity=0.4 ][line width=0.75]    (84.75,36.3) -- (362.5,36) ;
\draw  [draw opacity=0][pattern=_xgx43540n,pattern size=6pt,pattern thickness=0.75pt,pattern radius=0pt, pattern color={gray}][dash pattern={on 4.5pt off 4.5pt}] (231.53,55.97) .. controls (231.53,72.41) and (227.51,85.75) .. (222.54,85.75) .. controls (222.53,85.75) and (222.52,85.75) .. (222.52,85.75) -- (222.52,26.2) .. controls (217.56,26.2) and (213.54,39.57) .. (213.53,56.09) .. controls (213.53,56.05) and (213.53,56.01) .. (213.53,55.97) .. controls (213.52,39.53) and (217.55,26.2) .. (222.52,26.2) .. controls (227.49,26.19) and (231.52,39.52) .. (231.53,55.97) -- cycle ;
\draw  [draw opacity=0][dash pattern={on 0.75pt off 0.75pt on 0.75pt off 0.75pt}] (342.68,26) .. controls (342.68,26) and (342.68,26) .. (342.68,26) .. controls (347.58,26) and (351.56,39.47) .. (351.56,55.78) .. controls (351.56,71.81) and (347.73,84.86) .. (342.94,85.31) -- (342.68,55.78) -- cycle ; \draw  [color={byzantium}  ,draw opacity=0.2 ][dash pattern={on 0.75pt off 0.75pt on 0.75pt off 0.75pt}] (342.68,26) .. controls (342.68,26) and (342.68,26) .. (342.68,26) .. controls (347.58,26) and (351.56,39.47) .. (351.56,55.78) .. controls (351.56,71.81) and (347.73,84.86) .. (342.94,85.31) ;  
\draw  [draw opacity=0][dash pattern={on 0.75pt off 0.75pt on 0.75pt off 0.75pt}] (302.53,26.44) .. controls (302.53,26.44) and (302.53,26.44) .. (302.53,26.44) .. controls (307.43,26.44) and (311.41,39.66) .. (311.41,55.97) .. controls (311.41,72) and (307.58,85.05) .. (302.79,85.5) -- (302.53,55.98) -- cycle ; \draw  [color={byzantium}  ,draw opacity=0.3 ][dash pattern={on 0.75pt off 0.75pt on 0.75pt off 0.75pt}] (302.53,26.44) .. controls (302.53,26.44) and (302.53,26.44) .. (302.53,26.44) .. controls (307.43,26.44) and (311.41,39.66) .. (311.41,55.97) .. controls (311.41,72) and (307.58,85.05) .. (302.79,85.5) ;  
\draw  [draw opacity=0][dash pattern={on 0.75pt off 0.75pt on 0.75pt off 0.75pt}] (262.54,26) .. controls (262.54,26) and (262.54,26) .. (262.54,26) .. controls (262.54,26) and (262.54,26) .. (262.54,26) .. controls (267.44,26) and (271.42,39.22) .. (271.42,55.53) .. controls (271.43,71.56) and (267.59,84.61) .. (262.8,85.06) -- (262.54,55.53) -- cycle ; \draw  [color={byzantium}  ,draw opacity=0.4 ][dash pattern={on 0.75pt off 0.75pt on 0.75pt off 0.75pt}] (262.54,26) .. controls (262.54,26) and (262.54,26) .. (262.54,26) .. controls (262.54,26) and (262.54,26) .. (262.54,26) .. controls (267.44,26) and (271.42,39.22) .. (271.42,55.53) .. controls (271.43,71.56) and (267.59,84.61) .. (262.8,85.06) ;  
\draw  [draw opacity=0][dash pattern={on 0.75pt off 0.75pt on 0.75pt off 0.75pt}] (222.52,26.2) .. controls (222.52,26.2) and (222.52,26.2) .. (222.52,26.2) .. controls (222.52,26.2) and (222.52,26.2) .. (222.52,26.2) .. controls (227.42,26.2) and (231.4,39.42) .. (231.4,55.73) .. controls (231.41,71.76) and (227.57,84.81) .. (222.78,85.25) -- (222.52,55.73) -- cycle ; \draw  [color={byzantium}  ,draw opacity=0.5 ][dash pattern={on 0.75pt off 0.75pt on 0.75pt off 0.75pt}] (222.52,26.2) .. controls (222.52,26.2) and (222.52,26.2) .. (222.52,26.2) .. controls (222.52,26.2) and (222.52,26.2) .. (222.52,26.2) .. controls (227.42,26.2) and (231.4,39.42) .. (231.4,55.73) .. controls (231.41,71.76) and (227.57,84.81) .. (222.78,85.25) ;  
\draw  [draw opacity=0][dash pattern={on 0.75pt off 0.75pt on 0.75pt off 0.75pt}] (102.79,26) .. controls (102.79,26) and (102.79,26) .. (102.79,26) .. controls (102.79,26) and (102.79,26) .. (102.79,26) .. controls (107.69,26) and (111.67,39.22) .. (111.67,55.53) .. controls (111.68,71.56) and (107.84,84.61) .. (103.05,85.06) -- (102.79,55.53) -- cycle ; \draw  [color={byzantium}  ,draw opacity=0.8 ][dash pattern={on 0.75pt off 0.75pt on 0.75pt off 0.75pt}] (102.79,26) .. controls (102.79,26) and (102.79,26) .. (102.79,26) .. controls (102.79,26) and (102.79,26) .. (102.79,26) .. controls (107.69,26) and (111.67,39.22) .. (111.67,55.53) .. controls (111.68,71.56) and (107.84,84.61) .. (103.05,85.06) ;  
\draw  [draw opacity=0][dash pattern={on 0.75pt off 0.75pt on 0.75pt off 0.75pt}] (142.54,26) .. controls (142.54,26) and (142.54,26) .. (142.54,26) .. controls (142.54,26) and (142.54,26) .. (142.54,26) .. controls (147.44,26) and (151.42,39.22) .. (151.42,55.53) .. controls (151.43,71.56) and (147.59,84.61) .. (142.8,85.06) -- (142.54,55.53) -- cycle ; \draw  [color={byzantium}  ,draw opacity=0.7 ][dash pattern={on 0.75pt off 0.75pt on 0.75pt off 0.75pt}] (142.54,26) .. controls (142.54,26) and (142.54,26) .. (142.54,26) .. controls (142.54,26) and (142.54,26) .. (142.54,26) .. controls (147.44,26) and (151.42,39.22) .. (151.42,55.53) .. controls (151.43,71.56) and (147.59,84.61) .. (142.8,85.06) ;  
\draw  [draw opacity=0][dash pattern={on 0.75pt off 0.75pt on 0.75pt off 0.75pt}] (182.77,26) .. controls (182.77,26) and (182.77,26) .. (182.77,26) .. controls (182.77,26) and (182.77,26) .. (182.77,26) .. controls (187.67,26) and (191.65,39.22) .. (191.65,55.53) .. controls (191.66,71.56) and (187.82,84.61) .. (183.03,85.06) -- (182.77,55.53) -- cycle ; \draw  [color={byzantium}  ,draw opacity=0.6 ][dash pattern={on 0.75pt off 0.75pt on 0.75pt off 0.75pt}] (182.77,26) .. controls (182.77,26) and (182.77,26) .. (182.77,26) .. controls (182.77,26) and (182.77,26) .. (182.77,26) .. controls (187.67,26) and (191.65,39.22) .. (191.65,55.53) .. controls (191.66,71.56) and (187.82,84.61) .. (183.03,85.06) ;  
\draw  [draw opacity=0][fill={byzantium}  ,fill opacity=0.4 ] (213.53,56.2) .. controls (213.52,39.63) and (217.55,26.2) .. (222.52,26.2) -- (222.52,86.2) .. controls (217.56,86.16) and (213.53,72.74) .. (213.53,56.2) -- cycle (222.52,26.2) .. controls (222.52,26.2) and (222.52,26.2) .. (222.52,26.2) -- (222.52,26.2) -- cycle ;
\draw  [draw opacity=0][pattern=_6ptzhly6o,pattern size=6pt,pattern thickness=0.75pt,pattern radius=0pt, pattern color={gray}][dash pattern={on 4.5pt off 4.5pt}] (231.53,56.42) .. controls (231.53,58.96) and (231.43,61.43) .. (231.25,63.79) .. controls (231.42,61.5) and (231.51,59.11) .. (231.51,56.64) .. controls (231.5,40.1) and (227.48,26.68) .. (222.52,26.64) -- (222.52,86.2) .. controls (217.57,86.16) and (213.55,72.91) .. (213.53,56.54) .. controls (213.54,40.02) and (217.56,26.65) .. (222.52,26.64) -- (222.52,26.64) .. controls (222.52,26.64) and (222.52,26.64) .. (222.52,26.64) -- (222.52,26.64) .. controls (227.49,26.64) and (231.52,39.97) .. (231.53,56.42) -- cycle ;
\draw  [draw opacity=0][fill={byzantium}  ,fill opacity=0.1 ] (333.95,55.31) .. controls (333.94,38.74) and (337.97,25.31) .. (342.94,25.31) -- (342.94,85.31) .. controls (337.98,85.27) and (333.96,71.85) .. (333.95,55.31) -- cycle (342.94,25.31) .. controls (342.94,25.31) and (342.94,25.31) .. (342.94,25.31) -- (342.94,25.31) -- cycle ;
\draw  [draw opacity=0][fill={byzantium}  ,fill opacity=0.2 ] (293.8,55.5) .. controls (293.79,38.93) and (297.82,25.5) .. (302.79,25.5) -- (302.79,85.5) .. controls (297.83,85.46) and (293.8,72.05) .. (293.8,55.5) -- cycle (302.79,25.5) .. controls (302.79,25.5) and (302.79,25.5) .. (302.79,25.5) -- (302.79,25.5) -- cycle ;
\draw  [draw opacity=0][fill={byzantium}  ,fill opacity=0.3 ] (253.55,56) .. controls (253.54,39.43) and (257.57,26) .. (262.54,26) -- (262.54,86) .. controls (257.58,85.96) and (253.55,72.55) .. (253.55,56) -- cycle (262.54,26) .. controls (262.54,26) and (262.54,26) .. (262.54,26) -- (262.54,26) -- cycle ;
\draw  [draw opacity=0][fill={byzantium}  ,fill opacity=0.5 ] (173.78,56) .. controls (173.77,39.43) and (177.8,26) .. (182.77,26) -- (182.77,86) .. controls (177.8,85.96) and (173.78,72.55) .. (173.78,56) -- cycle (182.77,26) .. controls (182.77,26) and (182.77,26) .. (182.77,26) -- (182.77,26) -- cycle ;
\draw  [draw opacity=0][fill={byzantium}  ,fill opacity=0.6 ] (133.55,55.75) .. controls (133.54,39.18) and (137.57,25.75) .. (142.54,25.75) -- (142.54,85.75) .. controls (137.58,85.71) and (133.55,72.3) .. (133.55,55.75) -- cycle (142.54,25.75) .. controls (142.54,25.75) and (142.54,25.75) .. (142.54,25.75) -- (142.54,25.75) -- cycle ;
\draw  [draw opacity=0][fill={byzantium}  ,fill opacity=0.7 ] (93.78,56) .. controls (93.77,39.43) and (97.8,26) .. (102.77,26) -- (102.77,86) .. controls (97.8,85.96) and (93.78,72.55) .. (93.78,56) -- cycle (102.77,26) .. controls (102.77,26) and (102.77,26) .. (102.77,26) -- (102.77,26) -- cycle ;
\draw  [draw opacity=0] (342.68,85.78) .. controls (342.68,85.78) and (342.68,85.78) .. (342.68,85.78) .. controls (337.71,85.78) and (333.68,72.35) .. (333.68,55.78) .. controls (333.68,39.5) and (337.57,26) .. (342.43,25.8) -- (342.68,55.78) -- cycle ; \draw  [color={byzantium}  ,draw opacity=0.2 ] (342.68,85.78) .. controls (342.68,85.78) and (342.68,85.78) .. (342.68,85.78) .. controls (337.71,85.78) and (333.68,72.35) .. (333.68,55.78) .. controls (333.68,39.5) and (337.57,26) .. (342.43,25.8) ;  
\draw  [draw opacity=0] (303.02,85.74) .. controls (303.02,85.74) and (303.02,85.74) .. (303.02,85.74) .. controls (298.05,85.74) and (294.02,72.31) .. (294.02,55.74) .. controls (294.02,39.46) and (297.92,26.2) .. (302.77,25.75) -- (303.02,55.74) -- cycle ; \draw  [color={byzantium}  ,draw opacity=0.3 ] (303.02,85.74) .. controls (303.02,85.74) and (303.02,85.74) .. (303.02,85.74) .. controls (298.05,85.74) and (294.02,72.31) .. (294.02,55.74) .. controls (294.02,39.46) and (297.92,26.2) .. (302.77,25.75) ;  
\draw  [draw opacity=0] (262.54,86) .. controls (262.54,86) and (262.54,86) .. (262.54,86) .. controls (257.57,86) and (253.54,72.57) .. (253.54,56) .. controls (253.54,39.72) and (257.43,26.46) .. (262.28,26.01) -- (262.54,56) -- cycle ; \draw  [color={byzantium}  ,draw opacity=0.4 ] (262.54,86) .. controls (262.54,86) and (262.54,86) .. (262.54,86) .. controls (257.57,86) and (253.54,72.57) .. (253.54,56) .. controls (253.54,39.72) and (257.43,26.46) .. (262.28,26.01) ;  
\draw  [draw opacity=0] (222.77,86.18) .. controls (222.77,86.18) and (222.77,86.18) .. (222.77,86.18) .. controls (217.8,86.18) and (213.77,72.75) .. (213.77,56.18) .. controls (213.77,39.9) and (217.67,26.65) .. (222.52,26.2) -- (222.77,56.18) -- cycle ; \draw  [color={byzantium}  ,draw opacity=0.5 ] (222.77,86.18) .. controls (222.77,86.18) and (222.77,86.18) .. (222.77,86.18) .. controls (217.8,86.18) and (213.77,72.75) .. (213.77,56.18) .. controls (213.77,39.9) and (217.67,26.65) .. (222.52,26.2) ;  
\draw  [draw opacity=0] (183.05,85.99) .. controls (183.05,85.99) and (183.05,85.99) .. (183.05,85.99) .. controls (178.07,85.99) and (174.05,72.56) .. (174.05,55.99) .. controls (174.05,39.71) and (177.94,26.45) .. (182.79,26) -- (183.05,55.99) -- cycle ; \draw  [color={byzantium}  ,draw opacity=0.6 ] (183.05,85.99) .. controls (183.05,85.99) and (183.05,85.99) .. (183.05,85.99) .. controls (178.07,85.99) and (174.05,72.56) .. (174.05,55.99) .. controls (174.05,39.71) and (177.94,26.45) .. (182.79,26) ;  
\draw  [draw opacity=0] (142.96,85.74) .. controls (142.96,85.74) and (142.96,85.74) .. (142.96,85.74) .. controls (137.9,85.74) and (133.8,72.36) .. (133.8,55.87) .. controls (133.8,39.65) and (137.76,26.45) .. (142.71,26.01) -- (142.96,55.87) -- cycle ; \draw  [color={byzantium}  ,draw opacity=0.7 ] (142.96,85.74) .. controls (142.96,85.74) and (142.96,85.74) .. (142.96,85.74) .. controls (137.9,85.74) and (133.8,72.36) .. (133.8,55.87) .. controls (133.8,39.65) and (137.76,26.45) .. (142.71,26.01) ;  
\draw  [draw opacity=0] (102.77,86) .. controls (102.77,86) and (102.77,86) .. (102.77,86) .. controls (97.8,86) and (93.77,72.57) .. (93.77,56) .. controls (93.77,39.72) and (97.66,26.46) .. (102.51,26.01) -- (102.77,56) -- cycle ; \draw  [color={byzantium}  ,draw opacity=0.8 ] (102.77,86) .. controls (102.77,86) and (102.77,86) .. (102.77,86) .. controls (97.8,86) and (93.77,72.57) .. (93.77,56) .. controls (93.77,39.72) and (97.66,26.46) .. (102.51,26.01) ;  
\path  [shading=_1rdukeo9l,_vqe0ka0ee] (87.7,26.12) .. controls (86.12,26.32) and (84.89,28.49) .. (84.91,31.1) .. controls (84.94,33.69) and (86.18,35.78) .. (87.74,35.96) -- (87.74,36.16) .. controls (89.5,36.54) and (90.86,38.69) .. (90.87,41.29) .. controls (90.88,43.89) and (89.54,46.05) .. (87.78,46.44) -- (87.78,46.5) .. controls (86.18,46.7) and (84.94,48.87) .. (84.96,51.49) .. controls (84.99,54.08) and (86.24,56.17) .. (87.81,56.35) -- (87.81,56.5) .. controls (89.54,56.58) and (90.94,58.72) .. (90.95,61.35) .. controls (90.96,63.99) and (89.58,66.14) .. (87.85,66.23) -- (87.85,66.23) .. controls (87.84,66.23) and (87.83,66.23) .. (87.82,66.23) .. controls (86.27,66.27) and (85.04,68.5) .. (85.07,71.23) .. controls (85.09,73.94) and (86.35,76.11) .. (87.89,76.1) .. controls (89.59,76.09) and (90.98,78.33) .. (90.99,81.09) .. controls (91,83.48) and (89.98,85.48) .. (88.6,85.98) -- (76.72,86.04) .. controls (72.07,84.88) and (68.38,71.94) .. (68.36,56.14) .. controls (68.34,39.97) and (72.16,26.79) .. (76.96,26.15) -- (86.69,26.11) -- (87.7,26.11) -- (87.7,26.12) -- cycle ; 
 \draw   (87.7,26.12) .. controls (86.12,26.32) and (84.89,28.49) .. (84.91,31.1) .. controls (84.94,33.69) and (86.18,35.78) .. (87.74,35.96) -- (87.74,36.16) .. controls (89.5,36.54) and (90.86,38.69) .. (90.87,41.29) .. controls (90.88,43.89) and (89.54,46.05) .. (87.78,46.44) -- (87.78,46.5) .. controls (86.18,46.7) and (84.94,48.87) .. (84.96,51.49) .. controls (84.99,54.08) and (86.24,56.17) .. (87.81,56.35) -- (87.81,56.5) .. controls (89.54,56.58) and (90.94,58.72) .. (90.95,61.35) .. controls (90.96,63.99) and (89.58,66.14) .. (87.85,66.23) -- (87.85,66.23) .. controls (87.84,66.23) and (87.83,66.23) .. (87.82,66.23) .. controls (86.27,66.27) and (85.04,68.5) .. (85.07,71.23) .. controls (85.09,73.94) and (86.35,76.11) .. (87.89,76.1) .. controls (89.59,76.09) and (90.98,78.33) .. (90.99,81.09) .. controls (91,83.48) and (89.98,85.48) .. (88.6,85.98) -- (76.72,86.04) .. controls (72.07,84.88) and (68.38,71.94) .. (68.36,56.14) .. controls (68.34,39.97) and (72.16,26.79) .. (76.96,26.15) -- (86.69,26.11) -- (87.7,26.11) -- (87.7,26.12) -- cycle ; 

\path  [shading=_j679kea5w,_q2wx3gikc] (369.93,25.74) .. controls (374.9,25.73) and (378.94,39.16) .. (378.97,55.73) .. controls (378.99,71.48) and (375.37,84.4) .. (370.74,85.64) -- (349.12,85.64) -- (349.12,85.63) .. controls (350.7,85.44) and (351.94,83.27) .. (351.92,80.66) .. controls (351.91,78.08) and (350.67,75.98) .. (349.12,75.79) -- (349.12,75.59) .. controls (347.35,75.21) and (346,73.05) .. (346,70.45) .. controls (346,67.85) and (347.35,65.7) .. (349.12,65.31) -- (349.12,65.25) .. controls (350.71,65.06) and (351.96,62.9) .. (351.95,60.27) .. controls (351.93,57.69) and (350.69,55.58) .. (349.12,55.41) -- (349.12,55.25) .. controls (347.39,55.17) and (346,53.02) .. (346,50.39) .. controls (346,47.75) and (347.39,45.6) .. (349.12,45.52) -- (349.12,45.52) .. controls (349.13,45.52) and (349.14,45.52) .. (349.15,45.52) .. controls (350.7,45.49) and (351.94,43.26) .. (351.92,40.54) .. controls (351.91,37.83) and (350.65,35.65) .. (349.12,35.65) .. controls (347.41,35.65) and (346.03,33.41) .. (346.03,30.65) .. controls (346.03,28.26) and (347.07,26.27) .. (348.45,25.77) -- (369.93,25.74) -- cycle ; 
 \draw   (369.93,25.74) .. controls (374.9,25.73) and (378.94,39.16) .. (378.97,55.73) .. controls (378.99,71.48) and (375.37,84.4) .. (370.74,85.64) -- (349.12,85.64) -- (349.12,85.63) .. controls (350.7,85.44) and (351.94,83.27) .. (351.92,80.66) .. controls (351.91,78.08) and (350.67,75.98) .. (349.12,75.79) -- (349.12,75.59) .. controls (347.35,75.21) and (346,73.05) .. (346,70.45) .. controls (346,67.85) and (347.35,65.7) .. (349.12,65.31) -- (349.12,65.25) .. controls (350.71,65.06) and (351.96,62.9) .. (351.95,60.27) .. controls (351.93,57.69) and (350.69,55.58) .. (349.12,55.41) -- (349.12,55.25) .. controls (347.39,55.17) and (346,53.02) .. (346,50.39) .. controls (346,47.75) and (347.39,45.6) .. (349.12,45.52) -- (349.12,45.52) .. controls (349.13,45.52) and (349.14,45.52) .. (349.15,45.52) .. controls (350.7,45.49) and (351.94,43.26) .. (351.92,40.54) .. controls (351.91,37.83) and (350.65,35.65) .. (349.12,35.65) .. controls (347.41,35.65) and (346.03,33.41) .. (346.03,30.65) .. controls (346.03,28.26) and (347.07,26.27) .. (348.45,25.77) -- (369.93,25.74) -- cycle ; 

\draw  [fill={white}  ,fill opacity=1 ] (370.06,25.88) .. controls (375.03,25.88) and (379.07,39.31) .. (379.07,55.88) .. controls (379.08,72.44) and (375.05,85.88) .. (370.08,85.88) .. controls (365.11,85.88) and (361.08,72.45) .. (361.07,55.88) .. controls (361.07,39.31) and (365.09,25.88) .. (370.06,25.88) -- cycle ;
\draw  [fill={denim}  ,fill opacity=0.15 ] (370.06,25.88) .. controls (375.03,25.88) and (379.07,39.31) .. (379.07,55.88) .. controls (379.08,72.44) and (375.05,85.88) .. (370.08,85.88) .. controls (365.11,85.88) and (361.08,72.45) .. (361.07,55.88) .. controls (361.07,39.31) and (365.09,25.88) .. (370.06,25.88) -- cycle ;
\draw [color={denim}  ,draw opacity=1 ][line width=1.5]    (360.82,55.76) -- (391,55.62) ;
\draw [shift={(395,55.6)}, rotate = 179.73] [fill={denim}  ,fill opacity=1 ][line width=0.08]  [draw opacity=0] (6.43,-3.09) -- (0,0) -- (6.43,3.09) -- (4.27,0) -- cycle    ;
\draw [color={denim}  ,draw opacity=0.2 ][line width=0.75]    (366.5,29.6) -- (390.5,29.6) ;
\draw [shift={(393.5,29.6)}, rotate = 180] [fill={denim}  ,fill opacity=0.2 ][line width=0.08]  [draw opacity=0] (5.36,-2.57) -- (0,0) -- (5.36,2.57) -- (3.56,0) -- cycle    ;
\draw [color={denim}  ,draw opacity=0.4 ][line width=0.75]    (363,35.85) -- (390.75,35.85) ;
\draw [shift={(393.75,35.85)}, rotate = 180] [fill={denim}  ,fill opacity=0.4 ][line width=0.08]  [draw opacity=0] (5.36,-2.57) -- (0,0) -- (5.36,2.57) -- (3.56,0) -- cycle    ;
\draw [color={denim}  ,draw opacity=0.6 ][line width=0.75]    (361.5,42.35) -- (390.75,42.35) ;
\draw [shift={(393.75,42.35)}, rotate = 180] [fill={denim}  ,fill opacity=0.6 ][line width=0.08]  [draw opacity=0] (5.36,-2.57) -- (0,0) -- (5.36,2.57) -- (3.56,0) -- cycle    ;
\draw [color={denim}  ,draw opacity=0.8 ][line width=0.75]    (361.5,48.6) -- (390.75,48.6) ;
\draw [shift={(393.75,48.6)}, rotate = 180] [fill={denim}  ,fill opacity=0.8 ][line width=0.08]  [draw opacity=0] (5.36,-2.57) -- (0,0) -- (5.36,2.57) -- (3.56,0) -- cycle    ;
\draw [color={denim}  ,draw opacity=0.8 ][line width=0.75]    (361.75,62.35) -- (391,62.35) ;
\draw [shift={(394,62.35)}, rotate = 180] [fill={denim}  ,fill opacity=0.8 ][line width=0.08]  [draw opacity=0] (5.36,-2.57) -- (0,0) -- (5.36,2.57) -- (3.56,0) -- cycle    ;
\draw [color={denim}  ,draw opacity=0.6 ][line width=0.75]    (361.75,68.6) -- (391,68.6) ;
\draw [shift={(394,68.6)}, rotate = 180] [fill={denim}  ,fill opacity=0.6 ][line width=0.08]  [draw opacity=0] (5.36,-2.57) -- (0,0) -- (5.36,2.57) -- (3.56,0) -- cycle    ;
\draw [color={denim}  ,draw opacity=0.4 ][line width=0.75]    (363.25,75.1) -- (391,75.1) ;
\draw [shift={(394,75.1)}, rotate = 180] [fill={denim}  ,fill opacity=0.4 ][line width=0.08]  [draw opacity=0] (5.36,-2.57) -- (0,0) -- (5.36,2.57) -- (3.56,0) -- cycle    ;
\draw [color={denim}  ,draw opacity=0.25 ][line width=0.75]    (366,81.35) -- (391.25,81.35) ;
\draw [shift={(395,81.35)}, rotate = 180] [fill={denim}  ,fill opacity=0.25 ][line width=0.08]  [draw opacity=0] (5.36,-2.57) -- (0,0) -- (5.36,2.57) -- (3.56,0) -- cycle    ;
\draw  [draw opacity=0] (370.07,85.88) .. controls (370.07,85.88) and (370.07,85.88) .. (370.07,85.88) .. controls (365.1,85.88) and (361.07,72.45) .. (361.07,55.88) .. controls (361.07,39.6) and (364.96,26.34) .. (369.82,25.89) -- (370.07,55.88) -- cycle ; \draw   (370.07,85.88) .. controls (370.07,85.88) and (370.07,85.88) .. (370.07,85.88) .. controls (365.1,85.88) and (361.07,72.45) .. (361.07,55.88) .. controls (361.07,39.6) and (364.96,26.34) .. (369.82,25.89) ;  
\draw [color={gray}  ,draw opacity=1 ]   (222.52,26) -- (222.5,19.67) ;
\draw [shift={(222.5,19.67)}, rotate = 89.84] [color={gray}  ,draw opacity=1 ][line width=0.75]    (0,5.59) -- (0,-5.59)   ;
\draw [color={byzantium}  ,draw opacity=1 ]   (102.77,26) -- (102.75,19.67) ;
\draw [shift={(102.75,19.67)}, rotate = 89.84] [color={byzantium}  ,draw opacity=1 ][line width=0.75]    (0,5.59) -- (0,-5.59)   ;
\draw [color={denim}  ,draw opacity=1 ]   (342.68,55.78) -- (342.75,19.75) ;
\draw [shift={(342.75,19.75)}, rotate = 90.11] [color={denim}  ,draw opacity=1 ][line width=0.75]    (0,5.59) -- (0,-5.59)   ;

\draw [color={shamrockgreen}  ,draw opacity=0.5 ] [dash pattern={on 0.75pt off 0.75pt on 0.75pt off 0.75pt}]  (93.8,55.88) -- (102.77,56) ;
\draw [color={shamrockgreen!50}  ,draw opacity=1 ] [dash pattern={on 0.75pt off 0.75pt on 0.75pt off 0.75pt}]  (79.23,55.72) -- (93.78,55.88) ;
\draw [color={shamrockgreen!50}  ,draw opacity=1 ] [dash pattern={on 0.75pt off 0.75pt on 0.75pt off 0.75pt}]  (102.77,56) -- (133.53,55.88) ;
\draw [color={shamrockgreen}  ,draw opacity=0.5 ] [dash pattern={on 0.75pt off 0.75pt on 0.75pt off 0.75pt}]  (133.55,55.88) -- (142.52,56) ;
\draw [color={shamrockgreen!50}  ,draw opacity=1 ] [dash pattern={on 0.75pt off 0.75pt on 0.75pt off 0.75pt}]  (143.04,56) -- (173.8,55.88) ;
\draw [color={shamrockgreen}  ,draw opacity=0.5 ] [dash pattern={on 0.75pt off 0.75pt on 0.75pt off 0.75pt}]  (173.8,55.88) -- (182.77,56) ;
\draw [color={shamrockgreen!50}  ,draw opacity=1 ] [dash pattern={on 0.75pt off 0.75pt on 0.75pt off 0.75pt}]  (183.05,55.99) -- (213.81,55.87) ;
\draw [color={shamrockgreen}  ,draw opacity=0.5 ] [dash pattern={on 0.75pt off 0.75pt on 0.75pt off 0.75pt}]  (213.81,55.87) -- (222.77,55.99) ;
\draw [color={shamrockgreen!50}  ,draw opacity=1 ] [dash pattern={on 0.75pt off 0.75pt on 0.75pt off 0.75pt}]  (222.79,56) -- (253.55,55.88) ;
\draw [color={shamrockgreen}  ,draw opacity=0.5 ] [dash pattern={on 0.75pt off 0.75pt on 0.75pt off 0.75pt}]  (253.55,55.88) -- (262.52,56) ;
\draw [color={shamrockgreen!50}  ,draw opacity=1 ] [dash pattern={on 0.75pt off 0.75pt on 0.75pt off 0.75pt}]  (262.54,56) -- (293.3,55.88) ;
\draw [color={shamrockgreen}  ,draw opacity=0.5 ] [dash pattern={on 0.75pt off 0.75pt on 0.75pt off 0.75pt}]  (293.3,55.88) -- (302.27,56) ;
\draw [color={shamrockgreen!50}  ,draw opacity=1 ] [dash pattern={on 0.75pt off 0.75pt on 0.75pt off 0.75pt}]  (302.53,55.98) -- (333.29,55.85) ;
\draw [color={shamrockgreen}  ,draw opacity=0.5 ] [dash pattern={on 0.75pt off 0.75pt on 0.75pt off 0.75pt}]  (333.29,55.85) -- (342.26,55.98) ;
\draw [color={shamrockgreen!50}  ,draw opacity=1 ] [dash pattern={on 0.75pt off 0.75pt on 0.75pt off 0.75pt}]  (342.26,55.98) -- (347.38,56) ;
\draw [color={shamrockgreen!50}  ,draw opacity=1 ] [dash pattern={on 0.75pt off 0.75pt on 0.75pt off 0.75pt}]  (362.53,55.74) -- (413.63,55.51) ;
\draw [shift={(416.63,55.5)}, rotate = 179.74] [fill={shamrockgreen!50}  ,fill opacity=1 ][line width=0.08]  [draw opacity=0] (5.36,-2.57) -- (0,0) -- (5.36,2.57) -- (3.56,0) -- cycle    ;
\path  [shading=_j679kea5w,_q2wx3gikc] (87.34,26.14) .. controls (85.75,26.35) and (84.52,28.51) .. (84.55,31.13) .. controls (84.58,33.71) and (85.82,35.8) .. (87.37,35.98) -- (87.37,36.19) .. controls (89.14,36.56) and (90.5,38.71) .. (90.51,41.31) .. controls (90.52,43.92) and (89.18,46.08) .. (87.41,46.47) -- (87.41,46.53) .. controls (85.82,46.72) and (84.58,48.89) .. (84.6,51.51) .. controls (84.63,54.1) and (85.88,56.2) .. (87.45,56.37) -- (87.45,56.53) .. controls (89.18,56.6) and (90.58,58.74) .. (90.59,61.38) .. controls (90.6,64.02) and (89.22,66.17) .. (87.49,66.25) -- (87.49,66.26) .. controls (87.48,66.26) and (87.47,66.26) .. (87.46,66.26) .. controls (85.91,66.29) and (84.68,68.53) .. (84.7,71.25) .. controls (84.73,73.96) and (85.99,76.13) .. (87.53,76.12) .. controls (89.23,76.12) and (90.62,78.35) .. (90.63,81.11) .. controls (90.64,83.5) and (89.62,85.5) .. (88.24,86) -- (76.35,86.07) .. controls (71.7,84.9) and (68.02,71.96) .. (68,56.16) .. controls (67.98,40) and (71.79,26.81) .. (76.6,26.17) -- (86.33,26.14) -- (87.34,26.14) -- (87.34,26.14) -- cycle ; 
 \draw   (87.34,26.14) .. controls (85.75,26.35) and (84.52,28.51) .. (84.55,31.13) .. controls (84.58,33.71) and (85.82,35.8) .. (87.37,35.98) -- (87.37,36.19) .. controls (89.14,36.56) and (90.5,38.71) .. (90.51,41.31) .. controls (90.52,43.92) and (89.18,46.08) .. (87.41,46.47) -- (87.41,46.53) .. controls (85.82,46.72) and (84.58,48.89) .. (84.6,51.51) .. controls (84.63,54.1) and (85.88,56.2) .. (87.45,56.37) -- (87.45,56.53) .. controls (89.18,56.6) and (90.58,58.74) .. (90.59,61.38) .. controls (90.6,64.02) and (89.22,66.17) .. (87.49,66.25) -- (87.49,66.26) .. controls (87.48,66.26) and (87.47,66.26) .. (87.46,66.26) .. controls (85.91,66.29) and (84.68,68.53) .. (84.7,71.25) .. controls (84.73,73.96) and (85.99,76.13) .. (87.53,76.12) .. controls (89.23,76.12) and (90.62,78.35) .. (90.63,81.11) .. controls (90.64,83.5) and (89.62,85.5) .. (88.24,86) -- (76.35,86.07) .. controls (71.7,84.9) and (68.02,71.96) .. (68,56.16) .. controls (67.98,40) and (71.79,26.81) .. (76.6,26.17) -- (86.33,26.14) -- (87.34,26.14) -- (87.34,26.14) -- cycle ; 

\draw (308.5,7.5) node [anchor=north west][inner sep=0.75pt]  [color={denim}  ,opacity=1 ]  {$\mathbf{v}( t,x) =v( t,\overline{x})\mathbf{e}_{1}$};
\draw (68.5,7.5) node [anchor=north west][inner sep=0.75pt]  [,color={byzantium}  ,opacity=1 ]  {$\pi ( t,x) =\Gamma ( t) x_{1}$};
\draw (56.8,58.4) node [anchor=north west][inner sep=0.75pt]  [font=\small,color={rgb, 255:red, 0; green, 0; blue, 0 }  ,opacity=1 ]  {$x_{1}$};
\draw (15,72) node [anchor=north west][inner sep=0.75pt]  [font=\small,color={rgb, 255:red, 0; green, 0; blue, 0 }  ,opacity=1 ]  {$x_{2}$};
\draw (18.5,24) node [anchor=north west][inner sep=0.75pt]  [font=\small,color={rgb, 255:red, 0; green, 0; blue, 0 }  ,opacity=1 ]  {$x_{3}$};
\draw (216,5) node [font=\LARGE,anchor=north west][inner sep=0.75pt]  [color={gray}  ,opacity=1 ]  {$\Sigma $};
\draw (405,42.5) node [anchor=north west][inner sep=0.75pt]  [color={shamrockgreen}  ,opacity=1 ]  {$\mathbb{R}\mathsf{a}$};

\end{tikzpicture}

    \caption{Schematic diagram of an infinite pipe $\Omega\coloneqq \mathbb{R}\times \Sigma$ with cross-section $\Sigma\subseteq \mathbb{R}^{d-1}$: in \textcolor{denim}{blue}, the velocity field $\mathbf{v}\colon I\times \Omega\to \mathbb{R}^d$ of the form~\eqref{eq:restrictions}, which depends only on the $\overline{x}$-variable and points only in the $\mathbb{R}\mathbf{e}_1$-direction; 
    in \textcolor{byzantium}{purple}, the pressure field $\pi\colon I\times \Omega\to \mathbb{R}$~of~the~form~\eqref{eq:restrictions}, the gradient of which is parallel to the axis $\mathbb{R}\mathsf{a}$ (\textcolor{shamrockgreen}{green}) and which  depends only on the $x_1$-variable. 
    }
    \label{fig:enter-label.1}
\end{figure} 

\textcolor{black}{The interest in this class of problems has been renewed in recent years, in particular due to applications in hemodynamics (\textit{cf}.\ \cite{Qua2002,VV2005,VV2007}). Moreover, explicit fully-developed solutions are useful as benchmark solutions for debugging and testing complex three-dimensional Computational Fluid Dynamics (CFD) codes, even though such idealized flows may fail to be stable or directly observable in turbulent regimes.} 
The problems arising for the Navier--Stokes equations in the unsteady time-periodic case are reviewed and addressed in Beir\~ao~da~Veiga~\cite{Bei2005c}, while the modern role in applied and computational problems is highlighted in Galdi~\cite{Gal2008}, Quarteroni~\cite{Qua2002}, and Veneziani and Vergara~\cite{VV2005,VV2007}, with emphasis on the role of boundary conditions at the exit of a finite pipe. In recent years, there have been several improvements~around~these~\mbox{results}: Considering instead of a time-periodic, an almost time-periodic motion~\cite{BR2012}; Motion in deformable pipes as in Formaggia, Veneziani, and Vergara~\cite{FVV2010}; Motion coupled with electro-magnetic~effects~in~\cite{BMMS2013}, extensions to non-Newtonian fluids in~Galdi~and~Grisanti~\cite{GaldiGrisanti2016}.\enlargethispage{1.5mm}

Our main objective is to extend the results from~\cite{GaldiGrisanti2016} to a broader class of non-Newtonian~fluids by studying the case in which the stress tensor involves a position-dependent power-law index
$p\colon \Sigma\to (1,+\infty)$; see the precise assumptions in Section~\ref{sec:extra-stress-tensor}. A prototypical example~of~a~stress tensor  we will consider
(within the family with so-called \emph{$(p(\cdot)$,$\delta)$-structure}, for $\delta\geq0$)~is
\begin{equation}\label{intro:stress}
    \mathbf{S}(\cdot,\mathbf{Dv})\coloneqq (\delta+|\mathbf{Dv}|)^{p(\cdot)-2}\mathbf{Dv}\quad\text{ a.e.\ in }\Omega\,.
\end{equation} 
This model naturally arises in the description of \emph{`smart fluids'}, such as electro-rheological~(\textit{cf}.~\cite{Ruz2000}), magneto-rheological (\textit{cf}.\ \cite{bia_2005}), thermo-rheological  (\textit{cf}.~\cite{AR06}), and chemically-reacting (\textit{cf}.\ \cite{HMPR10}) fluids.
 The non-linearity \eqref{intro:stress} 
also occurs, \textit{e.g.}, 
in  homogenization~\cite{Z87},~quasi-Newtonian~fluids~\cite{Z97},~the thermistor \hspace{-0.2mm}problem~\cite{termistor}, \hspace{-0.2mm}fluid \hspace{-0.2mm}flow \hspace{-0.2mm}in 
\hspace{-0.2mm}porous \hspace{-0.2mm}media \hspace{-0.2mm}\cite{porous_media}, \hspace{-0.2mm}magnetostatics \hspace{-0.2mm}\cite{magnetostatics},~\hspace{-0.2mm}and~\hspace{-0.2mm}\mbox{image}~\hspace{-0.2mm}\mbox{processing}~\hspace{-0.2mm}\cite{image_processing}.

The study of  $p(\cdot)$-fluids, particularly their mathematical properties and numerical analysis, is an active research field. Due to space limitations, we do not thoroughly review the relevant literature here, but emphasize that the need to benchmark recent numerical results from~\cite{BK2024, BK2025}~(or~\cite{BerselliBreitDiening2016})\linebreak motivated our analysis of the exact solutions in Section~\ref{sec:Exact-Solutions}.  
Given the applied nature of this paper, we present an alternative proof of the existence of weak solutions for the evolution problem using a fully-discrete finite-difference/-element discretization, consistent with~the~numerical~experiments. In the constant exponent case (\textit{i.e.},  $p = \textrm{const}$), this yields an alternative proof of the~results~from \cite{GaldiGrisanti2016}, and, in the Navier--Stokes case (\textit{i.e.}, $ p = 2$), it offers new insights into the results from~\cite{Bei2005c, BR2012}.  
Our approach handles all values of $p(\overline{x})\in (1,+\infty)$, $\overline{x}\in \Sigma$, without requiring a Newtonian~term, unlike~\cite{GaldiGrisanti2016}.  
In the proof of the existence of discrete solutions, we use a fully-constructive fixed point argument. 
The focus of the paper is on the `\textcolor{black}{\emph{assigned flux}}' problem~\eqref{eq:periodic_pNSE} with a prescribed time-periodic flow rate, but the derived techniques extend to the `\textcolor{black}{\emph{assigned pressure}}' problem with a prescribed time-periodic~pressure~gradient,~as~in~the~original~Womersley~formulation~(\textit{cf}.~\cite{Womersley1995}).\enlargethispage{8.5mm}\smallskip

\textit{Plan of the paper.} In Section~\ref{sec:Preliminaries}, we recall fundamental aspects of the functional analytic framework tailored to unsteady problems involving position-dependent stress tensors and introduce the fully-discrete finite-difference/-element discretization employed in both the fully-constructive existence analysis and numerical experiments. Section~\ref{sec:fully} is dedicated to the derivation~and~simplification of the governing equations for fully-developed, time-periodic~flows~in~cylindrical~geometries. Based on these reductions, we present the effective equations and discuss the structural properties of the stress tensor in this setting.
In Section \ref{sec:existence}, we formulate the evolution problem with a prescribed time-periodic flow rate in both variational and flux-free forms. Then, we establish the existence of discrete (numerical) solutions using a constructive fixed-point argument, 
along with their (strong) stability and (weak) convergence to solutions of the associated~continuous~problem.
Section~\ref{sec:pressure} addresses the complementary case of a  prescribed time-periodic pressure gradient, highlighting how the analysis adapts to this alternative formulation. 
In Section~\ref{sec:Exact-Solutions}, we identify 
explicit solutions inspired by analogies with two-dimensional fluid mechanics problems. Finally,\linebreak in Section \ref{sec:experiments}, we present a series of numerical experiments that illustrate 
the~theoretical~findings.\vspace{-2mm}
 
    \section{Preliminaries}\label{sec:Preliminaries}\vspace{-2mm}

\hspace{5mm}Throughout the entire paper, by $\Sigma\subseteq \mathbb{R}^{d-1}$, $d\in \{2,3\}$, we denote a bounded polyhedral Lipschitz domain. All functions considered in this paper are time-periodic with~period~${L\in (0,+\infty)}$.\newpage  For this reason, we  restrict our attention to the time interval $I\coloneqq (0,L)$.~On~this~time~\mbox{interval}, for a given Banach space $(X,\|\cdot\|_X)$ and a given integrability exponent $r\in [1,+\infty]$, we employ standard notation for  Bochner--Lebesgue spaces $L^r(I;X)$ and Bochner--Sobolev~spaces~$W^{1,r}(I;X)$.

The fact that the stress tensor $\mathbf{S}\colon \hspace{-0.1em}\Omega \times \mathbb{R}^{d\times d}\hspace{-0.1em}\to\hspace{-0.1em} \mathbb{R}^{d\times d}$ has a position-dependent~(with~\mbox{respect}~to~$\Omega$) power-law~index makes it natural to employ variable Lebesgue spaces~and~\mbox{variable}~Sobolev~spaces.\vspace{-1.5mm}

\subsection{Variable Lebesgue spaces and variable Sobolev spaces}\vspace{-1mm}

\hspace*{5mm}Let $\omega\subseteq \mathbb{R}^n$, $n\in \mathbb{N}$, be an open set and $L^0(\omega)$ the linear space of scalar (Lebesgue) measurable functions on $\omega$. For~${p\in L^0(\omega)}$,  we define
	$p^+\coloneqq \textup{ess\,sup}_{x\in
			\omega}{\{p(x)\}}$~and~$p^-\coloneqq \textup{ess\,inf}_{x\in
			\omega}{\{p(x)\}}$. Then, by
	$\mathcal{P}^{\infty}(\omega)\coloneqq \{p\in L^0(\omega)\mid
	1\leq p^-\leq p^+<\infty\}$, we denote the \textit{set of bounded variable exponents}. For  ${p\in\mathcal{P}^\infty(\omega)}$ and $f\in L^0(\omega)$,  the \textit{modular (with respect~to~$p$)}~is~defined~by 
	\begin{align*}
		\rho_{p(\cdot),\omega}(f)\coloneqq \int_{\omega}{\vert f(x)\vert^{p(x)}\,\mathrm{d}x}\,.
	\end{align*}
	Then, for given $p\in \mathcal{P}^\infty(\omega)$, 
	the \textit{variable Lebesgue} and \textit{Sobolev spaces}, respectively, are defined by
	\begin{align*}
	L^{p(\cdot)}(\omega)&\coloneqq \bigl\{ f\in L^0(\omega)\mid \rho_{p(\cdot),\omega}(f)<\infty\bigr\}\,,\\
    W^{1,p(\cdot)}(\omega)&\coloneqq \bigl\{ f\in L^{p(\cdot)}(\omega)\mid \nabla f\in (L^{p(\cdot)}(\omega))^n\bigr\}\,,
	\end{align*}
    which are Banach spaces (\textit{cf}.\ \cite[Thm.\ 3.2.13]{DHHR2011}), when equipped with the norms
    \begin{align*}
        \begin{aligned}  
        \| f\|_{p(\cdot),\omega}&\coloneqq \inf\bigl\{\lambda> 0\mid \rho_{p(\cdot),\omega}(\tfrac{f}{\lambda})\leq 1\bigr\}\,,&&\quad \text{ for }f\in L^{p(\cdot)}(\omega)\,,\\[-0.5mm]
        \| f\|_{1,p(\cdot),\omega}&\coloneqq \| f\|_{p(\cdot),\omega}+\|\nabla f\|_{p(\cdot),\omega}\,,&&\quad \text{ for }f\in W^{1,p(\cdot)}(\omega)\,.
        \end{aligned}
    \end{align*}
     The \hspace{-0.15mm}closure \hspace{-0.15mm}of \hspace{-0.15mm}$C^\infty_c(\omega)$ \hspace{-0.15mm}in \hspace{-0.15mm}$W^{1,p(\cdot)}(\omega)$ \hspace{-0.15mm}is \hspace{-0.15mm}denoted \hspace{-0.15mm}by \hspace{-0.15mm}$W^{1,p(\cdot)}_0(\omega)$. \hspace{-0.15mm}If \hspace{-0.15mm}$p(\cdot)\hspace{-0.175em}=\hspace{-0.175em}p\hspace{-0.175em}\in\hspace{-0.175em} [1,+\infty)$,~\hspace{-0.15mm}\mbox{variable}~\hspace{-0.15mm}Lebesgue \hspace{-0.15mm}and \hspace{-0.15mm}Sobolev \hspace{-0.15mm}spaces \hspace{-0.15mm}coincide \hspace{-0.15mm}with \hspace{-0.15mm}customary \hspace{-0.15mm}Lebesgue \hspace{-0.15mm}and \hspace{-0.15mm}Sobolev \hspace{-0.15mm}spaces~\hspace{-0.15mm}and 
    \hspace{-0.15mm}${\|\hspace{-0.175em}\cdot\hspace{-0.175em}\|_{p(\cdot),\omega}\hspace{-0.175em}=\hspace{-0.175em}(\int_{\omega}{\hspace{-0.175em} \vert\hspace{-0.175em} \cdot\hspace{-0.175em}\vert^p\,\hspace{-0.1em} \mathrm{d}x})^{\smash{\frac{1}{p}}}}$. 
     For $\ell\in \{1,d-1,d,d\times d\}$, the  $(L^2(\omega))^\ell$-inner product and -norm are abbreviated~via~$(\cdot,\cdot)_{\omega}$~and~${\|\cdot\|_{\omega}}$.\enlargethispage{10mm}\vspace{-1mm}

\subsection{Stress tensor}\label{sec:extra-stress-tensor}\vspace{-1mm}

\hspace{5mm}By using a classical framework (see, \textit{e.g.}, M{\'a}lek \textit{et al.}~\cite{MNRR1996}), the stress tensor ${\mathbf{S}\colon \hspace{-0.175em}\Omega\hspace{-0.175em}\times\hspace{-0.175em}\mathbb{R}^{d\times d}\hspace{-0.175em}\to \hspace{-0.175em}\mathbb{R}^{d\times d}}$, for a.e.\ $x\in \Omega$~and~every~$\mathbf{A}\in \mathbb{R}^{d\times d}$, is~defined~by
\begin{align}
\label{eq:S}
    \mathbf{S}(x,\mathbf{A})\coloneqq \boldsymbol{\nu}(x,\vert \mathbf{A}\vert^2)\mathbf{A}\,,
\end{align}
where the \textit{generalized viscosity} 
$\boldsymbol{\nu}\colon \Omega\times [0,+\infty)\to [0,+\infty)$ is a (Lebesgue) measurable mapping such that, for a given power-law index $p\in \mathcal{P}^\infty(\Omega)$ with $p^->1$,  the following conditions are met:  
\begin{itemize}[noitemsep,topsep=2pt,leftmargin=!,labelwidth=\widthof{(S.3)},font=\itshape]
    \item[(S.1)]\hypertarget{S.1}{} $\boldsymbol{\nu}\colon \hspace{-0.1em}\Omega\times [0,+\infty)\hspace{-0.1em}\to \hspace{-0.1em}[0,+\infty)$ is a \emph{{Carath\'eodory} mapping}, \textit{i.e.}, $\boldsymbol{\nu}(x,\cdot)\colon \hspace{-0.1em} [0,+\infty)\hspace{-0.1em}\to\hspace{-0.1em} [0,+\infty)$~is continuous for a.e.\ $x\in \Omega$ and  $\boldsymbol{\nu}(\cdot,a)\colon \Omega\to [0,+\infty)$ is (Lebesgue) measurable  for all $a\ge 0$;
    \item[(S.2)]\hypertarget{S.2}{} There exist $\mathcal{K}_1\hspace{-0.1em}>\hspace{-0.1em}0$ and  $\mathcal{K}_2\hspace{-0.1em}\in\hspace{-0.1em} L^1(\Omega) $ such that for a.e.\ $x\hspace{-0.1em}\in\hspace{-0.1em} \Omega$ and every $\mathbf{A}\hspace{-0.1em}\in\hspace{-0.1em} \mathbb{R}^{d\times d}$,~we~have~that
    \begin{align*}
        \mathbf{S}(x,\mathbf{A}):\mathbf{A}\ge \mathcal{K}_1\vert \mathbf{A}\vert^{p(x)}-\mathcal{K}_2(x)\,;
    \end{align*}

    \item[(S.3)]\hypertarget{S.3}{} There exist $\mathcal{K}_3\ge 0$ and $\mathcal{K}_4\in \smash{L^{p'(\cdot)}(\Omega)}$, where $p'\coloneqq \smash{\frac{p}{p-1}}\in\mathcal{P}^{\infty}(\Omega)$ is the \emph{Hölder~conjugate exponent},  with $\mathcal{K}_4\ge 0$ a.e.\ in $\Omega$ such that for a.e.\ $x\in \Omega$ and every $\mathbf{A}\in \mathbb{R}^{d\times d}$,~we~have~that\vspace{-0.5mm}
    \begin{align*}
       \vert \mathbf{S}(x,\mathbf{A})\vert\le \mathcal{K}_3\vert \mathbf{A}\vert^{p(x)-1}+\mathcal{K}_4(x)\,;
    \end{align*}

    \item[(S.4)]\hypertarget{S.4}{} For a.e.\ $x\in \Omega$ and every  $\mathbf{A},\mathbf{B}\in \smash{\mathbb{R}^{d\times d}}$ with $\smash{\mathbf{A}\neq \mathbf{B}}$, we have that\vspace{-0.5mm}
    \begin{align*}
    (\mathbf{S}(x,\mathbf{A})-\mathbf{S}(x,\mathbf{B})):(\mathbf{A}-\mathbf{B})>0\,.
    \end{align*}
\end{itemize} 
\begin{remark} 
    Assumptions (\hyperlink{S.2}{S.2}) and (\hyperlink{S.3}{S.3})
    are 
     standard lower and upper bound assumptions.  
     Since no additional regularity of solutions in the spatial variables is required in our analysis,~we~do not assume strong monotonicity but  (\hyperlink{S.4}{S.4}). Assumption (\hyperlink{S.1}{S.1})  
     ensures the existence of a potential.\linebreak 
     We deliberately refrain from pursuing~maximal~\mbox{generality}, but focus on representative and physically meaningful examples, in line with existing literature, to emphasize the more applied aspects of the problem.
\end{remark}\newpage

    \subsection{Spatial and temporal discretization}\vspace{-0.5mm}

    \hspace*{5mm}In this section, we introduce the discrete spaces and discrete operators needed for our later fully-discrete finite-difference/-element approximation.\vspace{-1mm}
	
	\subsubsection{Spatial discretization}\vspace{-0.5mm}
	
	\hspace*{5mm}Throughout \hspace{-0.15mm}the \hspace{-0.15mm}entire \hspace{-0.15mm}paper, \hspace{-0.15mm}let \hspace{-0.15mm}$\{\mathcal{T}_h\}_{h>0}$ be a \hspace{-0.15mm}family \hspace{-0.15mm}of \hspace{-0.15mm}shape-regular 
    \hspace{-0.15mm}triangulations~\hspace{-0.15mm}of~\hspace{-0.15mm}${\Sigma\!\subseteq\! \mathbb{R}^{d-1}}$,  consisting of triangles (if $d=3$) or intervals (if $d=2$), where 
	$h>0$ denotes the \textit{\mbox{maximal~mesh-size}}, \textit{i.e.}, $h=\max_{T\in \mathcal{T}_h}{\{h_T\coloneqq  \textup{diam}(T)\}}$. 
    Then, for
	$\ell \in \mathbb N\cup\{0\}$, let us denote by $\mathbb{P}^{\ell}(\mathcal{T}_h)$ the family of functions that are polynomials of degree at most $\ell$ on each  $T\in \mathcal{T}_h$.  
	For given $\ell_v\in\mathbb{N}$, let
	\begin{align}\label{def:fe_space} 
            V_h\coloneqq \widehat{V}_h\cap  W^{1,1}_0(\Sigma)\,,\qquad 
			\widehat{V}_h\subseteq \mathbb{P}_c^{\smash{\ell_v}}(\mathcal{T}_h)\coloneqq \mathbb{P}^{\smash{\ell_v}}(\mathcal{T}_h) \cap W^{1,1}(\Sigma)\,, 
	\end{align}
	be a finite element space such that the following assumption 
    is satisfied:\vspace{-0.5mm}
 
	\begin{assumption}[Projection operator $\Pi_h$]\label{ass:PiV}
		We assume that $\mathbb{P}^1(\mathcal{T}_h) \cap W^{1,1}(\Sigma) \hspace{-0.15em}\subseteq\hspace{-0.15em} \widehat{V}_h$ and there exists a linear projection operator $\Pi_h\colon \hspace{-0.15em}  W^{1,1}(\Sigma)\hspace{-0.15em}\to\hspace{-0.15em} \widehat{V}_h$  (\textit{i.e.}, $\Pi_h\phi_h\hspace{-0.15em}=\hspace{-0.15em}\phi_h$ for all ${\phi_h\hspace{-0.15em}\in\hspace{-0.15em}  \widehat{V}_h}$),~which~\mbox{satisfies}  $\Pi_h(W_0^{1,1}(\Sigma))\subseteq V_h$ and is \textup{locally $W^{1,1}$-stable}, \textit{i.e.}, 
        for every $\phi\in W^{1,1}(\Sigma)$ and $T\in \mathcal{T}_h$,~there~holds\vspace{-0.5mm}
			\begin{align*}
				\|\Pi_h\phi\|_{1,T}&\lesssim  
				\|\phi\|_{1,\omega_T} + h_T   \|\nabla \phi\|_{1,\omega_T} \,,
			\end{align*} 
        where $\omega_T\coloneqq \bigcup\{T'\in \mathcal{T}_h\mid T\cap T'\neq \emptyset\}$ denotes the \emph{element patch (surrounding $T$)}. 
	\end{assumption}	

    \begin{remark}
        Assumption \ref{ass:PiV}, \textit{e.g.}, is satisfied
        by the          
        Scott--Zhang interpolation operator~(\textit{cf}.~\cite{ScottZhang1990}).\vspace{-1mm}
    \end{remark}

	\subsubsection{Temporal discretization}\vspace{-0.5mm}
	
	\hspace{5mm}Throughout the entire paper, for a finite number 
    of time steps $M\hspace{-0.15em}\in\hspace{-0.15em}\mathbb{N}$, the~time~step~size~${\tau\hspace{-0.15em}\coloneqq\hspace{-0.15em} \frac{L}{M}}$, time steps $t_m\coloneqq \tau m$, and intervals $I_m\coloneqq  \left(t_{m-1},t_m\right]$, $m=1,\ldots,M$,~we~set ${\mathcal{I}_\tau\coloneqq  \{I_m\}_{m=1,\ldots,M}}$ and $\mathcal{I}_\tau^0 \hspace{-0.15em}\coloneqq\hspace{-0.15em} \mathcal{I}_\tau\cup\{I_0\}$, where $I_0\hspace{-0.15em}\coloneqq \hspace{-0.15em}(t_{-1},t_0]\hspace{-0.15em}\coloneqq\hspace{-0.15em} (-\tau,0]$. Given a~Banach~space $(X,\|\cdot\|_X)$,~we~denote~by 
	\begin{align*}
		\mathbb{P}^0(\mathcal{I}_\tau;X)&\coloneqq \{f\colon (0,L]\to X\mid f(s)
		=f(t)\text{ in }X\text{ for all }t,s\in I_m\,,\;m=1,\ldots,M\}\,,\\
        \mathbb{P}^0(\mathcal{I}_\tau^0;X)&\coloneqq \{f\colon (-\tau,L]\to X\mid f(s)
		=f(t)\text{ in }X\text{ for all }t,s\in I_m\,,\;m=0,\ldots,M\}\,,
	\end{align*}
    \textit{the spaces of $X$-valued temporally piece-wise constant (with respect to $\mathcal{I}_\tau$ and $\mathcal{I}_\tau^0$,~\mbox{respectively})~functions}. 
    For every $f^\tau\in 	\mathbb{P}^0(\mathcal{I}_\tau^0;X)\cup C^0(\overline{I};X)$, 
		 the \textit{backward difference quotient}  $\mathrm{d}_\tau f^\tau\in \mathbb{P}^0(\mathcal{I}_\tau;X)$~is \mbox{defined}~by\vspace{-0.5mm}
		\begin{align*}
		\mathrm{d}_\tau f^\tau|_{I_m}\coloneqq \tfrac{1}{\tau}\bigl(f^\tau(t_m)-f^\tau(t_{m-1})\bigr)\quad\text{ in }X \quad\text{ for all }m=1,\ldots,M\,.
		\end{align*} 
If $X$ is a Hilbert space equipped with inner product $(\cdot,\cdot)_X$, for every $f^\tau,g^\tau\hspace{-0.15em}\in \hspace{-0.15em}\mathbb{P}^0(\mathcal{I}_\tau^0;X)$,~we~have~the following \textit{discrete integration-by-parts formula}: for every $m,n = 0,\ldots,M$~with~$n\ge m$,~there~holds
	\begin{align}
		\int_{t_m}^{t_n}{( \mathrm{d}_\tau f^\tau(t),
			g^\tau(t))_X\,\mathrm{d}t}
		=[(f^\tau(t_i),
			g^\tau(t_i))_X]_{i=m}^{i=n}
        -\int_{t_m}^{t_n}{(\mathrm{d}_\tau  g^\tau(t),
			(\textcolor{black}{\mathrm{T}_{\tau}^-} f^\tau)(t))_X\,\mathrm{d}t} 
        \,,\label{eq:discrete_integration-by-parts}
	\end{align}
    where $\textcolor{black}{\mathrm{T}_{\tau}^-} f^\tau\coloneqq f^\tau (\cdot-\tau)$ a.e.\ in $I$ and  which, in the special case $f^\tau=g^\tau\in \mathbb{P}^0(\mathcal{I}_\tau^0;X)$,
    reduces~to
    \begin{align}
		\int_{t_m}^{t_n}{( \mathrm{d}_\tau f^\tau(t),
			f^\tau(t))_X\,\mathrm{d}t}
		=\tfrac{1}{2}[\|f^\tau(t_i)\|_X^2]_{i=m}^{i=n} 
        +\int_{t_m}^{t_n}{\tfrac{\tau}{2}\|\mathrm{d}_\tau f^\tau(t)\|_X^2\,\mathrm{d}t}
        \,.\label{eq:discrete_integration-by-parts_reduced}
	\end{align}
    The \textit{temporal (local) $L^2$-projection operator} 
    $\Pi^0_{\tau}\colon L^1(I;X)\to \mathbb{P}^0(\mathcal{I}_\tau;X)$, for every $f\in L^1(I;X)$, is defined by\vspace{-1mm}
	\begin{align}\label{def:Pit} \Pi^0_{\tau}f|_{I_m}
    \coloneqq  \tfrac{1}{\tau}(f,1)_{I_m}
    \quad\textup{ in }X\quad \text{ for all }m=1,\ldots,M\,.
	\end{align} 
    The \textit{
    temporal nodal 
    interpolation operator} $\mathrm{I}^0_{\tau}\colon C^0(\overline{I};X)\to \mathbb{P}^0(\mathcal{I}_\tau^0;X)$, for every $f\in C^0(\overline{I};X)$, is defined by\vspace{-1mm}
	\begin{align*}
		\mathrm{I}^0_{\tau}f|_{I_m}\coloneqq  f(t_m)\quad\textup{ in }X\quad \text{ for all }m=0,\ldots,M\,.
	\end{align*}

\section{The fully-developed time-periodic flow}\label{sec:fully}

\hspace{5mm}In this section, we derive the relevant equations of a fully-developed time-periodic flow of a simplified smart fluid specified by the properties (\hyperlink{S.1}{S.1})--(\hyperlink{S.4}{S.4}) 
and provide a variational~formulation.\enlargethispage{6.5mm}

We recall that from the ansatz~\eqref{eq:restrictions} (with \eqref{def:x_bar}) for a fully-developed flow,~it~follows~that\vspace{-0.5mm}
\begin{subequations} 
\begin{itemize}[noitemsep,topsep=2pt,leftmargin=!,labelwidth=\widthof{$\bullet$}]
    \item[$\bullet$] \emph{(Incompressibility).} The flow is incompressible, \textit{i.e.}, we have that\vspace{-0.5mm}
\begin{align}
    \smash{\textup{div}\,\mathbf{v}= \partial_{x_1} v=0\quad \text{ a.e.\ in }I\times \Omega\,;}\label{eq:formula.0.1}\\[-6mm]\notag
\end{align}
\item[$\bullet$] \emph{(Laminarity).} There is no convection and, therefore, the flow is laminar, \textit{i.e.}, we have that\vspace{-0.5mm}
\begin{align}
    \textup{div}(\mathbf{v}\otimes \mathbf{v})
   =
    \left(\begin{array}{c}
         \partial_{x_1}\vert v\vert^2  \\
         \mathbf{0}_{d-1}  
    \end{array}\right)=\mathbf{0}_d
   \quad \text{ a.e.\ in }I\times \Omega\,;\label{eq:formula.0.2}\\[-6mm]\notag
\end{align}
\item[$\bullet$] \emph{($x_1$-independence of strain).} 
The strain-rate tensor depends only on the $\overline{x}$-gradient,~that~is~$\nabla_{\overline{x}}v$ (\textit{i.e.}, with respect to $\overline{x}$, \textit{cf}.\ \eqref{def:x_bar}),  of $v$ and, thus,  the shear-rate $\vert\mathbf{D}\mathbf{v}\vert =\smash{\frac{1}{\sqrt{2}}}\vert \nabla_{\overline{x}} v\vert$~as~well~as the stress tensor $\mathbf{S}(\cdot,\mathbf{Dv})$ and its divergence\vspace{-0.5mm}
\begin{align}\label{eq:formula.0.3}
    \textup{div}\,\mathbf{S}(\cdot,\mathbf{D}\mathbf{v})=
    \left(\begin{array}{cc}
       \textup{div}(\boldsymbol{\nu}(\cdot,\tfrac{1}{2}\vert \nabla _{\overline{x}} v\vert^2)\tfrac{1}{2} \nabla _{\overline{x}} v)\\
        \mathbf{0}_{d-1}
    \end{array}\right)\quad \text{ a.e.\ in }I\times \Omega\,.\\[-6mm]\notag
\end{align}
\end{itemize}
\end{subequations}

As a consequence of \eqref{eq:formula.0.3}, if we additionally assume that the position-dependence of the generalized viscosity~in~\eqref{eq:S} is only through the $\overline{x}$-variable, 
then the viscous term $\textup{div}\,\mathbf{S}(\cdot,\mathbf{D}\mathbf{v})$ is a function only of  the $\overline{x}$-variable and with the last $(d-1)$-components vanishing.

In favour of lighter notation, for each $\overline{x}\in \Sigma$, we denote  $x=\overline{x}$  and omit the subscript~in~the~$\overline{x}$-gradient (\textit{i.e.}, we write $\nabla\coloneqq\nabla_{\overline{x}}$). 

%

In summary, taking into account the reductions \eqref{eq:formula.0.1}--\eqref{eq:formula.0.3}, introducing the \textit{(planar)~stress vector} $\mathbf{s}\colon \Sigma\times \mathbb{R}^{d-1}\to \mathbb{R}^{d-1}$, for a.e.\ $x\in \Sigma$ and every $\mathbf{a}\in \mathbb{R}^{d-1}$ defined by\vspace{-0.5mm}
\begin{align*}
    \smash{\mathbf{s}(x,\mathbf{a})\coloneqq \boldsymbol{\nu}(x,\tfrac{1}{2}\vert \mathbf{a}\vert^2)\tfrac{1}{2} \mathbf{a}\,,}\\[-6mm]\notag
\end{align*}
since, in this setting, $\mathbf{n}_{\Sigma}=\mathbf{e}_1$ on $\Sigma$, 
we arrive at a $(d-1)$-dimensional problem with scalar unknowns
$v \colon I \times \Sigma \to \mathbb R$ and $\Gamma \colon I \to \mathbb R$, \textcolor{black}{where $v$ denotes the \emph{axial velocity profile} and 
$\Gamma$ denotes the \emph{axial pressure-gradient coefficient}, both introduced in \eqref{eq:restrictions},}
such that\vspace{-0.5mm}
\begin{subequations}\label{eq:periodic_pLaplace}
\begin{alignat}{2}
    \partial_tv-\textup{div}\,\mathbf{s}(\cdot,\nabla v)+\Gamma&=0&&\quad \text{ in }I\times \Sigma\,,\label{eq:periodic_pLaplace.1}\\
    (v,1)_{\Sigma}&=\alpha&&\quad \text{ in }I\,,\label{eq:periodic_pLaplace.2}\\
    v&=0 &&\quad\text{ on } I\times\partial\Sigma\,,\label{eq:periodic_pLaplace.3}\\
    \Gamma(0)=\Gamma(L)\,,\;v(0)&=v(L)&&\quad\text{ in }\Sigma\,.\label{eq:periodic_pLaplace.4}\\[-6mm]\notag
\end{alignat}
\end{subequations}

If we introduce the \textit{(planar) generalized viscosity} $\nu\colon \Sigma\times [0,+\infty)\to [0,+\infty)$,  for a.e.\ $x\in \Sigma$ and $a \ge 0$ defined by\vspace{-1mm}
\begin{align}\label{def:planar_viscosity}
    \nu(x,a)\coloneqq \boldsymbol{\nu}(x,\tfrac{1}{2}a)\tfrac{1}{2}\,,\\[-6mm]\notag
\end{align}
then, the assumptions (\hyperlink{S.1}{S.1})--(\hyperlink{S.4}{S.4}) on the stress tensor $\mathbf{S}\colon \Omega\times\mathbb{R}^{d\times d}\to \mathbb{R}^{d\times d}$ translate~to~the~follow\-ing coercivity, boundedness, and monotonicity properties of the stress vector $\mathbf{s}\colon \Sigma \times \mathbb{R}^{d-1}\to \mathbb{R}^{d-1}$:
\begin{itemize}[noitemsep,topsep=2pt,leftmargin=!,labelwidth=\widthof{(S.3)},font=\itshape]
    \item[(s.1)]\hypertarget{s.1}{} $\nu\colon \Sigma\times [0,+\infty)\to [0,+\infty)$ is a {Carath\'eodory} mapping;
    \item[(s.2)]\hypertarget{s.2}{} There exist $\kappa_1>0$, $\kappa_2\in L^1(\Sigma) $ such that for a.e.\ $x\in \Sigma$ and every $\mathbf{a}\in \smash{\mathbb{R}^{d-1}}$, we have that\vspace{-0.5mm}
    \begin{align*}
        \smash{\mathbf{s}(x,\mathbf{a})\cdot\mathbf{a}\ge \kappa_1\vert \mathbf{a}\vert^{p(x)}-\kappa_2(x)\,;}
    \end{align*}

    \item[(s.3)]\hypertarget{s.3}{} There exist $\kappa_3\ge 0$ and $\kappa_4\in \smash{L^{p'(\cdot)}(\Sigma)}$ with $\kappa_4\ge 0$ a.e.\ in $\Sigma$ such that for a.e.\ $x\in \Sigma$ and every $\mathbf{a}\in \mathbb{R}^{d-1}$, we have that\vspace{-1mm}
    \begin{align*}
       \smash{\vert \mathbf{s}(x,\mathbf{a})\vert\le \kappa_3\vert \mathbf{a}\vert^{p(x)-1}+\kappa_4(x)\,;}
    \end{align*}

    \item[(s.4)]\hypertarget{s.4}{} For a.e.\ $x\in \Sigma$ and every  $\mathbf{a},\mathbf{b}\in \smash{\mathbb{R}^{d-1}}$ with $\mathbf{a}\neq \mathbf{b}$, we have that\vspace{-0.5mm}
    \begin{align*}
   \smash{(\mathbf{s}(x,\mathbf{a})-\mathbf{s}(x,\mathbf{b}))\cdot(\mathbf{a}-\mathbf{b})>0\,.}
    \end{align*}
\end{itemize}
\begin{remark}
    Using the notation of the previous section, in (\hyperlink{s.2}{s.2}), we could use $\kappa_1\coloneqq \smash{2^{-\smash{\frac{p^+}{2}}}\mathcal{K}_1}$ and $\kappa_2\coloneqq \mathcal{K}_2$, and, in (\hyperlink{s.3}{s.3}), we could use  $\kappa_3\coloneqq \smash{\frac{1}{\sqrt{2}}}\mathcal{K}_3$ and $\kappa_4\coloneqq \mathcal{K}_4$.
\end{remark}\newpage

In the course of the paper, we make frequent use of the following \emph{(variable) $\varepsilon$-Young inequality}.

\begin{lemma}[(Variable) $\varepsilon$-Young inequality]\label{lem:eps-young}
    For every $\varepsilon\in (0,1)$, there exists a constant $c_\varepsilon>0$ (depending only on $p^+$, $p^-$, $\varepsilon$, and \textcolor{black}{$\kappa_3$}) such that for a.e.\ $x\in \Sigma$ and every $\mathbf{a}, \mathbf{b}\in \mathbb{R}^{d-1}$,~we~have~that\vspace{-4.5mm}
    \begin{subequations}\label{eq:eps-young}
    \begin{align}
        \vert \mathbf{s}(x,\mathbf{a})\cdot\mathbf{b}\vert &\leq c_\varepsilon\,\bigl(\vert \mathbf{a}\vert^{p(x)}+ \kappa_4(x)^{p'(x)}\bigr)+\varepsilon\,\vert\mathbf{b}\vert^{p(x)}\,,\label{eq:eps-young.1}\\
        \vert \mathbf{s}(x,\mathbf{a})\cdot\mathbf{b}\vert &\leq \;\varepsilon\,\bigl(\vert \mathbf{a}\vert^{p(x)}+ \kappa_4(x)^{p'(x)}\bigr)+c_\varepsilon\,\vert\mathbf{b}\vert^{p(x)}\,.\label{eq:eps-young.2}
    \end{align}
    \end{subequations}
\end{lemma}

\begin{proof}
    \emph{ad \eqref{eq:eps-young.1}.} If we apply the (variable) $\varepsilon$-Young inequality in \cite[Prop.\ 2.8]{alex-book}, then, for every $\varepsilon\in (0,1)$,  a.e.\ $x\in \Sigma$, and every $\mathbf{a}, \mathbf{b}\in \mathbb{R}^{d-1}$, we find that
    \begin{align*}
        \vert \mathbf{s}(x,\mathbf{a})\cdot\mathbf{b}\vert&\leq \tfrac{\varepsilon^{-(p^-)'}}{(p^+)'}\vert \mathbf{s}(x,\mathbf{a})\vert^{p'(x)}+ \tfrac{\varepsilon^{p^-}}{p^-}\vert \mathbf{b}\vert^{\textcolor{black}{p(x)}}
        \\&\leq \tfrac{\varepsilon^{-(p^-)'}}{(p^+)'}2^{p^+-1}\bigl(\textcolor{black}{\smash{\kappa_3^{p'(x)}}}\vert \mathbf{a}\vert^{p(x)}+\kappa_4(x)^{p'(x)}\bigr) + \tfrac{\varepsilon^{p^-}}{p^-}\vert \mathbf{b}\vert^{\textcolor{black}{p(x)}}\,.
    \end{align*} 
    Then, a scaling argument yields the claimed estimate \eqref{eq:eps-young.1}. 

    \emph{ad \eqref{eq:eps-young.2}.} If we interchange the roles of $p\in \mathcal{P}^\infty(\Sigma)$ and its Hölder conjugate exponent $p'\in \mathcal{P}^\infty(\Sigma)$, from \cite[Prop.\ 2.8]{alex-book}, for every $\varepsilon\in (0,1)$,  a.e.\ $x\in \Sigma$, and every $\mathbf{a}, \mathbf{b}\in \mathbb{R}^{d-1}$,~it~follows~that  
    \begin{align*}
        \vert \mathbf{s}(x,\mathbf{a})\cdot\mathbf{b}\vert&\leq \tfrac{\varepsilon^{(p^+)'}}{(p^+)'}\vert \mathbf{s}(x,\mathbf{a})\vert^{p'(x)}+ \tfrac{\varepsilon^{-p^+}}{p^-}\vert \mathbf{b}\vert^{\textcolor{black}{p(x)}}
        \\&\leq \tfrac{\varepsilon^{(p^+)'}}{(p^+)'}2^{p^+-1}\bigl(\textcolor{black}{\smash{\kappa_3^{p'(x)}}}\vert \mathbf{a}\vert^{p(x)}+\kappa_4(x)^{p'(x)}\bigr) + \tfrac{\varepsilon^{-p^+}}{p^-}\vert \mathbf{b}\vert^{\textcolor{black}{p(x)}}\,.
    \end{align*}  
    Then, again, a scaling argument yields the claimed estimate \eqref{eq:eps-young.2}. 
\end{proof}

In the following lemma, we derive some elementary, but  crucial, properties related
to the coercivity and growth properties of the stress vector $\mathbf{s}\colon \Sigma\times \mathbb{R}^{d-1}\to \mathbb{R}^{d-1}$ (\textit{cf}.\ (\hyperlink{s.1}{s.1})--(\hyperlink{s.4}{s.4})), which will be useful in the following sections.\enlargethispage{6mm}

\begin{lemma}\label{lem:anti-derivative}
    The anti-derivative $\mathcal{V}\colon \Sigma\times[0,+\infty)\to [0,+\infty)$, for 
    a.e.\ $x\in \Sigma$ and every $a\in[0,+\infty)$ defined by
    \begin{align*}
        \mathcal{V}(x,a)\coloneqq \int_0^a{\nu(x,b)\,\mathrm{d}b}\,,
    \end{align*}
    has the following properties:
    \begin{itemize}[noitemsep,topsep=2pt,leftmargin=!,labelwidth=\widthof{(iii)}]
        \item[(i)]\hypertarget{lem:anti-derivative.i}{} For a.e.\ $x\in \Sigma$, we have that $\mathcal{V}(x,\cdot)\in C^1[0,+\infty)$ with $\frac{\mathrm{d}}{\mathrm{d}a}\mathcal{V}(x,\cdot)=\nu(x,\cdot)$ in $[0,+\infty)$;
        \item[(ii)]\hypertarget{lem:anti-derivative.ii}{} For a.e.\ $x\in \Sigma$ and every $a\in[0,+\infty)$, there holds  
        \begin{align*}
            0\leq \mathcal{V}(x,a)\leq \smash{2\bigl(\tfrac{ \textcolor{black}{\kappa_3}}{p(x)}a^{\smash{\frac{p(x)}{2}}}+\textcolor{black}{\kappa_4(x)}a^{\frac{1}{2}}\bigr)}\,;
        \end{align*} 
        \item[(iii)]\hypertarget{lem:anti-derivative.iii}{} If, in addition, $\kappa_2=0$ in 
        \emph{(\hyperlink{s.2}{s.2})}, then for a.e.\ $x\in \Sigma$ and every $a\in[0,+\infty)$, there holds
        \begin{align*}
           \tfrac{2\kappa_1}{p(x)}a^{\frac{p(x)}{2}} \leq  \mathcal{V}(x,a)\,.
        \end{align*} 
    \end{itemize}
\end{lemma}

\begin{proof}
    \emph{ad (\hyperlink{lem:anti-derivative.i}{i}).} That $\mathcal{V}(x,\cdot)\in C^1[0,+\infty)$ with the stated derivative for a.e.\ $x\in \Sigma$ is an immediate consequence of the {Carath\'eodory} mapping properties of $\nu\colon \Sigma \times[0,+\infty)\to [0,+\infty)$ (\textit{cf}.\ (\hyperlink{s.1}{s.1})).

    \emph{ad (\hyperlink{lem:anti-derivative.ii}{ii})/(\hyperlink{lem:anti-derivative.iii}{iii}).}
    The proofs of claim (\hyperlink{lem:anti-derivative.ii}{ii}) and claim (\hyperlink{lem:anti-derivative.iii}{iii}) follow along the lines of the proof of~\cite[Lem.~4.1]{GaldiGrisanti2016} up to minor adjustments.
\end{proof}
The special form of the stress vector  $\mathbf{s}\colon \Sigma\times \mathbb{R}^{d-1}\to \mathbb{R}^{d-1}$ 
allows the definition of a potential.
\begin{lemma}
\label{lem:convexity}
Let $\mathcal{U}\colon \Sigma\times\mathbb{R}^{d-1}\to [0,+\infty)$ be defined by 
$\mathcal{U}(x,\mathbf{a})\coloneqq \mathcal{V}(x,|\mathbf{a}|^2)$ for a.e.\ $x\in \Sigma$ and all $\mathbf{a}\in \mathbb{R}^{d-1}$. Then, there holds 
$\frac{\mathrm{d}}{\mathrm{d}\mathbf{a}}\mathcal{U}(x,\mathbf{a})=2\mathbf{s}(x,\mathbf{a})$ for a.e.\ $x\in \Sigma$ and all $\mathbf{a}\in \mathbb{R}^{d-1}$ and $\mathcal{U}(x,\cdot)\colon \mathbb{R}^{d-1}\to [0,+\infty)$ is convex for a.e.\ $x\in \Sigma$.  
\end{lemma}
\begin{proof}
    From $\mathcal{V}(x,\cdot)\hspace{-0.175em}\in\hspace{-0.175em} C^1[0,+\infty)$ for a.e.\ $x\hspace{-0.175em}\in\hspace{-0.175em} \Sigma$ (\textit{cf}.\ Lemma \ref{lem:anti-derivative}(\hyperlink{lem:anti-derivative.i}{i})), 
    it follows~that~${\mathcal{U}(x,\cdot)\hspace{-0.175em}\in\hspace{-0.175em} C^1(\mathbb{R}^{d-1})}$ for a.e.\ $x\hspace{-0.175em}\in\hspace{-0.175em} \Sigma$ with $\frac{\mathrm{d}}{\mathrm{d}\mathbf{a}}\mathcal{U}(x,\mathbf{a})\hspace{-0.175em}=\hspace{-0.175em}\nu(x,\vert \mathbf{a}\vert^2)2\mathbf{a}\hspace{-0.175em}=\hspace{-0.175em}2\boldsymbol{\nu}(x,\frac{1}{2}\vert \mathbf{a}\vert^2)\frac{1}{2}\mathbf{a}\hspace{-0.175em}=\hspace{-0.175em}2\mathbf{s}(x,\mathbf{a})$ for a.e.\ $x\hspace{-0.175em}\in\hspace{-0.175em} \Sigma$~and~all~${\mathbf{a}\hspace{-0.175em}\in\hspace{-0.175em} \mathbb{R}^{d-1}}$. As a consequence, since $\mathbf{s}(x,\cdot)\colon\mathbb{R}^{d-1}\to \mathbb{R}^{d-1}$ is monotone for a.e.\ $x\in \Sigma$, using \cite[Lem.\ 4.10]{GGZ1974}, we conclude that $\mathcal{U}(x,\cdot)\colon \mathbb{R}^{d-1}\to [0,+\infty)$ is convex for a.e.\ $x\in \Sigma$.  
\end{proof}
 \bigskip

\newpage
\section{The problem with an assigned time-periodic flow rate}\label{sec:existence}

\hspace{5mm}In this section, we introduce two equivalent variational formulations~of~the~$(d-1)$-dimensional problem~\eqref{eq:periodic_pLaplace}.

\begin{definition}[Variational formulation of~\eqref{eq:periodic_pLaplace}]\label{def:weak_solution}
Given a $L$-time-periodic flow rate $\alpha\in W^{1,q}(I)$, where $q\coloneqq \max\{2,(p^-)'\}$, a pair
    \begin{align*}
        (v,\Gamma)\in 
   (L^\infty(I;W^{1,p(\cdot)}_0(\Sigma))\cap W^{1,2}(I;L^2(\Sigma)))\times L^2(I)\,,
    \end{align*}
    is called \emph{(variational) solution} of~\eqref{eq:periodic_pLaplace} if 
    \begin{align}
        v(0)=v(L)\quad\text{ a.e.\ in }\Sigma\,,\label{def:weak_solution-0}
    \end{align} 
   and  for every $(\phi,\eta)\in 
  (L^1(I;W^{1,p(\cdot)}_0(\Sigma))\cap L^2(I;L^2(\Sigma)))
   \times L^2(I)$, there holds 
    \begin{subequations}
    \label{def:weak_solution-1}
    \begin{align}
        (\partial_t v,\phi)_{I\times 
        \Sigma}+(\mathbf{s}(\cdot,\nabla v), \nabla \phi)_{I\times\Sigma}+(\Gamma,\phi)_{I\times\Sigma}&=0\,,\label{def:weak_solution-1.1}
        \\
        (v,\eta)_{I\times \Sigma}&=(\alpha,\eta)_{I´}\,.\label{def:weak_solution-1.2}
    \end{align}
    \end{subequations}
\end{definition}
\begin{remark}[Periodicity condition \eqref{def:weak_solution-0} and flux condition \eqref{def:weak_solution-1.2}]\label{rem:periodicity}
    \begin{itemize}[noitemsep,topsep=2pt,leftmargin=!,labelwidth=\widthof{(ii)}]
    \item[(i)]\hypertarget{rem:periodicity.i}{} \emph{Periodicity condition \eqref{def:weak_solution-0}:} \hspace{-0.1mm}By \hspace{-0.1mm}the \hspace{-0.1mm}fundamental \hspace{-0.1mm}theorem \hspace{-0.1mm}of \hspace{-0.1mm}calculus \hspace{-0.1mm}for \hspace{-0.1mm}Bochner--Sobolev \hspace{-0.1mm}spaces \hspace{-0.1mm}(\textit{cf}.~\hspace{-0.1mm}\cite[Lem.~\hspace{-0.1mm}2.1.2]{Droniou2001}), we \hspace{-0.1mm}have \hspace{-0.1mm}the \hspace{-0.1mm}embedding \hspace{-0.1mm}$W^{1,2}(I;L^2(\Sigma))\hspace{-0.125em}\hookrightarrow \hspace{-0.125em}C^0(\overline{I};L^2(\Sigma))$, \hspace{-0.1mm}which \hspace{-0.1mm}implies \hspace{-0.1mm}that \hspace{-0.1mm}$v\in W^{1,2}(I;L^2(\Sigma))$, after redefinition on a set of zero (Lebesgue) measure, can be identified with a function in $C^0(\overline{I};L^2(\Sigma))$. This already ensures the well-posedness of the time-periodicity~condition~\eqref{def:weak_solution-0}. However, \hspace{-0.15mm}since \hspace{-0.15mm}also \hspace{-0.15mm}$v\hspace{-0.175em}\in\hspace{-0.175em}  L^\infty(I;W^{1,p(\cdot)}(\Sigma))$, \hspace{-0.15mm}using \hspace{-0.15mm}\cite[Lem.~\hspace{-0.2mm}1.4,~\hspace{-0.15mm}Chap.~\hspace{-0.2mm}III,~\hspace{-0.15mm}§1]{Temam77},~\hspace{-0.15mm}\mbox{after}~\hspace{-0.15mm}\mbox{redefinition} on a set of zero (Lebesgue) measure, $v$ can be identified with~a~\mbox{function}~in~$C_w^0(\overline{I};W^{1,p(\cdot)}_0(\Sigma))$, so that time-periodicity~condition~\eqref{def:weak_solution-0}  actually can~be~\mbox{interpreted}~as~an~\mbox{identity}~in~$W^{1,p(\cdot)}_0(\Sigma)$. Throughout the entire paper, without always stating explicitly, we extend each function satisfying the time-periodicity condition~\eqref{def:weak_solution-0} periodically to the whole real line $\mathbb{R}$,~so~that,~\textit{e.g.},
    \begin{align}\label{rem:periodicity.1}
        v(\cdot + L)\coloneqq v(\cdot)\quad \text{ in }\mathbb{R}\,.
    \end{align}
    Note that the time-periodicity condition~\eqref{rem:periodicity.1}
 allows us to extend~\eqref{def:weak_solution-1} to the   real line $\mathbb{R}$, \textit{i.e.}, \eqref{def:weak_solution-0} is equivalent to that \eqref{rem:periodicity.1} and  for every ${(\phi,\eta)\hspace{-0.15em}\in \hspace{-0.15em}
  C_c^\infty(\mathbb{R};W^{1,p(\cdot)}_0(\Sigma)) 
  \hspace{-0.2em} \times\hspace{-0.2em} C_c^\infty(\mathbb{R})}$,~there~holds\enlargethispage{5mm}\vspace{-4.5mm}
    \begin{subequations}\label{def:weak_solution-2}
    \begin{align}
        (\partial_t v,\phi)_{\mathbb{R}\times 
        \Sigma}+(\mathbf{s}(\cdot,\nabla v), \nabla \phi)_{\mathbb{R}\times\Sigma}+(\Gamma,\phi)_{\mathbb{R}\times\Sigma}&=0\,,\label{def:weak_solution-2.1}\\
        (v,\eta)_{\mathbb{R}\times \Sigma}&=(\alpha,\eta)_{\mathbb{R}´}\,.\label{def:weak_solution-2.2}
    \end{align}
    \end{subequations} 
    By the fundamental theorem in the calculus of variations applied to \eqref{def:weak_solution-2.1},~it~follows~that
    \begin{align}\label{rem:periodicity.2}
        \Gamma(\cdot+L)=\Gamma(\cdot)\quad \text{ a.e.\ in }\mathbb{R}\,,
    \end{align}
    \textit{i.e.}, the time-periodicity of $\Gamma\in L^2(I)$. For this reason, there is no need to explicitly~incorporate the latter into the variational formulation \eqref{def:weak_solution-2}. 
    
     \item[(ii)]\hypertarget{rem:periodicity.ii}{} \emph{Flux condition \eqref{def:weak_solution-2.2}:} By the fundamental theorem in the calculus of variations, also using that $(v,1)_{\Sigma}\in C^0(\overline{I})$ (\textit{cf}.\ (\hyperlink{rem:periodicity.i}{i})), the flux condition \eqref{def:weak_solution-2.2} is equivalent to $(v,1)_{\Sigma}=\alpha$ in $\overline{I}$. 
    \end{itemize} 
\end{remark}

\begin{remark}[Strong formulation of \eqref{eq:periodic_pLaplace}]\label{rem:strong_solution}
The variational formulation \textcolor{black}{from Definition~\ref{def:weak_solution}} is equivalent to the strong formulation of the $(d-1)$-dimensional problem~\eqref{eq:periodic_pLaplace}:~if~$(v,\Gamma)\in (L^\infty(I;W^{1,p(\cdot)}_0(\Sigma))\cap W^{1,2}(I;L^2(\Sigma)))\times L^2(I)$ is a variational solution, due to $\partial_t v+\Gamma\!\in\! L^2(I;L^2(\Sigma))$, from \eqref{def:weak_solution-1.1}, it follows that $\mathbf{s}(\cdot,\nabla v)\in L^2(I;V(\textup{div};\Sigma))$, where $V(\textup{div};\Sigma)\coloneqq \{\mathbf{w}\in (L^{p'(\cdot)}(\Sigma))^{\textcolor{black}{d-1}}\mid \textup{div}\,\mathbf{w}\in L^2(\Sigma)\}$, with $\textup{div}\,\mathbf{s}(\cdot,\nabla v)=\partial_t v+\Gamma$ in $L^2(I;L^2(\Sigma))$.~Therefore, by Remark~\ref{rem:periodicity}(\hyperlink{rem:periodicity.ii}{ii}), from the variational formulation \textcolor{black}{from Definition \ref{def:weak_solution}},~it~follows~that
\begin{subequations}\label{rem:strong_solution-0}
\begin{alignat}{2}
    \partial_t v-\textup{div}\,\mathbf{s}(\cdot,\nabla v)+\Gamma&=0&&\quad\text{ a.e.\ in }I\times \Sigma\,,\label{rem:strong_solution.1}\\
    (v,1)_{\Sigma}&=\alpha&& \quad \text{  in }I\,,\label{rem:strong_solution.2}\\
    v&=0&&\quad \text{ a.e.\ on }I\times \partial\Sigma\,,\label{rem:strong_solution.3}\\
    v(0)&=v(L)&&\quad\text{ a.e.\ in }\Sigma\,.\label{rem:strong_solution.4}
\end{alignat}
\end{subequations}  
\end{remark}\newpage

By shifting the variational formulation (in the sense of Def.\  \ref{def:weak_solution}),
we can incorporate~a flux-free condition in both the trial and the test function space. 
To this end,~we~fix~an~auxiliary function  
\begin{align*}
        \chi\in W^{1,p(\cdot)}_0(\Sigma)\quad \text{ with }\quad(\chi,1)_{\Sigma}=1\,,
\end{align*}
and make the ansatz
\begin{align}\label{def:u}
    u \coloneqq v-\alpha \chi\in L^\infty(I;W^{1,p(\cdot)}_0(\Sigma))\cap W^{1,2}(I;L^2(\Sigma))\,,
\end{align}
leading to the following flux-free formulation.\enlargethispage{7.5mm} 

\begin{definition}[Flux-free formulation of~\eqref{eq:periodic_pLaplace}]\label{def:fluxfree_solution}
   Given a $L$-time-periodic flow rate $\alpha\in W^{1,q}(I)$, where $q\coloneqq \max\{2,(p^-)'\}$, a function
    \begin{align*}
        u\in 
   L^\infty(I;W^{1,p(\cdot)}_0(\Sigma))\cap W^{1,2}(I;L^2(\Sigma))\,,
    \end{align*}
    is called a \emph{flux-free solution} of \eqref{eq:periodic_pLaplace} if 
    \begin{subequations}\label{def:fluxfree_solution.0} 
    \begin{alignat}{2}
        u(0)&=u(L)&&\quad \text{ a.e.\ in }\Sigma\,,\label{def:fluxfree_solution.0.1}\\
        (u,1)_{\Sigma}&=0&&\quad \text{ in }I\,, \label{def:fluxfree_solution.0.2}
    \end{alignat}
    \end{subequations}
    for every $\varphi\in L^1(I;W^{1,p(\cdot)}_0(\Sigma))\cap L^2(I;L^2(\Sigma))$ with $(\varphi,1)_{\Sigma}=0$ a.e.\ in $I$, there holds 
    \begin{align}\label{def:fluxfree_solution.1}
         (\partial_t u,\varphi)_{I\times\Sigma}+(\mathbf{s}(\cdot,\nabla u+\alpha \nabla\chi), \nabla \varphi)_{I\times\Sigma}=-(\partial_t \alpha \chi,\varphi)_{I\times \Sigma}\,.
    \end{align}
\end{definition}

%
   If 
   $(v,\Gamma)\in 
   (L^\infty(I;W^{1,p(\cdot)}_0(\Sigma))\cap W^{1,2}(I;L^2(\Sigma)))\times L^2(I)$ is a variational solution (in the sense of Def.~\ref{def:weak_solution}), then the function $u\in L^\infty(I;W^{1,p(\cdot)}_0(\Sigma))\cap W^{1,2}(I;L^2(\Sigma))$,  defined by \eqref{def:u}, is a flux-free solution (in the sense of Def.~\ref{def:fluxfree_solution}). 
   
   The following lemma establishes the converse: from a flux-free solution $u\in 
   L^\infty(I;W^{1,p(\cdot)}_0(\Sigma))\cap W^{1,2}(I;L^2(\Sigma))$ (in the sense of Def.~\ref{def:fluxfree_solution}), we can explicitly reconstruct a variational solution $(v,\Gamma)\in 
  ( L^\infty(I;W^{1,p(\cdot)}_0(\Sigma))\cap W^{1,2}(I;L^2(\Sigma)))\times L^2(I)$ (in the sense of Def.~\ref{def:weak_solution}), making the two definitions equivalent.

\begin{lemma}[Equivalence of variational and flux-free formulation]\label{lem:equiv_weak_form}
    Let $u\in 
   L^\infty(I;W^{1,p(\cdot)}_0(\Sigma))\cap W^{1,2}(I;L^2(\Sigma))$ \hspace{-0.1mm}be \hspace{-0.1mm}a \hspace{-0.1mm}flux-free \hspace{-0.1mm}solution \hspace{-0.1mm}(in \hspace{-0.1mm}the \hspace{-0.1mm}sense \hspace{-0.1mm}of \hspace{-0.1mm}Def.~\ref{def:fluxfree_solution}). \hspace{-0.1mm}Then,~\hspace{-0.1mm}the~\hspace{-0.1mm}variational~\hspace{-0.1mm}\mbox{solution} $(v,\Gamma)\hspace{-0.1em}\in \hspace{-0.1em}
  (L^\infty(I;W^{1,p(\cdot)}_0(\Sigma))\cap W^{1,2}(I;L^2(\Sigma)))\times L^2(I)$ (in the sense of Def.\  \ref{def:weak_solution})~is~\mbox{given}~via
   \begin{subequations}\label{lem:equiv_weak_form.1}
   \begin{align}\label{lem:equiv_weak_form.1.1}
       v&\coloneqq u+\alpha \chi\in L^\infty(I;W^{1,p(\cdot)}_0(\Sigma))\cap W^{1,2}(I;L^2(\Sigma))\,,\\
       \Gamma&\coloneqq -(\partial_tv,\chi)_{\Sigma}-(\mathbf{s}(\cdot,\nabla v),\nabla \chi)_{\Sigma}\in L^2(I)\,.\label{lem:equiv_weak_form.1.2}
   \end{align}
   \end{subequations}
\end{lemma}

\begin{proof}


First, for every $\eta\in L^2(I)$, we observe that $(v,\eta)_{I\times \Sigma}=(u,\eta)_{I\times \Sigma}+(\alpha,\eta)_{I}(\chi,1)_{\Sigma} 
    =(\alpha,\eta)_{I}$,
\textit{i.e.}, the flux condition \eqref{def:weak_solution-1.2} is satisfied.

Second, for every $\phi\in L^1(I;\smash{W^{1,p(\cdot)}_0(\Sigma)})\cap L^2(I;L^2(\Sigma))$ and  $\widetilde{\chi}\in \smash{W^{1,p(\cdot)}_0(\Sigma)}$ with  $(\widetilde{\chi},1)_\Sigma=1$,  $\varphi=\phi-\widetilde{\chi}(\phi,1)_\Sigma\in L^1(I;W^{1,p(\cdot)}_0(\Sigma))\cap L^2(I;L^2(\Sigma))$ satisfies $(\varphi,1)_{\Sigma}=0$ a.e.\ in $I$. Inserting the latter, for every $\phi\hspace{-0.1em}\in\hspace{-0.1em} L^1(I;W^{1,p(\cdot)}_0(\Sigma))\cap L^2(I;L^2(\Sigma))$ and  $\widetilde{\chi}\hspace{-0.1em}\in\hspace{-0.1em} \smash{W^{1,p(\cdot)}_0(\Sigma)}$ with~${(\widetilde{\chi},1)_\Sigma\hspace{-0.1em}=\hspace{-0.1em}1}$,~we~find~that
\begin{align*}
         (\partial_t v,\phi)_{I\times\Sigma}&+(\mathbf{s}(\cdot,\nabla v), \nabla \phi)_{I\times\Sigma}
         =((\partial_t v,\widetilde{\chi})_{\Sigma}+(\mathbf{s}(\cdot,\nabla v),\nabla\widetilde{\chi})_{\Sigma},(\phi,1)_\Sigma)_{I} 
         \,,
    \end{align*}
    so that, for every $\widetilde{\chi}\in \smash{W^{1,p(\cdot)}_0(\Sigma)}$ with  $(\widetilde{\chi},1)_\Sigma=1$,  setting 
\begin{equation}\label{lem:equiv_weak_form.2}
    \Gamma_{\smash{\widetilde{\chi}}}\coloneqq -(\partial_t v,\widetilde{\chi})_\Sigma - (\mathbf{s}(\cdot,\nabla v),\nabla\widetilde{\chi})_\Sigma\in L^2(I)\,,
\end{equation}
for every $\smash{\phi\in L^1(I;W^{1,p(\cdot)}_0(\Sigma))\cap L^2(I;L^2(\Sigma))}$, it turns out that
\begin{align}\label{lem:equiv_weak_form.3}
         (\partial_t v,\phi)_{I\times\Sigma}+(\mathbf{s}(\cdot,\nabla v), \nabla \phi)_{I\times\Sigma}
         =-(\Gamma_{\smash{\widetilde{\chi}}},\phi)_{I\times\Sigma}\,.
    \end{align}
    Since the left-hand side in \eqref{lem:equiv_weak_form.3} is independent of the function $\widetilde{\chi}\hspace{-0.1em}\in\hspace{-0.1em} \smash{W^{1,p(\cdot)}_0(\Sigma)}$ with  $(\widetilde{\chi},1)_\Sigma\hspace{-0.1em}=\hspace{-0.1em}1$,~the 
    mapping $(\widetilde{\chi}\mapsto \Gamma_{\smash{\widetilde{\chi}}})\colon\{\widetilde{\chi}\in \smash{W^{1,p(\cdot)}_0(\Sigma)}\mid (\widetilde{\chi},1)_{\Sigma}=1\}\to L^2(I)$ is constant, so that omitting the subscript $\widetilde{\chi}$ in \eqref{lem:equiv_weak_form.2}, leading to \eqref{lem:equiv_weak_form.1.2}, is justified. Eventually, inserting ${\Gamma\coloneqq\Gamma_{\smash{\chi}}\in L^2(I)}$~in~\eqref{lem:equiv_weak_form.3}, we conclude that \eqref{def:weak_solution-1.1} is satisfied.
\end{proof}
\newpage

\textcolor{black}{Let us emphasize that this is not a Stokes-type velocity-pressure saddle-point formulation. After the fully-developed reduction, incompressibility is already incorporated in the ansatz, and $\Gamma\in L^2(I)$ is only a constant-in-space Lagrange multiplier associated with the flow-rate constraint. Therefore, no velocity-pressure inf-sup condition is involved: the velocity is determined through the flux-free formulation by monotonicity, and $\Gamma\in L^2(I)$ is then recovered from \eqref{lem:equiv_weak_form.1.2}.}

 We aim to establish the well-posedness of the variational formulation (in the sense of Def.~\ref{def:weak_solution}) by means of a fully-discrete finite-difference/-element discretization.~To~this~end, 
 we approximate the flow rate $\alpha\in W^{1,q}(I)$ with temporally piece-wise constant \mbox{functions}\vspace{-0.5mm} 
 \begin{align}\label{def:alpha_approx}
     \alpha^{\tau}\coloneqq \mathrm{I}^0_{\tau}\alpha\in \mathbb{P}^0(\mathcal{I}_\tau^0)\,,\quad \tau>0\,,\\[-6mm]\notag
 \end{align}
 which, for some constant $c>0$, independent of $\tau>0$, satisfy the following~standard~estimates\vspace{-0.5mm}
 \begin{subequations}\label{eq:alpha_approx_esti}
 \begin{align}\label{eq:alpha_approx_esti.1}
     \|\alpha-\alpha^\tau\|_{q,I}+\tau\, \|\mathrm{d}_\tau \alpha^\tau\|_{q,I}&\leq c\,\tau \|\partial_t\alpha\|_{q,I}\,,\\
     \|\alpha^{\tau}\|_{\infty,I}&\leq c\,\smash{(\|\alpha\|_{q,I}+\|\partial_t\alpha\|_{q,I})}\,.\label{eq:alpha_approx_esti.2}\\[-6mm]\notag
 \end{align}
 \end{subequations}
 Note that, by a Sobolev embedding, we have that $\alpha\in C^0(\overline{I})$, so that \eqref{def:alpha_approx} is indeed well-defined. Alternative discretizations of $\alpha\in W^{1,q}(I)$ are possible as well (\textit{e.g.}, using $\alpha^{\tau}\coloneqq \Pi_\tau^0\alpha$, \textit{cf}.~\eqref{def:Pit}).\smallskip

Then, we consider the following fully-discrete finite-difference/-element discretization. 
 \begin{definition}[Discrete variational formulation]\label{scheme:weak_solution}
    For a finite number of time steps $M\in \mathbb{N}$ and step size $\tau\coloneqq \frac{L}{M}>0$, a pair\vspace{-1mm}
    \begin{align*}
        (v_h^\tau,\Gamma^\tau_h)\in \mathbb{P}^0(\mathcal{I}_\tau^0;V_h)\times \mathbb{P}^0(\mathcal{I}_\tau)\,,\\[-6mm]\notag
    \end{align*}
    is called \emph{discrete (variational) solution} of \eqref{eq:periodic_pLaplace}
    if\vspace{-0.5mm} \begin{align}\label{scheme:weak_solution.1}
v_h^\tau(0)=v_h^\tau(L)\quad\text{ a.e.\ in }\Sigma\,,\\[-6mm]\notag
    \end{align} 
%
    and for every $(\phi_h^\tau,\eta^\tau) \in \mathbb{P}^0(\mathcal{I}_\tau;V_h)\times \mathbb{P}^0(\mathcal{I}_\tau)$, there holds\vspace{-0.5mm}
\begin{subequations}\label{scheme:weak_solution.3} 
    \begin{align}
        (\mathrm{d}_{\tau} v_h^\tau,\phi_h^\tau)_{I\times\Sigma}+(\mathbf{s}(\cdot,\nabla v_h^\tau),\nabla \phi_h^\tau)_{I\times\Sigma}+(\Gamma^\tau_h,\phi_h^\tau)_{I\times\Sigma}&=0\,,\label{scheme:weak_solution.3.1} \\
        (v_h^\tau,\eta^\tau)_{I\times\Sigma}&=(\alpha^\tau,\eta^\tau)_{I}\,. \label{scheme:weak_solution.3.2}\\[-6mm]\notag
    \end{align}
    \end{subequations}
\end{definition}

By shifting the discrete variational formulation (\textit{cf}.\  Definition~\ref{scheme:weak_solution}),
we can incorporate a discrete flux-free condition in both the trial and the test function space. To this end, setting\footnote{Since $(\chi,1)_{\Sigma}=1$, by \eqref{eq:chih_approx_stab.1}, we have that $(\Pi_h\chi,1)_{\Sigma}>0$ for $h>0$ sufficiently small, which we assume in the rest of the paper.}\begin{align}\label{def:chih}
        \chi_h\coloneqq \smash{\tfrac{1}{(\Pi_h\chi,1)_{\Sigma}}}\Pi_h\chi\in V_h\quad\text{ with }\quad (\chi_h,1)_{\Sigma}=1\,,
    \end{align} 
    which, for any $n>0$ and some $c=c(n,p)>0$, independent of $h>0$, satisfies (\textit{cf}.\ Lemma~\ref{lem:chih_bound})\vspace{-0.5mm}
    \begin{subequations}\label{eq:chih_approx_stab}
    \begin{align}\label{eq:chih_approx_stab.1}
        \rho_{p(\cdot),\Sigma}(\chi-\chi_h)&\leq c\,\bigl(h^n+\rho_{p(\cdot),\Sigma}(h\|\nabla \chi\|_{1,\Sigma} \chi)+\rho_{p(\cdot),\Sigma}(h\nabla \chi)\bigr)\,,\\[-0.5mm]\label{eq:chih_approx_stab.1.2}
        \textcolor{black}{\rho_{p(\cdot),\Sigma}(\nabla \chi_h)}&\leq  \textcolor{black}{c\,\bigl(h^{n-1}+\rho_{p(\cdot),\Sigma}(\|\nabla \chi\|_{1,\Sigma} \chi)+\rho_{p(\cdot),\Sigma}(\nabla \chi)\bigr)}\,,\\[-0.5mm]
        \|\chi_h\|_{\Sigma}&\leq c\,\|\nabla \chi\|_{p(\cdot),\Sigma}\,,\label{eq:chih_approx_stab.2}\\[-6mm]\notag
    \end{align}
    \end{subequations}
we make the ansatz\vspace{-0.5mm}\enlargethispage{10mm}
\begin{align*}
    u_h^\tau \coloneqq v_h^\tau-\alpha^\tau \chi_h\in \mathbb{P}^0(\mathcal{I}_\tau^0;V_h)\,,\\[-6mm]\notag
\end{align*}
leading to the following discrete flux-free formulation. 

\begin{definition}[Discrete flux-free formulation]\label{scheme:fluxfree_solution}
    For a finite number of time steps $M\in \mathbb{N}$ and  step size $\tau\coloneqq \frac{L}{M}>0$, a function\vspace{-0.5mm}
    \begin{align*}
        u_h^\tau\in \mathbb{P}^0(\mathcal{I}_\tau^0;V_h)\,,\\[-6mm]\notag
    \end{align*}
   is called \emph{discrete flux-free solution} of \eqref{eq:periodic_pLaplace} if\vspace{-0.5mm}
    \begin{subequations} \label{scheme:fluxfree_solution.1}
    \begin{alignat}{2}\label{scheme:fluxfree_solution.1.1}
u_h^\tau(0)&=u_h^\tau(L)&&\quad \text{ a.e.\ in }\Sigma\,,\\
(u_h^\tau,1)_{\Sigma}&=0&&\quad \text{ a.e.\ in }I\,,\label{scheme:weak_solution.1.2}\\[-6mm]\notag
    \end{alignat}
    \end{subequations}
    and for every $\varphi_h^\tau\in \mathbb{P}^0(\mathcal{I}_\tau;V_h)$ with $(\varphi_h^\tau,1)_{\Sigma}=0$ a.e.\ in $I$, there holds\vspace{-0.5mm}
    \begin{align}\label{scheme:fluxfree_solution.2}
        (\mathrm{d}_{\tau} u_h^\tau,\varphi_h^\tau)_{I\times\Sigma}+(\mathbf{s}(\cdot,\nabla u_h^\tau+\alpha^\tau\nabla \chi_h),\nabla \varphi_h^\tau)_{I\times\Sigma}
        =-(\mathrm{d}_\tau \alpha^\tau \chi_h,\varphi_h^\tau)_{I\times \Sigma}\,.\\[-6mm]\notag
    \end{align} 
\end{definition}\newpage 
The \hspace{-0.1mm}following \hspace{-0.1mm}lemma \hspace{-0.1mm}(whose \hspace{-0.1mm}proof \hspace{-0.1mm}is \hspace{-0.1mm}the \hspace{-0.1mm}finite \hspace{-0.1mm}dimensional \hspace{-0.1mm}counterpart~\hspace{-0.1mm}of~\hspace{-0.1mm}the~\hspace{-0.1mm}one~\hspace{-0.1mm}of~\hspace{-0.1mm}Lemma~\hspace{-0.1mm}\ref{lem:equiv_weak_form}) establishes that from a discrete flux-free solution $u_h^\tau\in 
   \mathbb{P}^0(\mathcal{I}_\tau^0;V_h)$ (in the sense of Def.\ \ref{scheme:fluxfree_solution}), we can explicitly reconstruct a discrete variational solution $(v_h^\tau,\Gamma^\tau_h )\in \mathbb{P}^0(\mathcal{I}_\tau^0;V_h)\times \mathbb{P}^0(\mathcal{I}_\tau)$ (in the sense of~Definition~\ref{scheme:weak_solution}), making both definitions equivalent.\vspace{-0.5mm}
\begin{lemma}\label{lem:equiv_weak_form_discrete}
    Let $u_h^\tau\in 
   \mathbb{P}^0(\mathcal{I}_\tau^0;V_h)$ be a discrete flux-free solution (\textcolor{black}{in the sense of Def.\ \ref{scheme:fluxfree_solution}}). Then, a discrete variational solution $(v_h^\tau,\Gamma^\tau_h )\in 
   \mathbb{P}^0(\mathcal{I}_\tau^0;V_h)\times \mathbb{P}^0(\mathcal{I}_\tau)$ (\textcolor{black}{in the sense of Def.\ \ref{scheme:weak_solution}}) is given via\vspace{-1mm} 
   \begin{subequations}\label{lem:equiv_weak_form_discrete.1} 
   \begin{align}
       v_h^\tau &\coloneqq u_h^\tau+\alpha^\tau \chi_h\in \mathbb{P}^0(\mathcal{I}_\tau^0;V_h)\,,\label{lem:equiv_weak_form_discrete.1.1} \\
       \Gamma^\tau_h &\coloneqq -(\mathrm{d}_\tau v_h^\tau,\chi_h)_{\Sigma}-(\mathbf{s}(\cdot,\nabla v_h^\tau ),\nabla \chi_h)_{\Sigma}\in \mathbb{P}^0(\mathcal{I}_\tau)\,.\label{lem:equiv_weak_form_discrete.1.2}\\[-6mm]\notag
   \end{align}
   \end{subequations}
\end{lemma} 
 
To begin with, let us prove the well-posedness  (\textit{i.e.}, its unique solvability)  and weak stability  (\textit{i.e.}, \textit{a priori} bounds in the energy norm) of  the discrete variational formulation (\textit{cf}.\ Definition~\ref{scheme:weak_solution}).\enlargethispage{6mm}\vspace{-0.5mm}
\begin{lemma}[Well-posedness and weak stability]\label{lem:scheme:well_posedness}
   There exists a unique discrete (variational) solution  
    $(v_h^\tau,\Gamma^\tau_h)\in \mathbb{P}^0(\mathcal{I}_\tau^0;V_h)\times \mathbb{P}^0(\mathcal{I}_\tau)$  (in the sense of Def.\  \ref{scheme:weak_solution}). Moreover, there exists a constant $K_w>0$ such that for every $h>0$ and $\tau\in (0,1)$, we have that\vspace{-0.5mm}
    \begin{align}\label{eq:weak_stability}
            \rho_{p(\cdot),I\times\Sigma}(\nabla v_h^\tau)\leq K_w\,.\\[-6mm]\notag
    \end{align}  
\end{lemma}
The \hspace{-0.15mm}well-posedness \hspace{-0.15mm}of  \hspace{-0.15mm}Definition \hspace{-0.15mm}\ref{scheme:weak_solution} \hspace{-0.15mm}is \hspace{-0.15mm}based \hspace{-0.15mm}on \hspace{-0.1mm}a \hspace{-0.15mm}generalization~\hspace{-0.15mm}of~\hspace{-0.15mm}{Banach}'s~\hspace{-0.15mm}fixed~\hspace{-0.15mm}point~\hspace{-0.15mm}\mbox{theorem} for (only) contractive self-mappings on compact metric spaces,~\hspace{-0.1mm}tracing~\hspace{-0.1mm}back~\hspace{-0.1mm}to~\hspace{-0.1mm}{Edelstein}~\hspace{-0.1mm}(\textit{cf}.~\hspace{-0.1mm}\cite{Edelstein1962}).\vspace{-0.5mm}

\begin{theorem}[{Edelstein} fixed point theorem]\label{thm:edelstein}
    Let $(X,d)$ be a compact metric space and $\mathcal{F}\colon X\to X$ a contraction, \textit{i.e.}, for every $x,y\in X$ with $x\neq y$, it holds~that~${d(\mathcal{F}(x),\mathcal{F}(y))<d(x,y)}$.
    Then, the following statements apply:
    \begin{itemize}[noitemsep,topsep=2pt,leftmargin=!,labelwidth=\widthof{(ii)}]
        \item[(i)]\hypertarget{thm:edelstein.i}{} There exists a unique  $x^*\in X$ such that $\mathcal{F}(x^*)=x^*$ in $X$;
        \item[(ii)]\hypertarget{thm:edelstein.ii}{} For every starting point $x_0\in X$, the corresponding {Picard} iteration $(x_n)_{n\in \mathbb{N}}\subseteq X$, recursively defined by  $x_n\coloneqq \mathcal{F}(x_{n-1})$ for all $n\in \mathbb{N}$, satisfies $d(x_n,x^*)\to 0$ $(n\to \infty)$.
    \end{itemize}
\end{theorem}

\begin{proof}
    See \cite[Thm.\ 1, Rem.\ 3]{Edelstein1962}.
\end{proof}
A key ingredient in the verification that the {Edelstein} fixed point theorem is applicable is the following discrete {Gronwall} lemma in difference form,  tracing back to {Emmrich} (\textit{cf}.\ \cite{Emmrich1999}).\vspace{-0.5mm}
\begin{lemma}[Discrete {Gronwall} lemma in difference form]\label{lem:gronwall}
    Let $a^\tau\in \mathbb{P}^0(\mathcal{I}_\tau^0)$,  $g^\tau\in \mathbb{P}^0(\mathcal{I}_\tau)$,  
    and $\lambda\in \mathbb{R}\setminus\{0\}$ be such that $\mathrm{d}_\tau a^\tau\leq \lambda\, a^\tau +g^\tau$ a.e.\ in $I$. 
    If $1-\lambda \tau>0$, then for every $m=1,\ldots,M$, there holds\vspace{-1mm}
        \begin{align*}
            a^\tau(t_m)\leq \tfrac{1}{(1-\lambda \tau)^m}a^\tau(0)+\bigl(\tfrac{1}{(1-\lambda \tau)^m}-1\bigr)\tfrac{\|g^\tau\|_{\infty,I}}{\lambda}\,.\\[-6.5mm]\notag
        \end{align*}
\end{lemma}
\begin{proof}
    See \cite[Prop.~3.1]{Emmrich1999}.
\end{proof}
We now have everything at our disposal to prove Lemma~\ref{lem:scheme:well_posedness}.\vspace{-1mm}
\begin{proof}[Proof (of Lemma~\ref{lem:scheme:well_posedness}).]
The proof is divided into three main steps: 
\smallskip

\textit{1. Solvability:}
In order to apply the {Edelstein}~fixed~point~theorem (\textit{cf}.~\mbox{Theorem}~\ref{thm:edelstein}),~we~re-cast the discrete (variational) formulation (in the sense of Def.~\ref{scheme:weak_solution}) into a~fixed~point~problem. This is achieved by considering for fixed, but
arbitrary, discrete initial value $\widetilde{v}_h^0\in V_h$, the discrete initial value problem that seeks $(\widetilde{v}_h^\tau,\widetilde{\Gamma}^\tau)\in  \mathbb{P}^0(\mathcal{I}_\tau^0;V_h)\times \mathbb{P}^0(\mathcal{I}_\tau)$ with\vspace{-0.5mm}
\begin{align}\label{lem:scheme:well_posedness.0} 
  \widetilde{v}_h^\tau(0)=  \widetilde{v}_h^0\quad \text{ a.e.\ in }\Sigma\,,\\[-6mm]\notag
\end{align}
such that for every $(\phi_h^\tau,\eta^\tau) \in  \mathbb{P}^0(\mathcal{I}_\tau ;V_h)\times \mathbb{P}^0(\mathcal{I}_\tau )$, there holds\vspace{-0.5mm}
    \begin{subequations}\label{lem:scheme:well_posedness.1} 
    \begin{align}
        (\mathrm{d}_{\tau} \widetilde{v}_h^\tau,\phi_h^\tau)_{I\times\Sigma}+(\mathbf{s}(\cdot,\nabla \widetilde{v}_h^\tau),\nabla \phi_h^\tau)_{I\times\Sigma}+(\widetilde{\Gamma}^\tau,\phi_h^\tau)_{I\times\Sigma}&=0\,,\label{lem:scheme:well_posedness.1.1}\\
        (\widetilde{v}_h^\tau,\eta^\tau)_{I\times\Sigma}&=(\alpha^\tau,\eta^\tau)_{I}\,,\label{lem:scheme:well_posedness.1.2-new}\\[-6mm]\notag
    \end{align}
    \end{subequations}
    and, then, to seek the unique fixed point of 
    the operator 
    $\mathcal{F}_h^\tau\colon V_h\to V_h$, for every $\widetilde{v}_h^0\in V_h$~defined~by\vspace{-0.5mm}
    \begin{align}\label{lem:scheme:well_posedness.2} 
       \mathcal{F}_h^\tau(\widetilde{v}_h^0)\coloneqq \widetilde{v}_h^\tau(L)\quad \text{ in }V_h\,.\\[-6mm]\notag
    \end{align}
    Apparently, any (unique) fixed point of $\mathcal{F}_h^\tau\colon V_h\to V_h$ is the unique discrete (variational) solution (in the sense of Def.~\ref{scheme:weak_solution}).
     Therefore, 
    we establish next that $\mathcal{F}_h^\tau\colon V_h\to V_h$ 
    is well-defined, contractive, and a self-mapping on a (finite-dimensional) closed ball of large enough radius:\smallskip\enlargethispage{7.5mm}

    \emph{1.1. Well-definedness of $\mathcal{F}_h^\tau$:} In order to verify the well-definedness of  ${\mathcal{F}_h^\tau\colon\hspace{-0.15em}  V_h\hspace{-0.15em}\to   \hspace{-0.15em}V_h}$,~from~\eqref{lem:scheme:well_posedness.2}, we need to establish the unique solvability of the discrete initial~value~problem~\eqref{lem:scheme:well_posedness.0}--\eqref{lem:scheme:well_posedness.1}. To this end, we shift the discrete initial value problem~\eqref{lem:scheme:well_posedness.0}--\eqref{lem:scheme:well_posedness.1} into a flux-free discrete initial~value problem. More precisely, given the auxiliary function $\chi_h\in V_h$, defined~by~\eqref{def:chih}, 
    we consider the discrete initial value problem that seeks  $\widetilde{u}_h^\tau\in  \mathbb{P}^0(\mathcal{I}_\tau^0;V_h)$ with 
    \begin{subequations}\label{lem:scheme:well_posedness.1.2} 
        \begin{alignat}{2}
            \widetilde{u}_h^\tau(0)&= \widetilde{v}_h^0-\alpha^\tau(0)\chi_h&&\quad \text{ in }\Sigma\,,\label{lem:scheme:well_posedness.1.2.1}\\
            (\widetilde{u}_h^\tau,1)_{\Sigma}&=0&&\quad\text{ a.e.\ in }I\,,\label{lem:scheme:well_posedness.1.2.2}
        \end{alignat}
    \end{subequations} 
    such that for every $\varphi_h^\tau\in \mathbb{P}^0(\mathcal{I}_\tau ;V_h)$ with $(\varphi_h^\tau,1)_{\Sigma}\hspace{-0.1em}=\hspace{-0.1em}0$~a.e.~in~$I$,~there~holds
    \begin{align}\label{lem:scheme:well_posedness.1.3} 
        (\mathrm{d}_{\tau} \widetilde{u}_h^\tau,\varphi_h^\tau)_{I\times\Sigma}+(\mathbf{s}(\cdot,\nabla \widetilde{u}_h^\tau+\alpha^\tau\nabla\chi_h ),\nabla \varphi_h^\tau)_{I\times\Sigma}&=-(\mathrm{d}_\tau\alpha^\tau \chi_h,\varphi_h^\tau)_{I\times\Sigma}\,.
    \end{align} 
    By monotone operator theory (\textit{cf}.\ \cite[§26.2]{ZeidlerII/B}), for every initial value $\widetilde{v}_h^0\in V_h$, the discrete initial value problem~\eqref{lem:scheme:well_posedness.1.2}--\eqref{lem:scheme:well_posedness.1.3} admits a unique solution $\widetilde{u}_h^\tau\in\mathbb{P}^0(\mathcal{I}_\tau^0;V_h)$. It~is~\mbox{readily}~checked~that
    \begin{subequations}\label{lem:scheme:well_posedness.1.4} 
    \begin{align}
        \widetilde{v}_h^\tau&\coloneqq \widetilde{u}_h^\tau+\alpha^\tau \chi_h\in \mathbb{P}^0(\mathcal{I}_\tau^0;V_h)\,,\label{lem:scheme:well_posedness.1.4.1}\\
        \widetilde{\Gamma}^\tau&\coloneqq -(\mathrm{d}_\tau \widetilde{v}_h^\tau,\chi_h)_{\Sigma}-(\mathbf{s}(\cdot,\nabla \widetilde{v}_h^\tau),\nabla\chi_h)_{\Sigma}\in \mathbb{P}^0(\mathcal{I}_\tau)\,,\label{lem:scheme:well_posedness.1.4.2}
    \end{align}
    \end{subequations}
    is the  unique solution of the discrete initial value~problem~\eqref{lem:scheme:well_posedness.1}. In other words, the fixed point operator  $\mathcal{F}_h^\tau\colon V_h\to V_h$, defined by \eqref{lem:scheme:well_posedness.2}, is indeed well-defined.\smallskip

    \emph{1.2.\ \hspace{-0.175mm}Contraction \hspace{-0.175mm}property \hspace{-0.175mm}of \hspace{-0.175mm}$\mathcal{F}_h^\tau$:}\hypertarget{1.2}{} \hspace{-0.25mm}Let \hspace{-0.175mm}$\widetilde{v}_h^0,\widetilde{w}_h^0\hspace{-0.2em}\in\hspace{-0.2em} V_h$ \hspace{-0.175mm}be \hspace{-0.175mm}two \hspace{-0.175mm}fixed, \hspace{-0.175mm}but \hspace{-0.175mm}arbitrary~\hspace{-0.175mm}discrete~\hspace{-0.175mm}\mbox{initial}~\hspace{-0.175mm}\mbox{values} with $\widetilde{v}_h^0\neq\widetilde{w}_h^0$ and 
     $(\widetilde{v}_h^\tau,\widetilde{\Gamma}^\tau),(\widetilde{w}_h^\tau,\widetilde{\Lambda}^\tau)\in  \mathbb{P}^0(\mathcal{I}_\tau^0;V_h)\times \mathbb{P}^0(\mathcal{I}_\tau)$ the associated solutions~of~the~\mbox{discrete} initial value problem \eqref{lem:scheme:well_posedness.0}--\eqref{lem:scheme:well_posedness.1}.  Then, for every ${(\phi_h^\tau,\eta^\tau) \in  \mathbb{P}^0(\mathcal{I}_\tau;V_h)\times\mathbb{P}^0(\mathcal{I}_\tau)}$,~there~holds
    \begin{subequations}\label{lem:scheme:well_posedness.3} 
    \begin{align}
        (\mathrm{d}_{\tau} (\widetilde{v}_h^\tau-\widetilde{w}_h^\tau),\phi_h^\tau)_{I\times\Sigma}+(\mathbf{s}(\cdot,\nabla \widetilde{v}_h^\tau)-\mathbf{s}(\cdot,\nabla \widetilde{w}_h^\tau),\nabla \phi_h^\tau)_{I\times\Sigma}+(\widetilde{\Gamma}^\tau-\widetilde{\Lambda}^\tau,\phi_h^\tau)_{I\times\Sigma}&=0\,,\label{lem:scheme:well_posedness.3.1}\\
        (\widetilde{v}_h^\tau-\widetilde{w}_h^\tau,\eta^\tau)_{I\times\Sigma}&=0\,.\label{lem:scheme:well_posedness.3.2}
    \end{align}
    \end{subequations}
    Choosing $\varphi_h^\tau=\widetilde{v}_h^\tau-\widetilde{w}_h^\tau\in  \mathbb{P}^0(\mathcal{I}_\tau;V_h)$ in \eqref{lem:scheme:well_posedness.3.1}, due to $(\widetilde{\Gamma}^\tau-\widetilde{\Lambda}^\tau,\widetilde{v}_h^\tau-\widetilde{w}_h^\tau)_{I\times \Sigma}=0$ (\textit{cf}.\ \eqref{lem:scheme:well_posedness.3.2}), also using the discrete integration-by-parts formula \eqref{eq:discrete_integration-by-parts_reduced}, we  obtain
    \begin{align}
    \label{lem:scheme:well_posedness.4}
    \begin{aligned} 
        [\tfrac{1}{2}\|\widetilde{v}_h^\tau(t_m)-\widetilde{w}_h^\tau(t_m)\|_{\Sigma}^2]_{m=0}^{m=M}&+\tfrac{\tau}{2}\|\mathrm{d}_{\tau}(\widetilde{v}_h^\tau-\widetilde{w}_h^\tau)\|_{I\times\Sigma}^2\\&
        +(\mathbf{s}(\cdot,\nabla \widetilde{v}_h^\tau)-\mathbf{s}(\cdot,\nabla \widetilde{w}_h^\tau),\nabla \widetilde{v}_h^\tau-\nabla \widetilde{w}_h^\tau)_{I\times\Sigma}=0\,.
        \end{aligned}
    \end{align} 
    We consider two cases:
    
    $\bullet$ \emph{Case 1:} If $\widetilde{v}_h^\tau(L)\neq\widetilde{w}_h^\tau(L)$,  by the strict monotonicity of $\mathbf{s}(x,\cdot)\colon \mathbb{R}^{d-1}\to \mathbb{R}^{d-1}$ for a.e.\ $x\in \Sigma$ (\textit{cf}.\ (\hyperlink{s.4}{s.4})), we~have~that $(\mathbf{s}(\cdot,\nabla \widetilde{v}_h^\tau)-\mathbf{s}(\cdot,\nabla \widetilde{w}_h^\tau),\nabla \widetilde{v}_h^\tau-\nabla \widetilde{w}_h^\tau)_{I\times\Sigma}>0$,
    and, thus, \eqref{lem:scheme:well_posedness.4}~implies~that
    \begin{align*}
        \|\mathcal{F}^\tau_h(\widetilde{v}^0_h)-\mathcal{F}^\tau_h(\widetilde{w}^0_h)\|_{\Sigma}=\|\widetilde{v}_h^\tau(L)-\widetilde{w}_h^\tau(L)\|_{\Sigma}< \|\widetilde{v}^0_h-\widetilde{w}^0_h\|_{\Sigma}\,;
    \end{align*}
    
  $\bullet$ \emph{Case 2:} If $\widetilde{v}_h^\tau(L)=\widetilde{w}^\tau_h(L)$, then, due to $\widetilde{v}^0_h\neq\widetilde{w}^0_h$, we have that 
    \begin{align*}
        \|\mathcal{F}^\tau_h(\widetilde{v}^0_h)-\mathcal{F}^\tau_h(\widetilde{w}^0_h)\|_{\Sigma}=\|\widetilde{v}_h^\tau(L)-\widetilde{w}_h^\tau(L)\|_{\Sigma}=0< \|\widetilde{v}^0_h-\widetilde{w}^0_h\|_{\Sigma}\,.
    \end{align*}
    In summary, the fixed point operator $\mathcal{F}^\tau_h\colon V_h\to V_h$, defined by \eqref{lem:scheme:well_posedness.2}, is a contraction.\smallskip

    \emph{1.3. Self-mapping property of $\mathcal{F}_h^\tau$:} If we choose 
    $\phi_h ^\tau=(\widetilde{v}_h^\tau-\alpha^\tau\chi_h)\chi_{I_m}\in \mathbb{P}^0(\mathcal{I}_\tau;V_h)$ in \eqref{lem:scheme:well_posedness.1.1} for all $m\in  \{1,\ldots,M\}$,  due to $(\widetilde{\Gamma}^\tau,\widetilde{v}_h^\tau-\alpha^\tau\chi_h)_{I_m\times\Sigma}=0$~(\textit{cf}.~\eqref{lem:scheme:well_posedness.1.2-new}) and the (variable) $\varepsilon$-Young inequality \eqref{eq:eps-young.2}, for every $\delta, \varepsilon>0$, we find that
    \begin{align}\label{lem:scheme:well_posedness.5}  
    \left.\begin{aligned}
        \mathrm{d}_\tau\{\tfrac{1}{2}\|\widetilde{v}^\tau_h\|_{\Sigma}^2\bigr)&+\tfrac{\tau}{2}\|\mathrm{d}_\tau \widetilde{v}_h^\tau\|_{\Sigma}^2+(\mathbf{s}(\cdot,\nabla \widetilde{v}_h^\tau),\nabla\widetilde{v}_h^\tau)_{\Sigma}\\&=(\mathrm{d}_{\tau}\widetilde{v}_h^\tau,\alpha^\tau\chi_h)_{\Sigma}+(\mathbf{s}(\cdot,\nabla\widetilde{v}_h^\tau),\alpha^\tau\nabla \chi_h)_{\Sigma}
        \\&\leq\tfrac{\delta}{2}\|\mathrm{d}_{\tau}\widetilde{v}_h^\tau\|_{\Sigma}^2+ \tfrac{1}{2\delta}\|\alpha^\tau\chi_h\|_{\Sigma}^2
        \\&\quad + \varepsilon\,(\rho_{p(\cdot),\Sigma}(\nabla\widetilde{v}_h^\tau)+\rho_{p'(\cdot),\Sigma}(\kappa_4))+c_{\varepsilon}\,\rho_{p(\cdot),\Sigma}(\alpha^\tau\nabla \chi_h)
        \end{aligned}\quad\right\}\quad \text{ a.e.\ in }I\,,
    \end{align}
    where, due to \eqref{eq:alpha_approx_esti} and \eqref{eq:chih_approx_stab}, we have that
    \begin{alignat}{2} 
        &\hspace*{0.9cm}\left.\begin{aligned}
        \|\alpha^\tau\chi_h\|_{\Sigma}^2&\leq  \|\alpha\|_{\infty,I}^2\|\chi_h\|_{\Sigma}^2
        \\& \leq c\,(\|\alpha\|_{q,I}+\|\partial_t\alpha\|_{q,I})^2\|\nabla \chi\|_{p(\cdot),\Sigma}^2
        \end{aligned}\hspace*{2.2cm}\quad \right\}\text{ a.e.\ in }I\,,\label{lem:scheme:well_posedness.6}\\
        &\left.\begin{aligned}
        \rho_{p(\cdot),\Sigma}(\alpha^\tau\nabla \chi_h)&\leq (1+\|\alpha^\tau\|_{\infty,I})^{\smash{p^+}}\rho_{p(\cdot),\Sigma}(\nabla \chi_h)
        \\&\leq c\,(1+\|\alpha\|_{q,I}+\|\partial_t\alpha\|_{q,I})^{\smash{p^+}}\\&\qquad\times\textcolor{black}{\bigl(h^{n-1}+\rho_{p(\cdot),\Sigma}(\|\nabla \chi\|_{1,\Sigma} \chi)+\rho_{p(\cdot),\Sigma}(\nabla \chi)\bigr)}
        \end{aligned}\quad \right\}\text{ a.e.\ in }I\,.\label{lem:scheme:well_posedness.7} 
    \end{alignat} 
From the choice $\delta=\tau$ and $\varepsilon=\frac{\kappa_1}{2}$ in \eqref{lem:scheme:well_posedness.5}, using the coercivity property  of $\mathbf{s}\colon \Sigma\times \mathbb{R}^{d-1}\to\mathbb{R}^{d-1}$ (\textit{cf}.\ (\hyperlink{s.2}{s.2})), it follows that\enlargethispage{5mm}
    \begin{align}\label{lem:scheme:well_posedness.8}
    \left.\begin{aligned}
        \mathrm{d}_\tau\{\|\widetilde{v}^\tau_h\|_{\Sigma}^2\}
        +\kappa_1\rho_{p(\cdot),\Sigma}(\nabla \widetilde{v}_h^\tau)
       & 
       \leq \tfrac{1}{\tau}\|\alpha^\tau\chi_h\|_{\Sigma}^2 +2  \|\kappa_2\|_{1,\Sigma}\\&\quad+\kappa_1\rho_{p'(\cdot),\Sigma}(\kappa_4)+2c_\varepsilon\,\rho_{p(\cdot),\Sigma}(\alpha^\tau\nabla \chi_h)
     \end{aligned}\quad\right\}\quad\text{ a.e.\ in }I\,.
    \end{align}
    The discrete {Gronwall} lemma in difference form (\textit{cf}.\ Lemma \ref{lem:gronwall} with $\lambda=1$) and  \eqref{lem:scheme:well_posedness.6},\eqref{lem:scheme:well_posedness.7} applied \hspace{-0.1mm}to \hspace{-0.1mm}\eqref{lem:scheme:well_posedness.8} \hspace{-0.1mm}yield \hspace{-0.1mm}the \hspace{-0.1mm}existence \hspace{-0.1mm}of \hspace{-0.1mm}a \hspace{-0.1mm}constant \hspace{-0.1mm}$c_0\hspace{-0.175em}>\hspace{-0.175em}0$  \hspace{-0.1mm}such \hspace{-0.1mm}that \hspace{-0.1mm}for \hspace{-0.1mm}every $m\hspace{-0.175em}=\hspace{-0.175em}1,\ldots,M$,~\hspace{-0.1mm}there~\hspace{-0.1mm}holds
    \begin{align}\label{lem:scheme:well_posedness.9}
        \begin{aligned}  \|\widetilde{v}^\tau_h(t_m)\|_{\Sigma}^2&\leq \tfrac{1}{(1-\tau)^m}\|\widetilde{v}^0_h\|_{\Sigma}^2+\bigl(\tfrac{1}{(1- \tau)^m}-1\bigr)\tfrac{c_0}{\tau}
        \\&\leq \tfrac{1}{(1-\tau)^M}\|\widetilde{v}^0_h\|_{\Sigma}^2+\bigl(\tfrac{1}{(1- \tau)^M}-1\bigr)\tfrac{c_0}{\tau} \,. 
        \end{aligned}
    \end{align}
    Since for every $r>0$, there exists a constant $c_r>0$ such that for every $a\ge 0$, there holds
    \begin{align*}
       a^2\leq \tfrac{1}{(1-\tau)^M} r^2+\bigl(\tfrac{1}{(1-\tau)^M}-1\bigr)\tfrac{c_0}{\tau}\quad \Rightarrow\quad a^2\leq a^{p^-}+c_r\,,
    \end{align*} 
    if $\|\widetilde{v}_h^0\|_{\Sigma}\leq r$ for some $r>0$, then, by \eqref{lem:scheme:well_posedness.9}, we have that
    \begin{align*}
       \|\widetilde{v}^\tau_h\|_{\Sigma}^2\leq \tfrac{1}{(1-\tau)^M}r^2+\bigl(\tfrac{1}{(1-\tau)^M}-1\bigr)\tfrac{c_0}{\tau}\quad \text{ a.e.\ in }I\,, 
    \end{align*}
and, thus,  using Poincar\'e's inequality (with constant $c_{\textrm{P}}=c_{\textrm{P}}(p^-)$) and that $a^{p^-}
\leq 2^{p^+-1}(1+a^{p(x)})$ for a.e.\ $x\in \Sigma$ and all $a\ge 0$,  we find that
    \begin{align}\label{lem:scheme:well_posedness.10}
        \left.\begin{aligned} 
        \|\widetilde{v}_h^\tau\|_{\Sigma}^2&\leq \|\widetilde{v}_h^\tau\|_{\smash{p^-},\smash{\Sigma}}^{p^-}+c_r\vert \Sigma\vert\\
        & \leq c_{\textrm{P}}\|\nabla\widetilde{v}_h^\tau\|_{\smash{p^-},\smash{\Sigma}}^{p^-}+c_r\vert \Sigma\vert
        \\
        & \leq c_{\textrm{P}}2^{\smash{p^+-1}}\bigl(\vert \Sigma\vert+\rho_{p(\cdot),\Sigma}(\nabla \widetilde{v}_h^\tau)\bigr)+c_r\vert \Sigma\vert
         \end{aligned}\quad\right\}\text{ a.e.\ in }I\,.
    \end{align}
   From the choice $\delta=\frac{\kappa_1}{c_{\textrm{P}}2^{p^+}}$ and $\varepsilon=\frac{\kappa_1}{2}$ in \eqref{lem:scheme:well_posedness.5}, if $\|\widetilde{v}_h^0\|_{\Sigma}\leq r$ for some $r>0$,  resorting to \eqref{lem:scheme:well_posedness.10}, it follows that 
    \begin{align}\label{lem:scheme:well_posedness.11}
    \left.\begin{aligned}  \mathrm{d}_\tau\{\|\widetilde{v}^\tau_h\|_{\Sigma}^2\}
        +\tfrac{\kappa_1}{c_{\mathrm{P}}2^{p^+}}\|\widetilde{v}_h^\tau\|_{\Sigma}^2
       & \leq\kappa_1 \vert \Sigma\vert\bigl(1+\tfrac{1}{c_{\mathrm{P}}2^{p^+-1}}\bigr)
       \\&\quad+\tfrac{1}{\tau}\|\alpha^\tau\chi_h\|_{\Sigma}^2+2  \|\kappa_2\|_{1,\Sigma} 
     \\&\quad+\kappa_1\rho_{p'(\cdot),\Sigma}   (\kappa_4)+2c_\varepsilon\,\rho_{p(\cdot),\Sigma}  (\alpha^\tau\nabla \chi_h)
     \end{aligned}\quad\right\}\quad \text{ a.e.\ in }I\,.
    \end{align}
   The discrete {Gronwall} lemma in difference form (\textit{cf}.\ Lemma \ref{lem:gronwall} with $\lambda=-\tfrac{\kappa_1}{c_{\mathrm{P}}2^{p^+}}$)~and~\eqref{lem:scheme:well_posedness.6}, \eqref{lem:scheme:well_posedness.7} applied to \eqref{lem:scheme:well_posedness.11} yield the existence of a constant $c_1>0$ such that for every $m=1,\ldots,M$, there~holds  
    \begin{align}\label{lem:scheme:well_posedness.10-bis}
        \|\widetilde{v}^\tau_h(t_m)\|_{\Sigma}^2 \leq \tfrac{1}{(1-\lambda \tau)^m}\|\widetilde{v}_h^0\|_{\Sigma}^2+\bigl(1-\tfrac{1}{(1-\lambda \tau)^m}\bigr)\tfrac{c_1}{-\lambda\tau }\,. 
    \end{align}
    In consequence, setting $B_h^\tau\hspace{-0.1em}\coloneqq \hspace{-0.1em}\{\varphi_h\hspace{-0.1em}\in\hspace{-0.1em} V_h\mid \|\varphi_h\|_{\Sigma}^2\hspace{-0.1em}\leq \hspace{-0.1em}\tfrac{c_1}{-\lambda \tau}\}$,
     from \eqref{lem:scheme:well_posedness.10-bis}, it follows that ${\mathcal{F}_h^\tau(B_h^\tau)\hspace{-0.1em}\subseteq\hspace{-0.1em} B_h^\tau}$.

     In summary, $\mathcal{F}_h^\tau\colon B_h^\tau\to B_h^\tau$ is well-defined and contractive, so that, by the compactness of the finite-dimensional closed ball $B_h^\tau$, the {Edelstein} fixed point theorem (\textit{cf}.\ Theorem~\ref{thm:edelstein}) yields the existence of a unique fixed point $v_h^0\in B_h^\tau$ of $\mathcal{F}_h^\tau\colon B_h^\tau\to B_h^\tau$.\enlargethispage{1mm}

    \textit{2.\ Uniqueness:} The {Edelstein} fixed point theorem (\textit{cf}.\ Theorem~\ref{thm:edelstein}) yields uniqueness in the class of discrete solutions $(v_h^\tau,\Gamma^\tau_h)\hspace{-0.1em}\in  \hspace{-0.1em}\mathbb{P}^0(\mathcal{I}_\tau^0;V_h)\times \mathbb{P}^0(\mathcal{I}_\tau)$ of \eqref{scheme:weak_solution.1}--\eqref{scheme:weak_solution.3.2}~with~${v_h^\tau(L)\hspace{-0.1em}=\hspace{-0.1em}v_h^\tau(0)\hspace{-0.1em}\in\hspace{-0.1em} B_h^\tau}$. However, uniqueness generally holds for the class of solutions $(v_h^\tau,\Gamma^\tau_h)\in  \mathbb{P}^0(\mathcal{I}_\tau^0;V_h)\times \mathbb{P}^0(\mathcal{I}_\tau)$ of \eqref{scheme:weak_solution.1}--\eqref{scheme:weak_solution.3.2} without the additional assumption that ${v_h^\tau(L),v_h^\tau(0)\in B_h^\tau}$. In order~to~see~this, let \hspace{-0.1mm}$(v_h^\tau,\Gamma^\tau_h),(w_h^\tau,\Lambda^\tau_h)\hspace{-0.1em}\in  \hspace{-0.1em}\mathbb{P}^0(\mathcal{I}_\tau^0;V_h)\hspace{-0.1em}\times \hspace{-0.1em}\mathbb{P}^0(\mathcal{I}_\tau)$ \hspace{-0.1mm}be \hspace{-0.1mm}discrete \hspace{-0.1mm}(variational) \hspace{-0.1mm}solutions \hspace{-0.1mm}solving~\hspace{-0.1mm}\mbox{\eqref{scheme:weak_solution.1}--\eqref{scheme:weak_solution.3.2}}. Then, by analogy with Step \hyperlink{Step 1.2}{1.2}, we have that\begin{align}
    \label{lem:scheme:well_posedness.4_new}
    \begin{aligned} 
        [\tfrac{1}{2}\|v_h^\tau(t_m)-w_h^\tau(t_m)\|_{\Sigma}^2]_{m=0}^{m=M}&+\tfrac{\tau}{2}\|\mathrm{d}_{\tau}(v_h^\tau-w_h^\tau)\|_{I\times\Sigma}^2\\&
        +(\mathbf{s}(\cdot,\nabla v_h^\tau)-\mathbf{s}(\cdot,\nabla w_h^\tau),\nabla v_h^\tau-\nabla w_h^\tau)_{I\times\Sigma}=0\,.
        \end{aligned}
    \end{align} 
    The time-periodicity of $v_h^\tau,w_h^\tau\in \mathbb{P}^0(\mathcal{I}_\tau^0;V_h)$ (\textit{cf}.\ \eqref{scheme:weak_solution.1}) implies that 
    \begin{align*}
        [\tfrac{1}{2}\|v_h^\tau(t_m)-w_h^\tau(t_m)\|_{\Sigma}^2]_{m=0}^{m=M}=0\,,
    \end{align*}
    so that from \eqref{lem:scheme:well_posedness.4_new}, we infer that
    \begin{align*}
        (\mathbf{s}(\cdot,\nabla v_h^\tau)-\mathbf{s}(\cdot,\nabla w_h^\tau),\nabla v_h^\tau-\nabla w_h^\tau)_{I\times\Sigma}\leq 0\,,
    \end{align*}
    which, by the strict monotonicity of $\mathbf{s}(x,\cdot)\colon \mathbb{R}^{d-1}\to \mathbb{R}^{d-1}$ for a.e.\ $x\in \Sigma$ (\textit{cf}.\ (\hyperlink{s.4}{s.4})), implies that $\nabla v_h^\tau=\nabla w_h^\tau$ a.e.\ in $I\times \Sigma$ and, 
    by Poincar\'e's inequality, that $v_h^\tau= w_h^\tau$ a.e.\ in  $I\times \Sigma$. Then, subtracting \eqref{scheme:weak_solution.3.1}, each  with $\varphi_h^\tau=\eta^\tau\chi_h\in \mathbb{P}^0(I_\tau;V_h)$, where
$\eta^\tau\in \mathbb{P}^0(\mathcal{I}_\tau)$ is arbitrary,~for~the~two solutions, using
$v_h^\tau=w_h^\tau$ a.e.\ in $I\times\Sigma$ and $(\chi_h,1)_\Sigma=1$, we obtain $ (\Gamma_h^\tau-\Lambda^\tau_h,\eta^\tau)_I=0$ for all $\eta^\tau\in \mathbb{P}^0(\mathcal{I}_\tau)$ and, thus,  $\Gamma_h^\tau=\Lambda^\tau_h$ a.e.\ in $I$.\enlargethispage{2.5mm}
     
     \textit{3.\ Weak stability estimate \eqref{eq:weak_stability}:} If we choose $\phi_h ^\tau\hspace{-0.175em}=\hspace{-0.175em}v_h^\tau-\alpha^\tau \chi_h\hspace{-0.175em}\in\hspace{-0.175em} \mathbb{P}^0(\mathcal{I}_\tau;V_h)$ in \eqref{scheme:weak_solution.3.1},~we~obtain
    \begin{align}\label{lem:scheme:weak_stability.1}
        \begin{aligned}
            (\mathrm{d}_\tau v_h^\tau, v_h^\tau)_{I\times\Sigma}+(\mathbf{s}(\cdot,\nabla v_h^\tau),\nabla v_h^\tau)_{I\times\Sigma}&= (  \alpha^\tau\chi_h,\mathrm{d}_\tau v_h^\tau)_{I\times\Sigma}
        +(\mathbf{s}(\cdot,\nabla v_h^\tau),\alpha^\tau\nabla \chi_h)_{I\times\Sigma}\,,
        \end{aligned}
    \end{align}
    where, by the discrete integration-by-parts formulas \eqref{eq:discrete_integration-by-parts_reduced},\eqref{eq:discrete_integration-by-parts}, respectively, and the time-perio\-dicity of $v_h^\tau\in \mathbb{P}^0(\mathcal{I}_\tau^0;V_h)$ (\textit{cf}.\ \eqref{scheme:weak_solution.1}) and $\alpha^\tau \in\mathbb{P}^0(\mathcal{I}_\tau^0) $, respectively,  we have that
    \begin{align}\label{lem:scheme:weak_stability.2}
        \begin{aligned}
            (\mathrm{d}_\tau v_h^\tau, v_h^\tau)_{I\times\Sigma}&=
        \tfrac{\tau}{2}\|\mathrm{d}_\tau v_h^\tau\|_{I\times\Sigma}^2+ [\tfrac{1}{2}\|v_h^\tau(t_m)\|_{\Sigma}^2]_{m=0}^{m=M}
       \\&\ge  [\tfrac{1}{2}\|v_h^\tau(t_m)\|_{\Sigma}^2]_{m=0}^{m=M}=0 \,,        \end{aligned}\hspace{2.65cm}\\[1mm]\label{lem:scheme:weak_stability.3}
        \begin{aligned} 
        (  \alpha^\tau\chi_h,\mathrm{d}_\tau v_h^\tau)_{I\times\Sigma}&=-( \mathrm{d}_\tau \alpha^\tau\chi_h, \textcolor{black}{\mathrm{T}_\tau^-} v_h^\tau)_{I\times\Sigma}
        +[(\alpha^\tau(t_m)\chi_h, v_h^\tau(t_m))_{\Sigma}]_{m=0}^{m=M}\\&=-( \mathrm{d}_\tau \alpha^\tau\chi_h, \textcolor{black}{\mathrm{T}_\tau^-} v_h^\tau)_{I\times\Sigma}\,.
        \end{aligned}
    \end{align}
    Using~\eqref{lem:scheme:weak_stability.2},\eqref{lem:scheme:weak_stability.3}, the coercivity property  of $\mathbf{s}\colon\Sigma\times\mathbb{R}^{d-1}\to \mathbb{R}^{d-1}$ (\textit{cf}.\ (\hyperlink{s.2}{s.2})), the embedding  $W^{1,p^-}_0(\Sigma)\hookrightarrow L^2(\Sigma)$,  and the (variable) $\varepsilon$-Young inequality~\eqref{eq:eps-young.2} in \eqref{lem:scheme:weak_stability.1}, for every $\delta,\varepsilon>0$, we find that
    \begin{align*}
    \begin{aligned} 
       \kappa_1\rho_{p(\cdot),I\times \Sigma}(\nabla v_h^\tau)- L\|\kappa_2\|_{1,\Sigma}&\leq c_\delta\| \mathrm{d}_\tau \alpha^\tau\|^{(p^-)'}_{(p^-)',I}\|\chi_h\|_{\Sigma}^{(p^-)'}+\delta\|\textcolor{black}{\mathrm{T}_\tau^-} \nabla v_h^\tau\|_{p^-,I\times \Sigma}^{p^-}
       \\&\quad+\varepsilon\,\bigl(\rho_{p(\cdot),I\times\Sigma}(\nabla v_h^\tau)+L\rho_{p'(\cdot),\Sigma}(\kappa_4)\bigr)\\&\quad+c_{\varepsilon}\,\rho_{p(\cdot),I\times\Sigma}(\alpha^\tau\nabla \chi_h)\,,
       \end{aligned}
    \end{align*} 
    where, due to \eqref{eq:alpha_approx_esti} and \eqref{eq:chih_approx_stab}, we have that
    \begin{align} \label{lem:scheme:weak_stability.5}
         \|\mathrm{d}_{\tau}\alpha^\tau\|_{(p^-)',I}\|\chi_h\|_{\Sigma}
        \leq c\,\|\partial_t \alpha\|_{q,I}\|\nabla \chi\|_{p(\cdot),\Sigma}\,.
    \end{align} 
    Moreover, due to the time-periodicity of $v_h^\tau\in \mathbb{P}^0(\mathcal{I}_\tau^0;V_h)$ (\textit{cf}.\ \eqref{scheme:weak_solution.1}), we have that
    \begin{align*}
        \begin{aligned} 
        \|\textcolor{black}{\mathrm{T}_\tau^-} \nabla v_h^\tau\|_{p^-,I\times\Sigma}^{p^-}&=\|\nabla v_h^\tau\|_{p^-,I\times\Sigma}^{p^-}
        \leq 2^{p^+-1}\bigl(L\vert\Sigma\vert+\rho_{p(\cdot),I\times \Sigma}(\nabla v_h^\tau)\bigr)\,.
        \end{aligned}
    \end{align*}
    Eventually, choosing $\varepsilon=\frac{\kappa_1}{4}$ and $\delta=\frac{\kappa_1}{4}2^{1-p^+}$, we arrive at
    \begin{align*}
         \tfrac{\kappa_1}{2}\rho_{p(\cdot),I\times \Sigma}(\nabla v_h^\tau)&\leq L\bigl(\|\kappa_2\|_{1,\Sigma}+\tfrac{\kappa_1}{4}\,\rho_{p'(\cdot),\Sigma}(\kappa_4)\bigr)\\&\quad+c_\delta\,\| \mathrm{d}_\tau \alpha^\tau\|^{(p^-)'}_{(p^-)',I}\|\chi_h\|_{\Sigma}^{(p^-)'} \\&\quad+c_{\varepsilon}\,\rho_{p(\cdot),I\times\Sigma}(\alpha^\tau\nabla \chi_h)\,,
    \end{align*}
    which, due to \eqref{lem:scheme:well_posedness.7} and  \eqref{lem:scheme:weak_stability.5}, is the claimed weak stability estimate \eqref{eq:weak_stability}.
\end{proof}\pagebreak

The constructive proof of Lemma \ref{lem:scheme:well_posedness} can be summarised as an algorithm, which~may~be~used~to iteratively compute the  discrete (variational) solution (in the sense of Def.\  \ref{scheme:weak_solution}).\enlargethispage{5mm}

\begin{algorithm}[H]
\caption{{Picard} iteration for approximating the discrete 
solution~of~\mbox{\eqref{scheme:weak_solution.1}--\eqref{scheme:weak_solution.3.2}}}
\label{alg:picard-iteration}
\begin{algorithmic}[1]
\Require initial guess $\widetilde{v}^0_h\hspace{-0.175em}\in\hspace{-0.175em} B_h^\tau$, tolerance $\texttt{tol}_{\textup{stop}}\hspace{-0.175em} > \hspace{-0.175em}0$, maximum iterations $\texttt{K}_{\textup{max}}\hspace{-0.175em}\in\hspace{-0.175em}\mathbb{N}$,~norm~${\|\hspace{-0.175em}\cdot\hspace{-0.175em}\|_{V_h}}$
\Ensure approximate solution $(v_h^\tau,\Gamma^\tau_h)\in \mathbb{P}^0(\mathcal{I}_\tau ^0;V_h)\times \mathbb{P}^0(\mathcal{I}_\tau )$ solving \eqref{scheme:weak_solution.1}--\eqref{scheme:weak_solution.3.2}
\State Set iteration counter: $k \coloneqq 0$
\State Set initial residual: $\mathtt{res}_h^{\tau,0} \coloneqq \textcolor{black}{+\infty}$
\While{$\mathtt{res}_h^{\tau,k} > \texttt{tol}_{\textup{stop}}$ and $k < \texttt{K}_{\textup{max}}$}
  \State Compute $(\widetilde{v}_h^{\tau,k+1},\widetilde{\Gamma}^{\tau,k+1})\in\mathbb{P}^0(\mathcal{I}_\tau ^0;V_h)\times\mathbb{P}^0(\mathcal{I}_\tau )$ solving \eqref{lem:scheme:well_posedness.0}--\eqref{lem:scheme:well_posedness.1.2-new} \textcolor{black}{with $\widetilde{v}_h^{\tau,k+1}(0)=\widetilde{v}_h^0$}
  \State Compute the residual: $\mathtt{res}_h^{\tau,k+1} \coloneqq \|\widetilde{v}_h^{\tau,k+1}(L)-\widetilde{v}_h^{\tau,k+1}(0)\|_{V_h}$
   \State Update initial value: $\widetilde{v}_h^0 \gets \widetilde{v}_h^{\tau,k+1}(L)$
  \State Update iteration: $k \gets k+1$
\EndWhile
\State \textbf{return} $(v_h^\tau,\Gamma^\tau_h) \coloneqq (\widetilde{v}_h^{\tau,k},\widetilde{\Gamma}^{\tau,k})\in \mathbb{P}^0(\mathcal{I}_\tau ^0;V_h)\times \mathbb{P}^0(\mathcal{I}_\tau )$
\end{algorithmic}
\end{algorithm}  

\begin{remark}[On convergence (rates) of Algorithm~\ref{alg:picard-iteration}]\label{rem:convergence_alg}
    \begin{itemize}[noitemsep,topsep=2pt,leftmargin=!,labelwidth=\widthof{(ii)}]
        \item[(i)] By the {Edelstein} fixed point theorem (\textit{cf}.\ Theorem \ref{thm:edelstein}(\hyperlink{thm:edelstein.ii}{ii})), for arbitrary $\widetilde{v}_h^0\in B_h^\tau$, where the diameter of $B_h^\tau$ increases~as ${\tau \to 0^+}$, Algorithm~\ref{alg:picard-iteration} converges. Independent of the smallness of $\tau>0$, for the trivial choice $\widetilde{v}_h^0=0\in B_h^\tau$, Algorithm~\ref{alg:picard-iteration}~converges; making $\widetilde{v}_h^0=0$ the default choice in Section \ref{sec:experiments}; 
        \item[(ii)] In general, we cannot make a statement about how fast Algorithm \ref{alg:picard-iteration}  converges. However,~if 
        $\mathbf{s}\colon \Sigma\times \mathbb{R}^{d-1}\to \mathbb{R}^{d-1}$ has \emph{$(p(\cdot),\delta)$-structure} (\textit{e.g.}, there exists $\delta\ge 0$ such that 
        $\mathbf{s}(x,\mathbf{a})\simeq (\delta +\vert \mathbf{a}\vert)^{p(x)-2}\mathbf{a}$ for a.e.\ $x\in \Sigma$ and all $\mathbf{a}\in \mathbb{R}^{d-1}$; for the precise definition, we refer to~\cite{BK2024}),
        then, for two solutions $\widetilde{v}_h^\tau,\widetilde{w}_h^\tau\in \mathbb{P}^0(\mathcal{I}_\tau^0;V_h)$ of the discrete initial value problem \eqref{lem:scheme:well_posedness.0}--\eqref{lem:scheme:well_posedness.1} with initial data $\widetilde{v}_h^0,\widetilde{w}_h^0\in B_h^\tau$, respectively, 
        according to \cite[Lem.\ B.1]{BK2024}, if $p^+\leq 2$,~we~have~that
        \begin{align}
        \label{rem:convergence_alg.1}
        \begin{aligned} 
            \|\nabla \widetilde{v}_h^\tau-\nabla\widetilde{w}_h^\tau\|_{p(\cdot),I\times\Sigma}^2&\lesssim (\mathbf{s}(\cdot,\nabla \widetilde{v}_h^\tau)-\mathbf{s}(\cdot,\nabla \widetilde{w}_h^\tau),\nabla \widetilde{v}_h^\tau-\textcolor{black}{\nabla \widetilde{w}_h^\tau})_{I\times \Sigma}\\&\quad\times (1+\rho_{p(\cdot),I\times \Sigma}(\vert\nabla \widetilde{v}_h^\tau\vert+\vert\nabla\widetilde{w}_h^\tau\vert))^{\smash{\frac{2}{p^-}}}\,,
            \end{aligned}
        \end{align}
        while, according to \cite[Lem.\ B.5]{BK2024}, if $\delta>0$, we have that 
        \begin{align}\label{rem:convergence_alg.2}
             \|\nabla \widetilde{v}_h^\tau-\nabla\widetilde{w}_h^\tau\|_{\min\{2,p(\cdot)\},I\times\Sigma}^2&\lesssim (\mathbf{s}(\cdot,\nabla \widetilde{v}_h^\tau)-\mathbf{s}(\cdot,\nabla \widetilde{w}_h^\tau),\nabla \widetilde{v}_h^\tau-\textcolor{black}{\nabla \widetilde{w}_h^\tau})_{I\times \Sigma}\\&\quad\times \bigl((1+\rho_{p(\cdot),I\times \Sigma}(\vert\nabla \widetilde{v}_h^\tau\vert+\vert\nabla\widetilde{w}_h^\tau\vert))^{\smash{\frac{2}{p^-}}} +(\min\{1,\delta\})^{2-p^+}\bigr)\,.\notag
        \end{align}
        On the other hand, by \eqref{lem:scheme:well_posedness.10-bis}, 
        for every $m=1,\ldots,M$, we have that
        \begin{align}\label{rem:convergence_alg.3}
            \|\widetilde{v}_h^\tau(t_m)\|_{\Sigma}^2+\|\widetilde{w}_h^\tau(t_m)\|_{\Sigma}^2
            \leq \textcolor{black}{\tfrac{2c_1}{-\lambda \tau}}\,.
        \end{align}
        In summary, if $p^+\leq 2$ or $\delta>0$, by discrete norm equivalences (\textit{cf}.\ \cite[Lem.\ 12.1]{ErnGuermond2020}),~from \eqref{rem:convergence_alg.1}--\eqref{rem:convergence_alg.3}, there follows the existence 
        of a constant $\mu_h^{\tau}\hspace{-0.15em}>\hspace{-0.15em}0$, deteriorating~as~${\tau\hspace{-0.15em}\to \hspace{-0.15em}0^+\hspace{-0.15em}}$~or~${h\hspace{-0.15em}\to\hspace{-0.15em} 0^+\hspace{-0.15em}}$, such that 
        \begin{align}\label{rem:convergence_alg.4}
             \tfrac{\mu_h^{\tau}}{2}\|\widetilde{v}_h^\tau(L)-\widetilde{w}_h^\tau(L)\|_{2,\Sigma}^2&\leq (\mathbf{s}(\cdot,\nabla \widetilde{v}_h^\tau)-\mathbf{s}(\cdot,\nabla \widetilde{w}_h^\tau),\nabla \widetilde{v}_h^\tau-\textcolor{black}{\nabla \widetilde{w}_h^\tau})_{I\times \Sigma}\,.
        \end{align}
        As a consequence, if we use in \eqref{lem:scheme:well_posedness.4} additionally \eqref{rem:convergence_alg.4}, we arrive at
        \begin{align*} 
        \smash{\{1+\mu_h^{\tau}\}}\|\mathcal{F}_h^\tau(\widetilde{v}_h^0)-\mathcal{F}_h^\tau(\widetilde{w}_h^0)\|_{\Sigma}^2 \leq \|\widetilde{v}_h^0-\widetilde{w}_h^0\|_{\Sigma}^2\,. 
        \end{align*} 
        \textit{i.e.}, $\mathcal{F}_h^\tau\colon V_h\to V_h$ is a $q$-contraction with $q^2=\tfrac{1}{1+\mu_h^{\tau}}\in (0,1)$. Hence, the {Banach} fixed point theorem (\textit{cf}.\ \cite{Banach1882}) can be applied and, for every $k\in \mathbb{N}$, 
        yields the \emph{a priori}~error~estimates 
        \begin{align*}
            \|\widetilde{v}_h^{\tau,k}(L)-\widetilde{v}_h^{\tau,k}(0)\|_{\Sigma}&\leq q^{k-1}\|\widetilde{v}_h^{\tau,1}(L)-\widetilde{v}_h^0\|_{\Sigma}\,,\\
            \|\widetilde{v}_h^{\tau,k}(L)-v_h^{\tau}(L)\|_{\Sigma}&\leq \tfrac{q^k}{1-q}\|\widetilde{v}_h^{\tau,1}(L)-\widetilde{v}_h^0\|_{\Sigma}\,,
        \end{align*}
        which provides some guaranteed orders of convergence of Algorithm \ref{alg:picard-iteration}. 
    \end{itemize}
\end{remark}
 
In addition to the weak stability result \eqref{eq:weak_stability} in Lemma \ref{lem:scheme:well_posedness}, the following 
strong stability result applies, which, subsequently, enables us to establish the (weak) convergence of the discrete (variational) formulation (in the sense of Def.\  \ref{scheme:weak_solution}).

\begin{lemma}[Strong stability]\label{lem:scheme:strong_stability}
    If $\kappa_2=0$ a.e.\ in $\Sigma$ in \textup{(\hyperlink{s.2}{s.2})}, then
    there exists a constant $K_s>0$ such that for every  $h>0$ and $\tau\in (0,1)$, we have that
    \begin{subequations}\label{eq:scheme:strong_stability}
    \begin{align}
        \|\mathrm{d}_{\tau} v_h^\tau\|_{I\times \Sigma}^2&\leq K_s\,,\label{eq:scheme:strong_stability.1}
        \\
         \textup{sup}_{t\in I}{\bigl\{\smash{\rho_{p(\cdot),\Sigma}(\nabla v_h^\tau(t))\bigr\}}}&\leq K_s\,,\label{eq:scheme:strong_stability.2}
         \\
         \textup{sup}_{t\in I}{\smash{\bigl\{\rho_{p'(\cdot),\Sigma}(\mathbf{s}(\cdot,\nabla v_h^\tau(t)))\bigr\}}}&\leq K_s\,,\label{eq:scheme:strong_stability.3}
         \\
        \|\Gamma^\tau_h\|_I^2&\leq K_s \,.\label{eq:scheme:strong_stability.4}
    \end{align}
    \end{subequations}
\end{lemma}

\begin{proof} 

    \emph{ad \eqref{eq:scheme:strong_stability.1}/\eqref{eq:scheme:strong_stability.2}.} {\color{black}In what follows, all discrete time-periodic functions are understood as their
\hspace{-0.15mm}$L$-periodic \hspace{-0.15mm}extensions \hspace{-0.15mm}to \hspace{-0.15mm}$\mathbb{R}$.
     \hspace{-0.15mm}Let \hspace{-0.15mm}$N\hspace{-0.175em}\in\hspace{-0.175em} \{\lceil \frac{M}{2}\rceil+\textcolor{black}{1},\ldots,\lceil \frac{M}{2}\rceil+M\}$~\hspace{-0.15mm}be~\hspace{-0.15mm}fixed,~\hspace{-0.15mm}but~\hspace{-0.15mm}\mbox{arbitrary}~and~\hspace{-0.15mm}\mbox{define} $ I_\tau^N \coloneqq (0,N\tau]$ and $\widehat{I}_\tau^N \coloneqq (\lceil \frac{M}{2}\rceil\tau,N\tau]$.
 Then, if we choose~$\phi_h^\tau=\iota_\tau\mathrm{d}_\tau(v_h^\tau-\alpha^\tau\chi_h)\chi_{I_\tau^{\smash{N}}}$, where  $\iota_\tau\coloneqq\mathrm{I}^0_{\tau}\textup{id}_{\mathbb{R}}$, in \eqref{scheme:weak_solution.3.1}, using that $(\Gamma^\tau_h,\iota_\tau\mathrm{d}_\tau(v_h^\tau-\alpha^\tau\chi_h))_{I_\tau^N\times \Sigma}=0$~as~${(\mathrm{d}_\tau(v_h^\tau-\alpha^\tau\chi_h),1)_\Sigma=0}$~a.e.~in~$I$, we obtain
    \begin{align}\label{lem:scheme:strong_stability.5}
        \begin{aligned}
            \|(\iota_\tau)^{\smash{\frac{1}{2}}}\mathrm{d}_\tau v_h^\tau\|_{I_\tau^{\smash{N}}\times\Sigma}^2+(\mathbf{s}(\cdot,\nabla v_h^\tau),\iota_\tau\mathrm{d}_\tau\nabla v_h^\tau)_{I_\tau^{\smash{N}}\times\Sigma}&= (\mathrm{d}_\tau v_h^\tau,\iota_\tau\mathrm{d}_\tau\alpha^\tau\chi_h)_{I_\tau^{\smash{N}}\times\Sigma}\\&\quad+(\mathbf{s}(\cdot,\nabla v_h^\tau),\iota_\tau\mathrm{d}_\tau\alpha^\tau\nabla \chi_h)_{I_\tau^{\smash{N}}\times\Sigma}\,,
        \end{aligned}
    \end{align}
    where, by the (variable) $\varepsilon$-Young inequality~\eqref{eq:eps-young.2} and $\|\iota_\tau\|_{\infty,I_\tau^{\smash{N}}}\hspace{-0.1em}\leq\hspace{-0.1em} 2L$, for every $\varepsilon\hspace{-0.1em}>\hspace{-0.1em}0$,~we~have~that\vspace{-3.5mm}
    \begin{subequations}\label{lem:scheme:strong_stability.6}
    \begin{align} (\mathrm{d}_\tau v_h^\tau,\iota_\tau\mathrm{d}_\tau\alpha^\tau\chi_h)_{I_\tau^{\smash{N}}\times\Sigma}&\leq L\|\mathrm{d}_\tau\alpha^\tau\chi_h\|_{I_\tau^{\smash{N}}\times\Sigma}^2+ \tfrac{1}{2}\|(\iota_\tau)^{\smash{\frac{1}{2}}}\mathrm{d}_\tau v_h^\tau\|_{I_\tau^{\smash{N}}\times\Sigma}^2\,,\label{lem:scheme:strong_stability.6.1}
    \\
        (\mathbf{s}(\cdot,\nabla v_h^\tau),\iota_\tau\mathrm{d}_\tau\alpha^\tau\nabla \chi_h)_{I_\tau^{\smash{N}}\times\Sigma}&\leq 2L\|\mathrm{d}_\tau\alpha^\tau\|_{1,I_\tau^{\smash{N}}}\bigl(\varepsilon\,\bigl(\textup{sup}_{t\in I}{\bigl\{\rho_{p(\cdot),\Sigma}(\nabla v_h^\tau(t))\bigr\}}+\rho_{p'(\cdot),\Sigma}(\kappa_4)\bigr)\notag\\&\qquad\qquad\qquad\quad+c_\varepsilon\,\rho_{p(\cdot),\Sigma}(\nabla \chi_h)\bigr)\,.
        \label{lem:scheme:strong_stability.6.2}
    \end{align}
    \end{subequations}
    Since 
    $\mathcal{U}(x,\cdot)\coloneqq \mathcal{V}(x,\vert \cdot\vert^2)\colon  \mathbb{R}^{d-1}\to [0,+\infty)$~is convex for a.e.\ $x\in\Sigma$ with $\frac{\mathrm{d}}{\mathrm{d}\mathbf{a}}\mathcal{U}(x,\mathbf{a})=2\mathbf{s}(x,\mathbf{a})$ for a.e.\ $x\in\Sigma$ and all $\mathbf{a}\in \mathbb{R}^{d-1}$  (\textit{cf}.~Lemma \ref{lem:convexity}), we have that 
    \begin{align*} 
        \tfrac{\mathrm{d}_\tau}{2}\{\|\mathcal{V}(\cdot,\vert \nabla v_h^\tau\vert^2)\|_{1,\Sigma}\}\leq 
        (\mathbf{s}(\cdot, \nabla v_h^\tau),\mathrm{d}_\tau \nabla v_h^\tau)_{\Sigma} \quad \text{ a.e.\ in }I\,,
    \end{align*} 
    so that, by discrete integration-by-parts formula \eqref{eq:discrete_integration-by-parts} together with $\mathrm{d}_\tau \iota_\tau=1$ in $I_\tau^{\smash{N}}$,~we~have~that
    \begin{align}\label{lem:scheme:strong_stability.7}
        \begin{aligned} 
        (\mathbf{s}(\cdot, \nabla v_h^\tau),\iota_\tau\mathrm{d}_\tau \nabla v_h^\tau)_{I_\tau^{\smash{N}}\times\Sigma}&\ge \tfrac{1}{2}(\iota_\tau,\mathrm{d}_\tau\{\mathcal{V}(\cdot,\vert \nabla v_h^\tau\vert^2)\})_{I_\tau^{\smash{N}}\times\Sigma}\\&=
        -\tfrac{1}{2}\|\mathcal{V}(\cdot,\vert \nabla v_h^\tau\vert^2)\|_{I_\tau^{\smash{N}}\times\Sigma}
        \\&\quad+\tfrac{N\tau}{2}\|\mathcal{V}(\cdot,\vert \nabla v_h^\tau(t_N)\vert^2)\|_{1,\Sigma}\,.
        \end{aligned}
    \end{align}
    If we combine \eqref{lem:scheme:strong_stability.6.1}, \eqref{lem:scheme:strong_stability.6.2}, and \eqref{lem:scheme:strong_stability.7} in \eqref{lem:scheme:strong_stability.5}, we obtain
    \begin{align}\label{lem:scheme:strong_stability.8}
    \begin{aligned} 
        \tfrac{1}{2}\|(\iota_\tau)^{\smash{\frac{1}{2}}}\mathrm{d}_\tau v_h^\tau\|_{I_\tau^{\smash{N}}\times\Sigma}^2&+\tfrac{N\tau}{2} \|\mathcal{V}(\cdot,\vert \nabla v_h^\tau(t_N)\vert^2)\|_{1,\Sigma}\\&\leq 
        \|\mathcal{V}(\cdot,\vert \nabla v_h^\tau\vert^2)\|_{I_\tau^{\smash{N}}\times\Sigma}+L\|\mathrm{d}_\tau\alpha^\tau\chi_h\|_{I_\tau^{\smash{N}}\times\Sigma}^2\\&\quad+2L\|\mathrm{d}_\tau\alpha^\tau\|_{1,I_\tau^{\smash{N}}}\bigl(\varepsilon\,\bigl(\textup{sup}_{t\in I}{\bigl\{\rho_{p(\cdot),\Sigma}(\nabla v_h^\tau(t))\bigr\}}+\rho_{p'(\cdot),\Sigma}(\kappa_4)\bigr)\\&\qquad\qquad\qquad\quad
        +c_\varepsilon\,\rho_{p(\cdot),\Sigma}(\nabla \chi_h)\bigr)\,.
        \end{aligned}
    \end{align} 
    Due to $N\tau\ge \frac{L}{2}$ and Lemma \ref{lem:anti-derivative}(\hyperlink{lem:anti-derivative.ii}{ii}),(\hyperlink{lem:anti-derivative.iii}{iii}),  we have that
    \begin{subequations} \label{lem:scheme:strong_stability.9} 
    \begin{align}
    \|(\iota_\tau)^{\smash{\frac{1}{2}}}\mathrm{d}_\tau v_h^\tau\|_{I_\tau^{\smash{N}}\times\Sigma}^2&\ge \tfrac{L}{2}\|\mathrm{d}_\tau v_h^\tau\|_{\widehat{I}_\tau^{\smash{N}}\times\Sigma}^2\,,\label{lem:scheme:strong_stability.9.1}\\
        \tfrac{N\tau}{2} \|\mathcal{V}(\cdot,\vert \nabla v_h^\tau(t_N)\vert^2)\|_{1,\Sigma}&\ge \tfrac{L\kappa_1}{2p^+}\rho_{p(\cdot),\Sigma}( \nabla v_h^\tau(t_N))\,,\label{lem:scheme:strong_stability.9.2}\\
        \|\mathcal{V}(\cdot,\vert \nabla v_h^\tau\vert^2)\|_{1,I_\tau^{\smash{N}}\times\Sigma}
        &\le 2\bigl(\tfrac{\kappa_3}{p^-}\rho_{p(\cdot),I_\tau^{\smash{N}}\times\Sigma}( \nabla v_h^\tau)+\|\kappa_4\nabla v_h^\tau\|_{1,I_\tau^{\smash{N}}\times\Sigma}\bigr)\,.\label{lem:scheme:strong_stability.9.3} 
    \end{align}
    \end{subequations}

    For $N=\lceil \frac{M}{2}\rceil+M$, the interval $\widehat I_\tau^N$ contains one full period. Hence, resorting to \eqref{lem:scheme:strong_stability.9.1}--\eqref{lem:scheme:strong_stability.9.3} in \eqref{lem:scheme:strong_stability.8} and, subsequently, forming the maximum with respect to $N\hspace{-0.1em}\in\hspace{-0.1em} \{\lceil \frac{M}{2}\rceil+1,\ldots,\lceil \frac{M}{2}\rceil+M\}$, exploiting the time-periodicity of $v_h^\tau \hspace{-0.1em}\in \hspace{-0.1em}\mathbb{P}^0(\mathcal{I}_\tau^0;V_h)$~(\textit{cf}.~\eqref{scheme:weak_solution.1}), we arrive at
    \begin{align*}
        \tfrac{L}{4}\|\mathrm{d}_\tau v_h^\tau\|_{I\times\Sigma}^2&+\tfrac{L\kappa_1}{2p^+}\textup{sup}_{t\in I}{\bigl\{\rho_{p(\cdot),\Sigma}( \nabla v_h^\tau(t))\bigr\}}\\&\leq 
        2\bigl(\tfrac{\kappa_3}{p^-}\rho_{p(\cdot),I\times\Sigma}( \nabla v_h^\tau)+\|\kappa_4\nabla v_h^\tau\|_{1,I\times\Sigma}\bigr)+L\|\mathrm{d}_\tau\alpha^\tau\chi_h\|_{I\times\Sigma}^2\\&\quad+2L\|\mathrm{d}_\tau\alpha^\tau\|_{1,I}\bigl(\varepsilon\,\bigl(\textup{sup}_{t\in I}{\bigl\{\rho_{p(\cdot),\Sigma}(\nabla v_h^\tau(t))\bigr\}}+\rho_{p'(\cdot),\Sigma}(\kappa_4)\bigr)
        +c_\varepsilon\,\rho_{p(\cdot),\Sigma}(\nabla \chi_h)\bigr)\,.
    \end{align*}
    Thus, choosing $\varepsilon>0$ sufficiently small, we arrive at the   strong stability estimates \eqref{eq:scheme:strong_stability.1}/\eqref{eq:scheme:strong_stability.2}.}
 
    \emph{ad \eqref{eq:scheme:strong_stability.3}.} Using the growth property of $\mathbf{s}\colon \Sigma\times \mathbb{R}^{d-1}\to \mathbb{R}^{d-1}$ (\textit{cf}.\ (\hyperlink{s.3}{s.3})), we find that
    \begin{align*}
        \rho_{p'(\cdot),\Sigma}(\mathbf{s}(\cdot,\nabla v_h^\tau))\leq 2^{(p^-)'-1}\bigl((1+\kappa_3)^{(p^-)'}\rho_{p(\cdot),\Sigma}(\nabla v_h^\tau)+
        \rho_{p'(\cdot),\Sigma}(\kappa_4)\bigr)\quad \text{ a.e.\ in }I\,,
    \end{align*}
    which, by the strong stability estimate \eqref{eq:scheme:strong_stability.2}, implies the strong stability estimate \eqref{eq:scheme:strong_stability.3}.

    \emph{ad \eqref{eq:scheme:strong_stability.4}.} Using the representation formula \eqref{lem:equiv_weak_form_discrete.1.2},
    Hölder's inequality, and the (variable) $\varepsilon$-Young inequality \eqref{eq:eps-young.1} (with $\varepsilon=1$), we find that
    \begin{align*}
        \|\Gamma^\tau_h\|_{I}\leq \|\mathrm{d}_\tau v_h^\tau\|_{I\times \Sigma}\|\chi_h\|_{\Sigma} + 
        L\bigl(\textup{sup}_{t\in I}{\bigl\{\rho_{p(\cdot),\Sigma}(\nabla v_h^\tau(t))\bigr\}}+\rho_{p'(\cdot),\Sigma}(\kappa_4)+c_1\,\rho_{p(\cdot),\Sigma}(\nabla \chi_h)\bigr)\,,
    \end{align*}
    which, \hspace{-0.15mm}by \hspace{-0.15mm}the \hspace{-0.15mm}strong \hspace{-0.15mm}stability \hspace{-0.15mm}estimates \hspace{-0.15mm}\eqref{eq:scheme:strong_stability.1},\eqref{eq:scheme:strong_stability.2}, \hspace{-0.15mm}implies \hspace{-0.15mm}the \hspace{-0.15mm}strong \hspace{-0.15mm}stability~\hspace{-0.15mm}estimate~\hspace{-0.15mm}\eqref{eq:scheme:strong_stability.4}.
\end{proof}

Eventually, we have everything at our disposal to establish the (weak) convergence of discrete variational formulation \eqref{scheme:weak_solution.1}--\eqref{scheme:weak_solution.3.2}  (in the sense of Def.~\ref{scheme:weak_solution}) to the variational formulation \eqref{def:weak_solution-0}--\eqref{def:weak_solution-1.2} (in the sense of Def.~\ref{def:weak_solution}) as $\tau,h\to 0^+$.

\begin{theorem}[Weak convergence]\label{thm:scheme:convergence}
    If $\kappa_2=0$ a.e.\ in $\Sigma$ in \textup{(\hyperlink{s.2}{s.2})}, then there exists a pair
    \begin{align*}
        (v,\Gamma)\in 
    (W^{1,2}(I;L^2(\Sigma))\cap L^\infty(I;\smash{W^{1,p(\cdot)}_0(\Sigma)}))\times L^2(I)\,,
    \end{align*}
    such that
    \begin{subequations}\label{thm:scheme:convergence.1} 
    \begin{alignat}{3}
        v_h^\tau&\overset{\ast}{\rightharpoondown} v&&\quad \text{ in } L^\infty(I;\smash{W^{1,p(\cdot)}_0(\Sigma)})&&\quad (\tau,h\to 0^+)\,,\label{thm:scheme:convergence.1.1}\\
         v_h^\tau(t)&\rightharpoonup v(t)&&\quad \text{ in } \smash{W^{1,p(\cdot)}_0(\Sigma)}\hookrightarrow \hookrightarrow L^2(\Sigma)&&\quad (\tau,h\to 0^+)\,,\quad t\in \{0,L\}\,,\label{thm:scheme:convergence.1.1.2} \\
         \operatorname{\textcolor{black}{div}}\mathbf{s}(\cdot,\nabla v_h^\tau)&\overset{\ast}{\rightharpoondown} \operatorname{\textcolor{black}{div}}\mathbf{s}(\cdot,\nabla v)&&\quad \text{ in } L^\infty(I;\smash{\textcolor{black}{(W^{1,p(\cdot)}_0(\Sigma))^*}})&&\quad (\tau,h\to 0^+)\,,\label{thm:scheme:convergence.1.2}\\
        \mathrm{d}_{\tau}  v_h^\tau&\rightharpoonup \partial_t v&&\quad \text{ in }L^2(I;L^2(\Sigma))&&\quad (\tau,h\to 0^+)\,,\label{thm:scheme:convergence.1.3}\\
        \Gamma^\tau_h&\rightharpoonup \Gamma &&\quad \text{ in }L^2(I)&&\quad (\tau,h\to 0^+)\,.\label{thm:scheme:convergence.1.4}
    \end{alignat}
    \end{subequations}
    In particular, it follows that $(v,\Gamma)\in (W^{1,2}(I;L^2(\Sigma))\cap L^\infty(I;\smash{W^{1,p(\cdot)}_0(\Sigma))})\times L^2(I)$ is the unique (variational) solution of  \eqref{eq:periodic_pLaplace} (in the sense of Def.~\ref{def:weak_solution}). 
\end{theorem}

\begin{proof} We proceed in two main steps:\enlargethispage{2.5mm}

    \emph{1. Solvability:} From the strong stability estimates (\textit{cf}.\ Lemma \ref{lem:scheme:strong_stability}), it follows that there~exist 
    \begin{align*}
        (v,v_L,\widehat{\mathbf{s}},\Gamma)\in 
    L^\infty(I;\smash{W^{1,p(\cdot)}_0(\Sigma)})\cap  W^{1,2}(I;L^2(\Sigma))\times \smash{W^{1,p(\cdot)}_0(\Sigma)}\times L^\infty(I;(L^{\smash{p'(\cdot)}}(\Sigma))^{d-1})\times L^2(I)\,,
    \end{align*}
    such that (up to a not relabelled subsequence)
    \begin{subequations}\label{thm:scheme:convergence.2} 
    \begin{alignat}{3}
        v_h^\tau&\overset{\ast}{\rightharpoondown} v&&\quad \text{ in } L^\infty(I;\smash{W^{1,p(\cdot)}_0(\Sigma)})&&\quad (\tau,h\to 0^+)\,,\label{thm:scheme:convergence.2.1} \\
        v_h^\tau(L)=v_h^\tau(0)&\rightharpoonup v_L&&\quad \text{ in } \smash{W^{1,p(\cdot)}_0(\Sigma)}\hookrightarrow \hookrightarrow L^2(\Sigma)&&\quad (\tau,h\to 0^+)\,,\label{thm:scheme:convergence.2.1.2} \\
         \mathbf{s}(\cdot,\nabla v_h^\tau)&\overset{\ast}{\rightharpoondown}\widehat{\mathbf{s}}&&\quad \text{ in } L^\infty(I;(L^{\smash{p'(\cdot)}}(\Sigma))^{d-1})&&\quad (\tau,h\to 0^+)\,,\label{thm:scheme:convergence.2.2}\\
        \mathrm{d}_{\tau}  v_h^\tau&\rightharpoonup \partial_t v&&\quad \text{ in }L^2(I;L^2(\Sigma))&&\quad (\tau,h\to 0^+)\,,\label{thm:scheme:convergence.2.3}\\
        \Gamma^\tau_h&\rightharpoonup \Gamma &&\quad \text{ in }L^2(I)&&\quad (\tau,h\to 0^+)\,.\label{thm:scheme:convergence.2.4}
    \end{alignat}
    \end{subequations}
    Let $(\phi,\eta) \in (L^1(I;\smash{W^{1,p(\cdot)}_0(\Sigma))}\cap L^2(I;L^2(\Sigma)))\times L^2(I)$ be a fixed, but arbitrary~test~\mbox{function}~\mbox{pair}.\newpage Then, if we
    choose $(\phi_h^\tau,\eta^\tau)\coloneqq (\Pi_{\tau}^0\Pi_h\phi,\Pi_{\tau}^0\eta)\in \mathbb{P}^0(\mathcal{I}_\tau;V_h)\times \mathbb{P}^0(\mathcal{I}_\tau)$ in \eqref{scheme:weak_solution.3}, we have that
    \begin{subequations}\label{thm:scheme:convergence.3} 
    \begin{align}\label{thm:scheme:convergence.3.1} 
        (\mathrm{d}_\tau v_h^\tau,\phi_h^\tau)_{I\times \Sigma}+(\mathbf{s}(\cdot,\nabla v_h^\tau),\nabla \phi_h^\tau)_{I\times \Sigma}+(\Gamma^\tau_h,\phi_h^\tau)_{I\times \Sigma}&=0\,,\\
        (v_h^\tau,\eta^\tau)_{I\times \Sigma}&=(\alpha^\tau,\eta^\tau)_{I}\,,\label{thm:scheme:convergence.3.2} 
    \end{align}
    \end{subequations}
    and, by the approximation properties of $\Pi_h$ (\textit{cf}.\ \cite[Lem.\ 3.5]{BreitDieningSchwarzacher2015}) and  $\Pi_\tau^0$ (\textit{cf}.\ \cite[Rem.\ 8.15]{Roubicek2013}),~at~least\vspace{-4mm}
    \begin{subequations}\label{thm:scheme:convergence.4} 
    \begin{alignat}{3}\label{thm:scheme:convergence.4.1} 
        \phi_h^\tau&\to \phi &&\quad \text{ in }L^2(I;L^2(\Sigma))&&\quad (\tau,h\to 0^+)\,, \\[-0.5mm]
        \phi_h^\tau&\to \phi&& \quad \text{ in }L^1(I; \smash{W^{1,p(\cdot)}_0(\Sigma)})&&\quad (\tau,h\to 0^+)\,,\label{thm:scheme:convergence.4.2}\\
        \eta^\tau &\to \eta &&\quad \text{ in }L^2(I)&&\quad (\tau \to 0^+)\,.\label{thm:scheme:convergence.4.3} 
    \end{alignat}
    \end{subequations}
    As a result, 
    if we pass for $\tau,h\to 0^+$ in \eqref{thm:scheme:convergence.3}, using the convergences \eqref{thm:scheme:convergence.2} and \eqref{thm:scheme:convergence.4}~in~doing~so, for every $(\phi,\eta) \in (L^1(I;W^{1,p(\cdot)}_0(\Sigma))\cap L^2(I;L^2(\Sigma)))\times L^2(I)$, we arrive at
    \begin{subequations}\label{thm:scheme:convergence.4_2} 
    \begin{align}\label{thm:scheme:convergence.4_2.1} 
        (\partial_t v,\phi)_{I\times \Sigma}+(\widehat{\mathbf{s}},\nabla \phi)_{I\times \Sigma}+(\Gamma,\phi)_{I\times \Sigma}&=0\,,\\
        (v,\eta)_{I\times \Sigma}&=(\alpha,\eta)_{I}\,.\label{thm:scheme:convergence.4_2.2} 
    \end{align}
    \end{subequations}

    \emph{1.1. Time-periodicity of $v$.} 
    If we use the discrete integration-by-parts formula \eqref{eq:discrete_integration-by-parts}~in~\eqref{thm:scheme:convergence.3.1}, for $\phi_h^\tau\coloneqq\mathrm{I}_{\tau}^0\Pi_h\phi\in \mathbb{P}^0(\mathcal{I}_\tau^0;V_h)$, where $\phi\in    W^{1,2}(I;W^{1,p(\cdot)}_0(\Sigma))$ is arbitrary, using~the~time-period-icity property of $v_h^\tau\in \mathbb{P}^0(\mathcal{I}_\tau^0;V_h)$ (\textit{cf}.\ \eqref{scheme:weak_solution.1}), we find that
    \begin{align}\label{thm:scheme:convergence.5} 
        \hspace{-1mm}-( \mathrm{T}_{\tau}^-v_h^\tau,\mathrm{d}_\tau\phi_h^\tau)_{I\times \Sigma}+(\mathbf{s}(\cdot,\nabla v_h^\tau),\nabla \phi_h^\tau)_{I\times \Sigma}+(v_h^\tau(L),\phi_h^\tau(L)-\phi_h^\tau(0))_{\Sigma}+(\Gamma^\tau_h,\phi_h^\tau)_{I\times \Sigma}=0\,,
    \end{align}
    and, resorting to the approximation properties of $\Pi_h$ (\textit{cf}.\ \cite[Lem.\ 3.5]{BreitDieningSchwarzacher2015}) and  $\mathrm{I}_\tau^0$ (\textit{cf}.~\cite[Lem.~8.7]{Roubicek2013}) as well as that $\mathrm{d}_\tau \phi_h^\tau=\Pi_h(\mathrm{d}_\tau \phi)$ a.e.\ in $L^2(I;L^2(\Sigma))$ together with $\phi\in W^{1,2}(I;L^2(\Sigma))$, at least\enlargethispage{4mm}
    \begin{subequations}\label{thm:scheme:convergence.6} 
    \begin{alignat}{3}\label{thm:scheme:convergence.6.1} 
        \phi_h^\tau&\to \phi&& \quad \text{ in }L^1(I; W^{1,p(\cdot)}_0(\Sigma))&&\quad (\tau,h\to 0^+)\,,\\
        \phi_h^\tau&\to \phi &&\quad \text{ in }L^2(I;L^2(\Sigma))&&\quad (\tau,h\to 0^+)\,,\label{thm:scheme:convergence.6.2}\\ 
        \mathrm{d}_\tau \phi_h^\tau&\to \partial_t\phi &&\quad \text{ in }L^2(I;L^2(\Sigma))&&\quad (\tau,h\to 0^+)\,, \label{thm:scheme:convergence.6.3}\\[-0.5mm]
        \phi_h^\tau(t)&\to \phi(t) &&\quad \text{ in }W^{1,p(\cdot)}_0(\Sigma)\hookrightarrow L^2(\Sigma)&&\quad (\tau,h\to 0^+)\,,\quad \text{ for }t\in\{0,L\}\,, \label{thm:scheme:convergence.6.4}
    \end{alignat}
    \end{subequations}
    Then, by passing for $h,\tau\to 0^+$ in \eqref{thm:scheme:convergence.5}, using \eqref{thm:scheme:convergence.6.1}--\eqref{thm:scheme:convergence.6.4} and ${\mathrm{T}_{\tau}^-v_h^\tau-v_h^\tau=-\tau\mathrm d_\tau v_h^\tau\to0}$ in $L^2(I;L^2(\Sigma))$ $(h,\tau\to 0^+)$ doing so, we obtain
    \begin{align*}
        -( v,\partial_t \phi)_{I\times \Sigma}+(\widehat{\mathbf{s}},\nabla \phi)_{I\times \Sigma}+(v_L,\phi(L)-\phi(0))_{\Sigma}+(\Gamma,\phi)_{I\times \Sigma}=0\,,
    \end{align*}
    so that, by the integration-by-parts formula in $W^{1,2}(I;L^2(\Sigma))$ (\textit{cf}.\ \cite[Prop.\ 2.5.2]{Droniou2001}), we arrive at
    \begin{align}\label{thm:scheme:convergence.7} 
       (v_L,\phi(L)-\phi(0))_{\Sigma}=(v(L),\phi(L))_{\Sigma}-(v(0),\phi(0))_{\Sigma}\,.
    \end{align}
    If \hspace{-0.15mm}we \hspace{-0.15mm}choose \hspace{-0.15mm}for \hspace{-0.15mm}arbitrary \hspace{-0.15mm}$w\hspace{-0.175em}\in\hspace{-0.175em} W^{1,p(\cdot)}_0(\Sigma)$ \hspace{-0.15mm}as \hspace{-0.15mm}test \hspace{-0.15mm}function \hspace{-0.15mm}in \hspace{-0.15mm}\eqref{thm:scheme:convergence.7}, \hspace{-0.15mm}a \hspace{-0.15mm}function
    \hspace{-0.15mm}$\phi\hspace{-0.175em}\in\hspace{-0.175em} W^{1,2}(I;W^{1,p(\cdot)}_0(\Sigma))$, which either satisfies $\phi(L)=w$ and  $\phi(0)=0$ a.e.\ in $\Sigma$ or 
     $\phi(L)=0$ and  $\phi(0)=w$~a.e.~in~$\Sigma$, we infer  that 
    \begin{align}\label{thm:scheme:convergence.8} 
      v(L)= v_L = v(0)\quad \text{ a.e.\ in }\Sigma\,.
    \end{align}
    In other words, the weak limit $v\in L^\infty(I;W^{1,p(\cdot)}_0(\Sigma))\cap  W^{1,2}(I;L^2(\Sigma))$ satisfies the time-periodicity condition \eqref{def:weak_solution-0}.

    \emph{1.2. Identification of $\widehat{\mathbf{s}}$ and $\mathbf{s}(\cdot,\nabla v)$.} 
    By the monotonicity property of $\mathbf{s}\colon\Sigma\times\mathbb{R}^{d-1}\to\mathbb{R}^{d-1}$ (\textit{cf}.~(\hyperlink{s.4}{s.4})), for every $\varphi\in L^\infty(I;W^{1,p(\cdot)}_0(\Sigma))$, we~have~that
    \begin{align}\label{thm:scheme:convergence.9} 
        (\mathbf{s}(\cdot,\nabla v_h^\tau)-\mathbf{s}(\cdot,\nabla \varphi),\nabla v_h^\tau-\nabla \varphi)_{I\times \Sigma}\ge 0\,.
    \end{align}
    Moreover, due to the time-periodicity properties \eqref{scheme:weak_solution.1} and \eqref{thm:scheme:convergence.8}, we have that 
    \begin{subequations}\label{thm:scheme:convergence.9_2}  
    \begin{align}\label{thm:scheme:convergence.9_2.1}  
        (\mathrm{d}_\tau v_h^\tau,v_h^\tau)_{I\times \Sigma}&\ge \tfrac{1}{2}\|v_h^\tau(L)\|_{\Sigma}^2-\tfrac{1}{2}\|v_h^\tau(0)\|_{\Sigma}^2=0\,,\\
        (\partial_t v,v)_{I\times \Sigma}&= \tfrac{1}{2}\|v(L)\|_{\Sigma}^2-\tfrac{1}{2}\|v(0)\|_{\Sigma}^2=0\,,\label{thm:scheme:convergence.9_2.2}  
    \end{align}
    \end{subequations}
    so that, by \eqref{scheme:weak_solution.3.1},\eqref{scheme:weak_solution.3.2} with $(\varphi_h^\tau,\eta^\tau)=(v_h^\tau,\Gamma^\tau_h) \in \mathbb{P}^0(\mathcal{I}_\tau;V_h)\times\mathbb{P}^0(\mathcal{I}_\tau)$, \eqref{thm:scheme:convergence.2.4}~together with \eqref{eq:alpha_approx_esti.1}, and  \eqref{thm:scheme:convergence.4_2.1},\eqref{thm:scheme:convergence.4_2.2} with $(\varphi,\eta) =(v,\Gamma)  \in {(L^1(I;W^{1,p(\cdot)}_0(\Sigma))\cap L^2(I;L^2(\Sigma)))\times L^2(I)}$, we observe that\vspace{-0.5mm}
    \begin{align}\label{thm:scheme:convergence.10} 
        \begin{aligned} 
        \limsup_{\tau,h\to 0^+}{\bigl\{(\mathbf{s}(\cdot,\nabla v_h^\tau),\nabla v_h^\tau)_{I\times \Sigma}\bigr\}}&\overset{\eqref{scheme:weak_solution.3.1}}{=}\limsup_{\tau,h\to 0^+}{\bigl\{-(\Gamma^\tau_h ,v_h^\tau)_{I\times \Sigma}-(\mathrm{d}_\tau v_h^\tau,v_h^\tau)_{I\times \Sigma}\bigr\}}
        \\[-1mm]&\overset{\eqref{thm:scheme:convergence.9_2.1} }{\leq} \limsup_{\tau,h\to 0^+}{\bigl\{-(\Gamma^\tau_h ,v_h^\tau)_{I\times \Sigma}\bigr\}}
        \\[-1mm]&\overset{\eqref{scheme:weak_solution.3.2}}{=} \limsup_{\tau,h\to 0^+}{\bigl\{-(\Gamma^\tau_h ,\alpha^\tau)_I\bigr\}}
        \\[-1mm]&\hspace{-5.25mm}\overset{\eqref{thm:scheme:convergence.2.4}+\eqref{eq:alpha_approx_esti.1}}{=}-(\Gamma,\alpha)_I
        \\[-1mm]&\overset{\eqref{thm:scheme:convergence.4_2.2}}{=}-(\Gamma,v)_{I\times \Sigma}
        \\[-1mm]&\overset{\eqref{thm:scheme:convergence.9_2.2} }{=}-(\Gamma,v)_{I\times \Sigma}-(\partial_t v,v)_{I\times \Sigma}
        \\[-1mm]&\overset{\eqref{thm:scheme:convergence.4_2.1}}{=}(\widehat{\mathbf{s}},\nabla v)_{I\times \Sigma}\,.
        \end{aligned}
    \end{align}
    As a consequence, taking the limit superior with respect to $\tau ,h\hspace{-0.1em}\to\hspace{-0.1em} 0^+$ in \eqref{thm:scheme:convergence.9}, using  \eqref{thm:scheme:convergence.2.1},\eqref{thm:scheme:convergence.2.2} and \eqref{thm:scheme:convergence.10} in doing so, for every $\varphi\in L^\infty(I;W^{1,p(\cdot)}_0(\Sigma))$, we find that
    \begin{align*}
        (\widehat{\mathbf{s}} -\mathbf{s}(\cdot,\nabla \varphi),\nabla v-\nabla \varphi)_{I\times \Sigma}\ge 0\,.
    \end{align*}
    Choosing $\varphi =v\pm r \widetilde{\varphi}\in L^\infty(I;W^{1,p(\cdot)}_0(\Sigma)) $, where $r>0$ and $\widetilde{\varphi}\in L^\infty(I;W^{1,p(\cdot)}_0(\Sigma))$ are arbitrary, we obtain\vspace{-0.5mm}
    \begin{align*}
        \pm r (\widehat{\mathbf{s}}-\mathbf{s}(\cdot,\nabla v \pm r \nabla\widetilde{\varphi}),\nabla \widetilde{\varphi})_{I\times \Sigma}\ge 0\,.
    \end{align*}
    Dividing by $r>0$ and passing  for $r\to 0^+$, for every $\widetilde{\varphi}\in L^\infty(I;W^{1,p(\cdot)}_0(\Sigma))$, we arrive at
    \begin{align*}
        (\widehat{\mathbf{s}}-\mathbf{s}(\cdot,\nabla v),\nabla \widetilde{\varphi})_{I\times \Sigma}=0\,,
    \end{align*}
    which, inserted in \eqref{thm:scheme:convergence.4_2.1}, yields  that $(v,\Gamma)\in (W^{1,2}(I;L^2(\Sigma))\cap L^\infty(I;W^{1,p(\cdot)}_0(\Sigma)))\times L^2(I)$ is the variational solution of problem~\eqref{eq:periodic_pLaplace} (in the sense of Def.~\ref{def:weak_solution}). Moreover,~as~soon~as we can prove that $(v,\Gamma)\in (W^{1,2}(I;L^2(\Sigma))\cap L^\infty(I;W^{1,p(\cdot)}_0(\Sigma)))\times L^2(I)$ is unique, the standard~subse\-quence convergence principle (\textit{cf}.\ \cite[Prop.\ 21.23(i)]{ZeidlerIIA}) yields that the convergences \eqref{thm:scheme:convergence.2.1}--\eqref{thm:scheme:convergence.2.4} hold for the entire sequence and not only a subsequence.\enlargethispage{3.5mm}

    \emph{2. Uniqueness:} Let $(v,\Gamma),(w ,\Lambda )\in (W^{1,2}(I;L^2(\Sigma))\cap L^\infty(I;W^{1,p(\cdot)}_0(\Sigma)))\times L^2(I)$  (variational) solutions  of \eqref{def:weak_solution-0}--\eqref{def:weak_solution-1.2}.  Then, for every $(\phi,\eta) \in  (L^1(I;W^{1,p(\cdot)}_0(\Sigma))\cap L^2(I;L^2(\Sigma)))\times L^2(I)$,~there holds\vspace{-0.5mm}
    \begin{subequations}\label{thm:scheme:convergence.11} 
    \begin{align}
        (\partial_t (v-w),\phi)_{I\times\Sigma}+(\mathbf{s}(\cdot,\nabla v)-\mathbf{s}(\cdot,\nabla w),\nabla \phi)_{I\times\Sigma}+(\Gamma -\Lambda,\phi)_{I\times\Sigma}&=0\,,\label{thm:scheme:convergence.11.1}\\
        (v-w,\eta)_{I\times\Sigma}&=0\,.\label{thm:scheme:convergence.11.2} 
    \end{align}
    \end{subequations}
    Choosing $\varphi =v-w\in  L^1(I;W^{1,p(\cdot)}_0(\Sigma))\cap L^2(I;L^2(\Sigma))$ in \eqref{thm:scheme:convergence.11.1}, due to $(\Gamma-\Lambda, v-w)_{I\times \Sigma}=0$ (\textit{cf}.\ \eqref{thm:scheme:convergence.11.2}), also using integration-by-parts in time, we  obtain
    \begin{align}
    \label{thm:scheme:convergence.12} 
    \begin{aligned} 
        [\tfrac{1}{2}\|v(t)-w(t)\|_{\Sigma}^2]_{t=0}^{t=L}
        +(\mathbf{s}(\cdot,\nabla v)-\mathbf{s}(\cdot,\nabla w),\nabla v-\nabla w)_{I\times\Sigma}=0\,.
        \end{aligned}
    \end{align} 
    The time-periodicity of $v,w\in W^{1,2}(I;L^2(\Sigma))\cap L^\infty(I;W^{1,p(\cdot)}_0(\Sigma))$ (\textit{cf}.\ \eqref{def:weak_solution-0}) implies that\enlargethispage{2.5mm}
    \begin{align*}
        [\tfrac{1}{2}\|v(t)-w(t)\|_{\Sigma}^2]_{t=0}^{t=L}=0\,,
    \end{align*}
    so that from \eqref{thm:scheme:convergence.12}, we infer that
    \begin{align*}
        (\mathbf{s}(\cdot,\nabla v)-\mathbf{s}(\cdot,\nabla w),\nabla v-\nabla w)_{I\times\Sigma}= 0\,,
    \end{align*}
    which, by the strict monotonicity of $\mathbf{s}(x,\cdot)\colon \mathbb{R}^{d-1}\to \mathbb{R}^{d-1}$ for a.e.\ $x\in \Sigma$ (\textit{cf}.\ (\hyperlink{s.4}{s.4})), implies that $\nabla v=\nabla w$ a.e.\ in $I\times \Sigma$ and, 
    by Poincar\'e's inequality, that $v= w$ a.e.\ in  $I\times \Sigma$. Then, subtracting \eqref{def:weak_solution-1.1}, each
    with $\varphi=\eta\chi$, where $\eta\in L^2(I)$ is arbitrary, 
    for the two solutions, using $v=w$ a.e.\ in $I\times\Sigma$ and   
$(\chi,1)_\Sigma=1$, we obtain $(\Gamma-\Lambda,\eta)_I=0$ for all $\eta\in L^2(I)$ and, thus, $\Gamma=\Lambda$ a.e.\ in $I$. 
\end{proof}\newpage

\section{The assigned pressure drop problem}\label{sec:pressure}\enlargethispage{2.5mm}

\hspace{5mm}The result of the existence of variational solutions with assigned pressure drop is significantly simpler than the previous one, since it concerns a direct parabolic problem. More precisely, in this section, we consider a $(d-1)$-dimensional problem with a single unknown $v\colon I\times \Sigma\to \mathbb{R}$~such~that\vspace{-4.5mm}
\begin{subequations}\label{eq:periodic_pLaplace_2}
\begin{alignat}{2}\label{eq:periodic_pLaplace_2.1}
    \partial_tv-\textup{div}\,\mathbf{s}(\cdot,\nabla v)+\Gamma&=0&&\quad \text{ in }I\times \Sigma\,,\\ 
    v&=0 &&\quad\text{ on } I\times\partial\Sigma\,,\label{eq:periodic_pLaplace_2.2}\\
    v(0)&=v(L)&&\quad\text{ in }\Sigma\,,\label{eq:periodic_pLaplace_2.3}
\end{alignat}
\end{subequations}
where $\Gamma\in L^2(I)$ is an assigned $L$-time-periodic (in the sense of \eqref{rem:periodicity.2}) \textit{pressure drop}.

For the $(d-1)$-dimensional problem 
\eqref{eq:periodic_pLaplace_2}, we introduce the following~variational~formulation.

\begin{definition}[Variational formulation of \eqref{eq:periodic_pLaplace_2}]\label{def:periodic_pLaplace_2_weak_solution}
    Given $\Gamma\in L^2(I)$, a function
    \begin{align*}
        v\in 
   L^\infty(I;W^{1,p(\cdot)}_0(\Sigma))\cap W^{1,2}(I;L^2(\Sigma))\,,
    \end{align*}
    is called \emph{(variational) solution} of \eqref{eq:periodic_pLaplace_2} if 
    \begin{align}\label{def:periodic_pLaplace_weak_solution_2.1}
        v(0)=v(L)\quad \text{ a.e.\ in }\Sigma\,,
    \end{align}
    and for every $\phi\in L^1(I;W^{1,p(\cdot)}_0(\Sigma))\cap L^2(I;L^2(\Sigma))$, there holds 
    \begin{align}
        (\partial_t v,\phi)_{I\times 
        \Sigma}+(\mathbf{s}(\cdot,\nabla v), \nabla \phi)_{I\times\Sigma}+(\Gamma,\phi)_{I\times\Sigma}&=0\,.\label{def:periodic_pLaplace_weak_solution_2.2}
    \end{align}
\end{definition}

 Once again, we intend to prove the well-posedness of the variational formulation (in the sense of Def.\  \ref{def:periodic_pLaplace_2_weak_solution}) by means of a fully-discrete finite-difference/-element discretization.~To~this~end, 
 we approximate $\Gamma\in L^2(I)$ with temporally piece-wise constant  functions 
 \begin{align}\label{def:gamma_approx}
     \Gamma^\tau\coloneqq \Pi_\tau^0\Gamma\in \mathbb{P}^0(\mathcal{I}_\tau)\,,\quad \tau>0\,,
 \end{align}
 which, \hspace{-0.15mm}by \hspace{-0.15mm}the \hspace{-0.15mm}$L^2(I)$-stability \hspace{-0.15mm}(with \hspace{-0.15mm}constant \hspace{-0.15mm}1) \hspace{-0.15mm}of \hspace{-0.15mm}$\Pi_\tau^0$ \hspace{-0.15mm}and \hspace{-0.15mm}a \hspace{-0.15mm}local \hspace{-0.15mm}inverse \hspace{-0.15mm}estimate~\hspace{-0.15mm}(\textit{cf}.~\hspace{-0.15mm}\mbox{\cite[\hspace{-0.15mm}Lem.~\hspace{-0.5mm}12.1]{ErnGuermond2020}}), satisfy\vspace{-1mm}
 \begin{subequations}\label{eq:gamma_approx_esti}
 \begin{align}\label{eq:gamma_approx_esti.1}
     \|\Gamma^\tau\|_{I}&\leq  \|\Gamma\|_{I}\,,
     \\
     \|\Gamma^\tau\|_{\infty,I}&\lesssim \smash{\tfrac{1}{\sqrt{\tau}}}\|\Gamma\|_{I}\,,
     \intertext{and, by the approximation properties of $\Pi_\tau^0$ (\textit{cf}.\ \cite[Rem.\ 8.15]{Roubicek2013}),}
     \Gamma^\tau\to \Gamma\quad \text{ in }&L^2(I)\quad (\tau\to 0^+)\,.
 \end{align}
 \end{subequations} 

Then, we consider the following fully-discrete finite-difference/-element discretization.
 \begin{definition}[Discrete variational formulation]\label{scheme:periodic_pLaplace_weak_solution}
    For a finite number of time steps $M\in \mathbb{N}$ and step size $\tau\coloneqq \frac{L}{M}>0$, a function 
    \begin{align*}
        v_h^\tau\in \smash{\mathbb{P}^0(\mathcal{I}_\tau^0;V_h)}\,,
    \end{align*}
     is called \emph{discrete (variational) solution} of \eqref{eq:periodic_pLaplace_2} if  
    \begin{align}\label{scheme:periodic_pLaplace_weak_solution.1}
v_h^\tau(0)=v_h^\tau(L)\quad \text{ a.e.\ in }\Sigma\,,
    \end{align}
    and for every $\phi_h^\tau\in \mathbb{P}^0(\mathcal{I}_\tau;V_h)$, there holds
    \begin{align}
        (\mathrm{d}_{\tau} v_h^\tau,\phi_h^\tau)_{I\times\Sigma}+(\mathbf{s}(\cdot,\nabla v_h^\tau),\nabla \phi_h^\tau)_{I\times\Sigma}+(\Gamma^\tau,\phi_h^\tau)_{I\times\Sigma}=0\,.\label{scheme:periodic_pLaplace_weak_solution.2}
    \end{align} 
\end{definition}

To begin with, let us prove the well-posedness  (\textit{i.e.}, its unique solvability)  and weak stability  (\textit{i.e.}, \textit{a priori} bounds in the energy norm) of  the discrete variational formulation (\textit{cf}.\ Definition~\ref{scheme:periodic_pLaplace_weak_solution}). 

\begin{lemma}[Well-posedness and weak stability]\label{lem:scheme:well_posedness-h}
    There exists a unique discrete (variational) solution
    $v_h^\tau\in \mathbb{P}^0(\mathcal{I}_\tau^0;V_h)$  (in the sense of Def.\  \ref{scheme:periodic_pLaplace_weak_solution}).  Moreover, there exists~a~constant~$K_w>0$ such that for every $h>0$ and $\tau\in (0,1)$, we have that
    \begin{align}\label{scheme:periodic_pLaplace_weak_solution-1}
        \rho_{p(\cdot),I\times \Sigma}(\nabla v_h^\tau)\leq K_w\,.
    \end{align}
\end{lemma}

\begin{proof} 
Similar to the proof of Lemma \ref{lem:scheme:well_posedness},~we~establish~the existence of a constant $\widetilde{c}_1>0$~such~that the fixed point  operator 
    $\widehat{\mathcal{F}}_h^\tau\colon \hspace{-0.1em}\widehat{B}_h^\tau\hspace{-0.1em}\to\hspace{-0.1em} \widehat{B}_h^\tau$, where $\widehat{B}_h^\tau\hspace{-0.1em}\coloneqq\hspace{-0.1em}\{\varphi_h\hspace{-0.1em}\in\hspace{-0.1em}  V_h\mid \|\varphi_h\|_{\Sigma}^2\hspace{-0.1em}\leq\hspace{-0.1em}\tfrac{\widetilde{c}_1}{-\lambda \tau}\}$, for~every~${\widetilde{v}_h^0\hspace{-0.1em}\in\hspace{-0.1em} \widehat{B}_h^\tau}$ defined by $ \widehat{\mathcal{F}}_h^\tau(\widetilde{v}_h^0)\coloneqq \widetilde{v}_h^\tau(L)$ in $V_h$, 
    where $\widetilde{v}_h^\tau\in  \mathbb{P}^0(\mathcal{I}_\tau^0;V_h)$ is such that\vspace{-0.5mm}
\begin{align}\label{lem:scheme:well_posedness_2.0} 
  \widetilde{v}_h^\tau(0)=  \widetilde{v}_h^0\quad \text{ a.e.\ in }\Sigma\,,
\end{align}
and  for every $\phi_h^\tau\in  \mathbb{P}^0(\mathcal{I}_\tau ;V_h)$, there holds\vspace{-0.5mm}
    \begin{align}\label{lem:scheme:well_posedness_2.1} 
        (\mathrm{d}_{\tau} \widetilde{v}_h^\tau,\phi_h^\tau)_{I\times\Sigma}+(\mathbf{s}(\cdot,\nabla \widetilde{v}_h^\tau),\nabla \phi_h^\tau)_{I\times\Sigma}+(\Gamma^\tau,\phi_h^\tau)_{I\times\Sigma}&=0\,, 
    \end{align} 
    meets the assumptions of the {Edelstein} fixed point theorem (\textit{cf}.\ Theorem \ref{thm:edelstein}). 
   The weak stability estimate \eqref{scheme:periodic_pLaplace_weak_solution-1} follows along the lines of the proof of Lemma \ref{lem:scheme:well_posedness} up to obvious simplifications.
\end{proof}

The outlined constructive proof of Lemma \ref{lem:scheme:well_posedness-h}, again,  can be summarised as~an~\mbox{algorithm},~which may be used to iteratively compute the  discrete (variational) solution (in the sense of~\mbox{Definition}~\ref{scheme:periodic_pLaplace_weak_solution}).\vspace{-0.5mm}

\begin{algorithm}[H]
\caption{Picard iteration for approximating the discrete solution of \eqref{scheme:periodic_pLaplace_weak_solution.1}--\eqref{scheme:periodic_pLaplace_weak_solution.2}}
\label{alg:picard-iteration.2}
\begin{algorithmic}[1]
\Require initial guess $\widetilde{v}^0_h\hspace{-0.175em}\in\hspace{-0.175em} \widehat{B}_h^\tau$, tolerance $\texttt{tol}_{\textup{stop}}\hspace{-0.175em} > \hspace{-0.175em}0$, maximum iterations $\texttt{K}_{\textup{max}}\hspace{-0.175em}\in \hspace{-0.175em} \mathbb{N}$,~norm~${\|\hspace{-0.175em}\cdot\hspace{-0.175em}\|_{V_h}}$
\Ensure approximate solution $v_h^\tau\in \mathbb{P}^0(\mathcal{I}_\tau ^0;V_h)$ solving \eqref{scheme:periodic_pLaplace_weak_solution.1}--\eqref{scheme:periodic_pLaplace_weak_solution.2}
\State Set iteration counter: $k \coloneqq 0$
\State Set initial residual: $\mathtt{res}_h^{\tau,0} \coloneqq \textcolor{black}{+\infty}$
\While{$\mathtt{res}_h^{\tau,k} > \texttt{tol}_{\textup{stop}}$ and $k < \texttt{K}_{\textup{max}}$}
  \State Compute $\widetilde{v}_h^{\tau,k+1}\in\mathbb{P}^0(\mathcal{I}_\tau ^0;V_h)$ solving \eqref{lem:scheme:well_posedness_2.0}--\eqref{lem:scheme:well_posedness_2.1} \textcolor{black}{with $\widetilde{v}_h^{\tau,k+1}(0)=\widetilde{v}_h^0$}
  \State Compute the residual: $\mathtt{res}_h^{\tau,k+1} \coloneqq \|\widetilde{v}_h^{\tau,k+1}(L)-\widetilde{v}_h^{\tau,k+1}(0)\|_{V_h}$
   \State Update initial value: $\widetilde{v}_h^0 \gets \widetilde{v}_h^{\tau,k+1}(L)$
  \State Update iteration: $k \gets k+1$
\EndWhile
\State \textbf{return} $ v_h^\tau\coloneqq\widetilde{v}_h^{\tau,k}\in \mathbb{P}^0(\mathcal{I}_\tau ^0;V_h)$
\end{algorithmic}
\end{algorithm} 

By analogy with Lemma~\ref{lem:scheme:strong_stability}, we have the following strong stability result for 
discrete~(variational) solutions (in the sense of Def.\  \ref{scheme:periodic_pLaplace_weak_solution}).\vspace{-0.5mm} 

\begin{lemma}[Strong stability]\label{lem:scheme:periodic_pLaplace_strong_stability}
    If $\kappa_2=0$ a.e.\ in $\Sigma$ in \textup{(\hyperlink{s.2}{s.2})}, then
    there exists a constant $K_s>0$ such that for every  $h>0$  and $\tau\in (0,1)$, we have that\enlargethispage{6mm}
    \begin{subequations}\label{eq:scheme:periodic_pLaplace_strong_stability}
    \begin{align}
        \|\mathrm{d}_{\tau} v_h^\tau\|_{I\times \Sigma}^2&\leq K_s\,,\label{eq:scheme:periodic_pLaplace_strong_stability.1}
        \\
         \textup{sup}_{t\in I}{\bigl\{\smash{\rho_{p(\cdot),\Sigma}(\nabla v_h^\tau(t))\bigr\}}}&\leq K_s\,,\label{eq:scheme:periodic_pLaplace_strong_stability.2}
         \\
         \textup{sup}_{t\in I}{\smash{\bigl\{\rho_{p'(\cdot),\Sigma}(\mathbf{s}(\cdot,\nabla v_h^\tau(t)))\bigr\}}}&\leq K_s\,.\label{eq:scheme:periodic_pLaplace_strong_stability.3}
    \end{align}
    \end{subequations}
\end{lemma}


 By means of Lemma \ref{lem:scheme:well_posedness-h} and Lemma \ref{lem:scheme:periodic_pLaplace_strong_stability}, we can prove the weak  convergence of 
  discrete (variational) solutions (in the sense of Def.\  \ref{scheme:periodic_pLaplace_weak_solution}).\vspace{-0.5mm}

\begin{theorem}[Weak convergence]\label{thm:scheme:periodic_pLaplace_convergence}
    \textcolor{black}{If $\kappa_2=0$ a.e.\ in $\Sigma$ in \textup{(\hyperlink{s.2}{s.2})}, then there exists a function}
    \begin{align*}
        \smash{v\in 
    W^{1,2}(I;L^2(\Sigma))\cap L^\infty(I;W^{1,p(\cdot)}_0(\Sigma))}\,,
    \end{align*}
    such that\vspace{-0.5mm}
    \begin{subequations}\label{thm:scheme:periodic_pLaplace_convergence.1} 
    \begin{alignat}{3}
        v_h^\tau&\overset{\ast}{\rightharpoondown} v&&\quad \text{ in } L^\infty(I;W^{1,p(\cdot)}_0(\Sigma))&&\quad (\tau,h\to 0^+)\,,\label{thm:scheme:periodic_pLaplace_convergence.1.1}\\
         \operatorname{\textcolor{black}{div}}\mathbf{s}(\cdot,\nabla v_h^\tau)&\overset{\ast}{\rightharpoondown} \operatorname{\textcolor{black}{div}}\mathbf{s}(\cdot,\nabla v)&&\quad \text{ in } L^\infty(I;\smash{\textcolor{black}{(W^{1,p(\cdot)}_0(\Sigma))^*}})&&\quad (\tau,h\to 0^+)\,,\label{thm:scheme:periodic_pLaplace_convergence.1.2}\\
        \mathrm{d}_{\tau}  v_h^\tau&\rightharpoonup \partial_t v&&\quad \text{ in }L^2(I;L^2(\Sigma))&&\quad (\tau,h\to 0^+)\,.\label{thm:scheme:convergence.1.3-h}
    \end{alignat}
    \end{subequations}
    In particular, it follows that $v\in W^{1,2}(I;L^2(\Sigma))\cap L^\infty(I;W^{1,p(\cdot)}_0(\Sigma))$ is the unique (variational) solution of  \eqref{eq:periodic_pLaplace_2} (in the sense of Def.~\ref{def:periodic_pLaplace_2_weak_solution}). 
\end{theorem}

The proofs of Lemma \ref{lem:scheme:periodic_pLaplace_strong_stability} and Theorem \ref{thm:scheme:periodic_pLaplace_convergence}  are very similar (by some obvious simplifications) to the corresponding ones of~Lemma~\ref{lem:scheme:strong_stability} and  Theorem~\ref{thm:scheme:convergence}, and are left to the interested reader. 

\newpage
\section{Explicit time-independent solutions}\label{sec:Exact-Solutions}\enlargethispage{15mm}
\hspace{5mm}In \hspace{-0.15mm}this \hspace{-0.15mm}section, \hspace{-0.15mm}we \hspace{-0.15mm}derive \hspace{-0.15mm}explicit \hspace{-0.15mm}solutions \hspace{-0.15mm}to \hspace{-0.15mm}the \hspace{-0.1mm}2D \hspace{-0.15mm}steady \hspace{-0.15mm}$p(\cdot)$-Navier--Stokes~\hspace{-0.15mm}\mbox{equations},~\hspace{-0.15mm}\mbox{under} the same framework of a fully-developed flow (\textit{cf}.\ Section \ref{sec:fully}). To the~best~of~the~authors'~knowledge, these are the first explicit solutions derived for the 2D steady $p(\cdot)$-Navier--Stokes~equations~in~the fully-developed case, which may be used to create reliable benchmarks for~more~\mbox{complex}~\mbox{problems}. Related results for the determination of the minimum of 
the 1D $p(\cdot)$-Dirichlet energy are obtained in Harjulehto, H\"ast\"o, and Koskenoja~\cite{HHK2003} (with different perspectives considering Lavrentiev~pheno\-menon and giving sufficient conditions for existence of the minimizer~if~$p\colon \Omega  \to  (1,+\infty)$~is~smooth). In \cite{HHK2003}, the derivation of explicit solutions is not the main goal and the derived solutions are of limited practical interest, since (as~described~by~the~authors)~\textit{``\ldots~\hspace{-0.25em}the~\mbox{formula}~is~not~quite~\mbox{transparent}''}. 

We have a different objective: we intend to explicitly solve a problem with 
a piece-wise~constant power-law index $p\colon \Omega  \to  (1,+\infty)$.  
 In fact, we consider~the 
 boundary value problem~with~homoge\-neous Dirichlet boundary condition and are interested in finding 
 sufficient
 conditions on~the~piece-wise \hspace{-0.1mm}constant \hspace{-0.1mm}power-law \hspace{-0.1mm}index \hspace{-0.1mm}$p\colon\hspace{-0.15em} \Omega  \hspace{-0.15em}\to \hspace{-0.15em} (1,+\infty)$ \hspace{-0.1mm}such \hspace{-0.1mm}that~\hspace{-0.1mm}the~\hspace{-0.1mm}\mbox{solution}~\hspace{-0.1mm}is~\hspace{-0.1mm}\mbox{explicitly}~\hspace{-0.1mm}computable.~\hspace{-0.1mm}To~\hspace{-0.1mm}be able to perform the  computations we have in mind, we consider the \mbox{motion}~of~an~electro-rheological fluid in an infinite strip, with applied electric field transverse to the axis~of~the~strip~(\textit{cf}.~\mbox{Figure}~\ref{fig:strip}). 

The connection to  prior sections is that the solution of this problem corresponds~to~a~\mbox{trivially} time-periodic  (\textit{i.e.}, time-independent) case. Moreover, the time-independent solution represents a good trial for an initial datum and a natural extension of the Hagen--Poiseuille~\mbox{solution}~(\textit{cf}.~\mbox{\cite{Hagen1839,Poiseuille1846}}): the latter is considered as a natural initial datum for the `\textcolor{black}{assigned pressure}' problem or a limiting solution~after the flow is re-organized in a long enough pipe for a Newtonian fluid (see,~\textit{e.g.},~\cite{Gal2008}).

In \hspace{-0.1mm}case \hspace{-0.1mm}of \hspace{-0.1mm}an \hspace{-0.1mm}electro-rheological \hspace{-0.1mm}fluid, \hspace{-0.1mm}with \hspace{-0.1mm}applied \hspace{-0.1mm}electric \hspace{-0.1mm}field
\hspace{-0.1mm}transverse \hspace{-0.1mm}to \hspace{-0.1mm}the \hspace{-0.1mm}axis~\hspace{-0.1mm}of~\hspace{-0.1mm}the~\hspace{-0.1mm}strip,
converging at infinity to a vector field varying only  transverse to the strip, we~expect~that~at~large distance the solution will approach  the ones we identify. Hence, 
we seek a 
solution~of~the~2D~steady $p(\cdot)$-Navier--Stokes equations
in an infinite strip $\Omega\hspace{-0.15em}\coloneqq \hspace{-0.15em}\mathbb{R}\hspace{-0.1em}\times\hspace{-0.1em}\Sigma\hspace{-0.15em}\subseteq \hspace{-0.15em}\mathbb{R}^2$ with~\mbox{cross-section}~${\Sigma\hspace{-0.15em}\coloneqq \hspace{-0.15em}(-r,r)\hspace{-0.15em}\subseteq\hspace{-0.15em} \mathbb{R}^1}$,  $r\in (0,+\infty)$,
\textit{i.e.},  for a given electric field $\mathbf{E}\colon\Omega\to \mathbb{R}^2$, for a.e.\footnote{By analogy with \eqref{def:x_bar}, we employ the notation $\overline{x}=x_2$.} $x=(x_1,\overline{x})\in \Omega$, of~the~form
\begin{align}\label{def:E}
    \mathbf{E}(x)=E(\overline{x})\mathbf{e}_2\,,
\end{align}
where $E\colon \Sigma\to \mathbb{R}$ is the  electric field in the $\mathbb{R}\mathbf{e}_2$-direction (\textit{cf}.\ Figure \ref{fig:strip}), 
depending~only~on~the $\overline{x}$-variable, and a power-law index  $p\coloneqq \widehat{p}\circ\vert \mathbf{E}\vert\colon \Omega\to \mathbb{R}$, where $\widehat{p}\colon [0,+\infty)\to (1,+\infty)$ is a material function,
we seek
a velocity field $\mathbf{v}\colon \Omega\to \mathbb{R}^2$ and a kinematic pressure $\pi\colon \Omega\to \mathbb{R}$~such~that
\begin{align}\label{eq:p-NSE_steady}
    \begin{aligned}
    -\textup{div}\,(\vert \mathbf{D}\mathbf{v}\vert^{\smash{p(\cdot)-2}}\mathbf{D}\mathbf{v})+\textup{div}\,(\mathbf{v}\otimes \mathbf{v})+\nabla \pi&=\mathbf{0}_2&&\quad \text{ in }\Omega\,,\\
    \textup{div}\,\mathbf{v}&=0&&\quad \text{ in }\Omega\,,\\
        \mathbf{v}&=\mathbf{0}_2&&\quad\text{ on }\partial \Omega\,.
    \end{aligned}
\end{align} 
For simplicity, 
we assume that 
the fluid is only moving in the $\mathbb{R}\mathbf{e}_1$-direction,~\textit{i.e.},~we~have~that 
\begin{align}
\label{assumption_x_direction}
    \mathbf{v}=(\smash{\textcolor{black}{\sqrt{2}}}v,0)\colon \Omega\to \mathbb{R}^2\,,
\end{align}
where $v\colon \Sigma\to  \mathbb{R}$ is the rescaled velocity in the $\mathbb{R}\mathbf{e}_1$-direction, depending only~on~the~\mbox{$\overline{x}$-variable}.\vspace{-2.5mm}

\begin{figure}[H]
    \centering

  
\tikzset {_z6tpanl3d/.code = {\pgfsetadditionalshadetransform{ \pgftransformshift{\pgfpoint{89.1 bp } { -108.9 bp }  }  \pgftransformscale{1.32 }  }}}
\pgfdeclareradialshading{_3oozyaof4}{\pgfpoint{-72bp}{88bp}}{rgb(0bp)=(1,1,1);
rgb(0bp)=(1,1,1);
rgb(25bp)=(0,0,0);
rgb(400bp)=(0,0,0)}

  
\tikzset {_8n615sces/.code = {\pgfsetadditionalshadetransform{ \pgftransformshift{\pgfpoint{89.1 bp } { -108.9 bp }  }  \pgftransformscale{1.32 }  }}}
\pgfdeclareradialshading{_gn1azds39}{\pgfpoint{-72bp}{88bp}}{rgb(0bp)=(1,1,1);
rgb(0bp)=(1,1,1);
rgb(25bp)=(0,0,0);
rgb(400bp)=(0,0,0)}
\tikzset{every picture/.style={line width=0.75pt}} 

\begin{tikzpicture}[x=1.125pt,y=1.125pt,yscale=-1,xscale=1]

\draw [color={Red}  ,draw opacity=1 ]   (28.23,60.25) -- (55.5,60.43) ;
\draw [shift={(58.5,60.45)}, rotate = 180.37] [fill={Red}  ,fill opacity=1 ][line width=0.08]  [draw opacity=0] (5.36,-2.57) -- (0,0) -- (5.36,2.57) -- (3.56,0) -- cycle    ;
\draw [color={green}  ,draw opacity=1 ]   (28.23,60.25) -- (28.36,32.75) ;
\draw [shift={(28.38,29.75)}, rotate = 90.28] [fill={green}  ,fill opacity=1 ][line width=0.08]  [draw opacity=0] (5.36,-2.57) -- (0,0) -- (5.36,2.57) -- (3.56,0) -- cycle    ;
\draw  [fill={rgb, 255:red, 0; green, 0; blue, 0 }  ,fill opacity=1 ] (29.02,60.25) .. controls (29.02,59.81) and (28.67,59.45) .. (28.23,59.45) .. controls (27.79,59.45) and (27.43,59.81) .. (27.43,60.25) .. controls (27.43,60.69) and (27.79,61.05) .. (28.23,61.05) .. controls (28.67,61.05) and (29.02,60.69) .. (29.02,60.25) -- cycle ;
\draw  [fill={denim}  ,fill opacity=0.1 ] (74.54,90.4) -- (74.54,30.4) -- (367.38,29.99) -- (367.38,89.99) -- (74.54,90.4) -- cycle ;
\draw [color={denim}  ,draw opacity=1 ][line width=1.5]    (76.75,59.97) -- (358.57,60.13) ;
\draw [color={denim}  ,draw opacity=0.8 ][line width=0.75]    (82.25,53.3) -- (358.5,53) ;
\draw [color={denim}  ,draw opacity=0.6 ][line width=0.75]    (82.5,47.05) -- (358.75,46.75) ;
\draw [color={denim}  ,draw opacity=0.25 ][line width=0.75]    (82.75,86.05) -- (363,85.75) ;
\draw [color={denim}  ,draw opacity=0.4 ][line width=0.75]    (82.5,79.8) -- (360.25,79.5) ;
\draw [color={denim}  ,draw opacity=0.6 ][line width=0.75]    (83,73.15) -- (359.25,72.85) ;
\draw [color={denim}  ,draw opacity=0.8 ][line width=0.75]    (82.75,67.05) -- (359,66.75) ;
\draw [color={denim}  ,draw opacity=0.2 ][line width=0.75]    (82.14,34.27) -- (363.38,34) ;
\draw [color={denim}  ,draw opacity=0.4 ][line width=0.75]    (82.25,40.55) -- (360,40.25) ;
\draw [color={denim}  ,draw opacity=1 ][line width=1.5]    (358.32,60.01) -- (388.5,59.87) ;
\draw [shift={(392.5,59.85)}, rotate = 179.73] [fill={denim}  ,fill opacity=1 ][line width=0.08]  [draw opacity=0] (6.43,-3.09) -- (0,0) -- (6.43,3.09) -- (4.27,0) -- cycle    ;
\draw [color={denim}  ,draw opacity=0.2 ][line width=0.75]    (363.38,34) -- (388.13,34) ;
\draw [shift={(391.13,34)}, rotate = 180] [fill={denim}  ,fill opacity=0.2 ][line width=0.08]  [draw opacity=0] (5.36,-2.57) -- (0,0) -- (5.36,2.57) -- (3.56,0) -- cycle    ;
\draw [color={denim}  ,draw opacity=0.4 ][line width=0.75]    (360.5,40.1) -- (388.25,40.1) ;
\draw [shift={(391.25,40.1)}, rotate = 180] [fill={denim}  ,fill opacity=0.4 ][line width=0.08]  [draw opacity=0] (5.36,-2.57) -- (0,0) -- (5.36,2.57) -- (3.56,0) -- cycle    ;
\draw [color={denim}  ,draw opacity=0.6 ][line width=0.75]    (359,46.6) -- (388.25,46.6) ;
\draw [shift={(391.25,46.6)}, rotate = 180] [fill={denim}  ,fill opacity=0.6 ][line width=0.08]  [draw opacity=0] (5.36,-2.57) -- (0,0) -- (5.36,2.57) -- (3.56,0) -- cycle    ;
\draw [color={denim}  ,draw opacity=0.8 ][line width=0.75]    (359,52.85) -- (388.25,52.85) ;
\draw [shift={(391.25,52.85)}, rotate = 180] [fill={denim}  ,fill opacity=0.8 ][line width=0.08]  [draw opacity=0] (5.36,-2.57) -- (0,0) -- (5.36,2.57) -- (3.56,0) -- cycle    ;
\draw [color={denim}  ,draw opacity=0.8 ][line width=0.75]    (359.25,66.6) -- (388.5,66.6) ;
\draw [shift={(391.5,66.6)}, rotate = 180] [fill={denim}  ,fill opacity=0.8 ][line width=0.08]  [draw opacity=0] (5.36,-2.57) -- (0,0) -- (5.36,2.57) -- (3.56,0) -- cycle    ;
\draw [color={denim}  ,draw opacity=0.6 ][line width=0.75]    (359.25,72.85) -- (388.5,72.85) ;
\draw [shift={(391.5,72.85)}, rotate = 180] [fill={denim}  ,fill opacity=0.6 ][line width=0.08]  [draw opacity=0] (5.36,-2.57) -- (0,0) -- (5.36,2.57) -- (3.56,0) -- cycle    ;
\draw [color={denim}  ,draw opacity=0.4 ][line width=0.75]    (360.25,79.5) -- (388,79.5) ;
\draw [shift={(391,79.5)}, rotate = 180] [fill={denim}  ,fill opacity=0.4 ][line width=0.08]  [draw opacity=0] (5.36,-2.57) -- (0,0) -- (5.36,2.57) -- (3.56,0) -- cycle    ;
\draw [color={denim}  ,draw opacity=0.25 ][line width=0.75]    (363.5,85.6) -- (388.13,85.73) ;
\draw [shift={(391.13,85.75)}, rotate = 180.31] [fill={denim}  ,fill opacity=0.25 ][line width=0.08]  [draw opacity=0] (5.36,-2.57) -- (0,0) -- (5.36,2.57) -- (3.56,0) -- cycle    ;
\draw [color={gray}  ,draw opacity=1 ]   (220.25,30) -- (220.25,24) ;
\draw [shift={(220,23.92)}, rotate = 89.84] [color={gray}  ,draw opacity=1 ][line width=0.75]    (0,5.59) -- (0,-5.59)   ;
\draw [color={byzantium}  ,draw opacity=1 ]   (100.01,30.26) -- (99.99,23.93) ;
\draw [shift={(99.99,23.93)}, rotate = 89.84] [color={byzantium}  ,draw opacity=1 ][line width=0.75]    (0,5.59) -- (0,-5.59)   ;
\draw [color={denim}  ,draw opacity=1 ]   (310.43,59.78) -- (310.5,23.75) ;
\draw [shift={(310.5,23.75)}, rotate = 90.11] [color={denim}  ,draw opacity=1 ][line width=0.75]    (0,5.59) -- (0,-5.59)   ;
\path  [shading=_3oozyaof4,_z6tpanl3d] (355.62,89.88) -- (355.62,89.87) .. controls (357.2,89.67) and (358.44,87.51) .. (358.42,84.9) .. controls (358.41,82.32) and (357.17,80.22) .. (355.62,80.03) -- (355.62,79.83) .. controls (353.85,79.44) and (352.5,77.29) .. (352.5,74.69) .. controls (352.5,72.09) and (353.85,69.93) .. (355.62,69.55) -- (355.62,69.49) .. controls (357.21,69.3) and (358.46,67.13) .. (358.45,64.51) .. controls (358.43,61.92) and (357.19,59.82) .. (355.62,59.64) -- (355.62,59.49) .. controls (353.89,59.41) and (352.5,57.26) .. (352.5,54.62) .. controls (352.5,51.99) and (353.89,49.84) .. (355.62,49.76) -- (355.62,49.76) .. controls (355.63,49.76) and (355.64,49.76) .. (355.65,49.76) .. controls (357.2,49.73) and (358.44,47.5) .. (358.42,44.77) .. controls (358.41,42.07) and (357.15,39.89) .. (355.62,39.89) .. controls (353.91,39.89) and (352.53,37.65) .. (352.53,34.89) .. controls (352.53,32.5) and (353.57,30.51) .. (354.95,30.01) -- (367.38,29.99) -- (367.38,89.88) -- (355.62,89.88) -- cycle ; 
 \draw   (355.62,89.88) -- (355.62,89.87) .. controls (357.2,89.67) and (358.44,87.51) .. (358.42,84.9) .. controls (358.41,82.32) and (357.17,80.22) .. (355.62,80.03) -- (355.62,79.83) .. controls (353.85,79.44) and (352.5,77.29) .. (352.5,74.69) .. controls (352.5,72.09) and (353.85,69.93) .. (355.62,69.55) -- (355.62,69.49) .. controls (357.21,69.3) and (358.46,67.13) .. (358.45,64.51) .. controls (358.43,61.92) and (357.19,59.82) .. (355.62,59.64) -- (355.62,59.49) .. controls (353.89,59.41) and (352.5,57.26) .. (352.5,54.62) .. controls (352.5,51.99) and (353.89,49.84) .. (355.62,49.76) -- (355.62,49.76) .. controls (355.63,49.76) and (355.64,49.76) .. (355.65,49.76) .. controls (357.2,49.73) and (358.44,47.5) .. (358.42,44.77) .. controls (358.41,42.07) and (357.15,39.89) .. (355.62,39.89) .. controls (353.91,39.89) and (352.53,37.65) .. (352.53,34.89) .. controls (352.53,32.5) and (353.57,30.51) .. (354.95,30.01) -- (367.38,29.99) -- (367.38,89.88) -- (355.62,89.88) -- cycle ; 

\draw [color={denim}  ,draw opacity=0.2 ][line width=0.75]    (82.14,34.27) -- (66.88,35) ;
\draw [color={denim}  ,draw opacity=0.25 ][line width=0.75]    (82.75,86.05) -- (66.88,86.25) ;
\draw [color={denim}  ,draw opacity=0.4 ][line width=0.75]    (66.88,40.75) -- (82.25,40.55) ;
\draw [color={denim}  ,draw opacity=0.6 ][line width=0.75]    (66.5,47.05) -- (84.63,47) ;
\draw [color={denim}  ,draw opacity=0.8 ][line width=0.75]    (66.75,53.55) -- (82.25,53.3) ;
\draw [color={denim}  ,draw opacity=1 ][line width=1.5]    (66.25,60) -- (76.73,59.97) ;
\draw [color={denim}  ,draw opacity=0.8 ][line width=0.75]    (66.25,67.05) -- (82.75,67.05) ;
\draw [color={denim}  ,draw opacity=0.6 ][line width=0.75]    (66.75,73.15) -- (83,73.15) ;
\draw [color={denim}  ,draw opacity=0.4 ][line width=0.75]    (66.88,80) -- (82.5,79.8) ;
\path  [shading=_gn1azds39,_8n615sces] (85.59,30.54) .. controls (83.94,30.75) and (82.65,32.91) .. (82.68,35.53) .. controls (82.71,38.11) and (84.01,40.2) .. (85.63,40.38) -- (85.63,40.59) .. controls (87.48,40.96) and (88.91,43.11) .. (88.92,45.71) .. controls (88.93,48.32) and (87.52,50.48) .. (85.68,50.87) -- (85.68,50.93) .. controls (84.01,51.12) and (82.71,53.29) .. (82.74,55.91) .. controls (82.76,58.5) and (84.07,60.6) .. (85.72,60.77) -- (85.72,60.93) .. controls (87.53,61) and (88.99,63.14) .. (88.5,65.78) .. controls (89.01,68.42) and (87.57,70.57) .. (85.76,70.65) -- (85.76,70.66) .. controls (85.74,70.66) and (85.73,70.66) .. (85.72,70.66) .. controls (84.1,70.69) and (82.81,72.93) .. (82.84,75.65) .. controls (82.87,78.36) and (84.19,80.53) .. (85.8,80.52) .. controls (87.58,80.52) and (89.03,82.75) .. (89.04,85.51) .. controls (89.05,87.9) and (87.98,89.9) .. (86.54,90.4) -- (74.11,90.47) .. controls (74.03,90.45) and (73.95,90.42) .. (73.88,90.4) -- (74.36,90.4) -- (74.36,30.57) -- (84.55,30.54) -- (85.59,30.54) -- (85.59,30.54) -- cycle ; 
 \draw   (85.59,30.54) .. controls (83.94,30.75) and (82.65,32.91) .. (82.68,35.53) .. controls (82.71,38.11) and (84.01,40.2) .. (85.63,40.38) -- (85.63,40.59) .. controls (87.48,40.96) and (88.91,43.11) .. (88.92,45.71) .. controls (88.93,48.32) and (87.52,50.48) .. (85.68,50.87) -- (85.68,50.93) .. controls (84.01,51.12) and (82.71,53.29) .. (82.74,55.91) .. controls (82.76,58.5) and (84.07,60.6) .. (85.72,60.77) -- (85.72,60.93) .. controls (87.53,61) and (88.99,63.14) .. (88.5,65.78) .. controls (89.01,68.42) and (87.57,70.57) .. (85.76,70.65) -- (85.76,70.66) .. controls (85.74,70.66) and (85.73,70.66) .. (85.72,70.66) .. controls (84.1,70.69) and (82.81,72.93) .. (82.84,75.65) .. controls (82.87,78.36) and (84.19,80.53) .. (85.8,80.52) .. controls (87.58,80.52) and (89.03,82.75) .. (89.04,85.51) .. controls (89.05,87.9) and (87.98,89.9) .. (86.54,90.4) -- (74.11,90.47) .. controls (74.03,90.45) and (73.95,90.42) .. (73.88,90.4) -- (74.36,90.4) -- (74.36,30.57) -- (84.55,30.54) -- (85.59,30.54) -- (85.59,30.54) -- cycle ; 

\draw [color={byzantium}  ,draw opacity=0.9 ][line width=7.0]  (100,30) -- (100,90) ;
\draw [color={byzantium}  ,draw opacity=0.85 ][line width=6.5]    (120,30) -- (120,90) ;
\draw [color={byzantium}  ,draw opacity=0.8 ][line width=6.0]    (140,30) -- (140,90) ;
\draw [color={byzantium}  ,draw opacity=0.75 ][line width=5.5]    (160.5,30) -- (160.5,90) ;
\draw [color={byzantium}  ,draw opacity=0.7 ][line width=5.0]    (180,30) -- (180,90) ;
\draw [color={byzantium}  ,draw opacity=0.65 ][line width=4.5]    (200,30) -- (200,90) ;
\draw [color={byzantium}  ,draw opacity=0.6 ][line width=4.0]    (220,30) -- (220,90) ;
\draw [color={byzantium}  ,draw opacity=0.55 ][line width=3.5]   (240,30) -- (240,90) ;
\draw [color={byzantium}  ,draw opacity=0.5 ][line width=3.0]    (260,30) -- (260,90) ;
\draw [color={byzantium}  ,draw opacity=0.45 ][line width=2.5]    (280,30) -- (280,90) ;
\draw [color={byzantium}  ,draw opacity=0.4 ][line width=2.0]    (300,30) -- (300,90) ;
\draw [color={byzantium}  ,draw opacity=0.35 ][line width=1.5]    (320,30) -- (320,90) ;
\draw [color={byzantium}  ,draw opacity=0.3 ][line width=1.0]    (340,30) -- (340,90) ;

\foreach \x in {0,...,72} {
\draw [color={shamrockgreen}  ,draw opacity=0.2, ->, >=latex, shorten >=0.5pt  ] [fill={shamrockgreen}  ,fill opacity=0.2 ] [line width=0.4] (77+4*\x,26) -- (77+4*\x,33) ;
\draw [color={shamrockgreen}  ,draw opacity=0.4, ->, >=latex, shorten >=0.5pt ][fill={shamrockgreen}  ,fill opacity=0.4 ]  [line width=0.4] (77+4*\x,33) -- (77+4*\x,40) ;
\draw [color={shamrockgreen}  ,draw opacity=0.6, ->, >=latex, shorten >=0.5pt][fill={shamrockgreen}  ,fill opacity=0.6 ]  [line width=0.4] (77+4*\x,40) -- (77+4*\x,47) ;
\draw [color={shamrockgreen}  ,draw opacity=0.8, ->, >=latex, shorten >=0.5pt][fill={shamrockgreen}  ,fill opacity=0.8 ] [line width=0.4]  (77+4*\x,47) -- (77+4*\x,54) ;
\draw [color={shamrockgreen}  ,draw opacity=1, ->, >=latex, shorten >=0.5pt][fill={shamrockgreen}  ,fill opacity=1 ]  [line width=0.4] (77+4*\x,54) -- (77+4*\x,61) ;
\draw [color={shamrockgreen}  ,draw opacity=1, ->, >=latex, shorten >=0.5pt][fill={shamrockgreen}  ,fill opacity=1 ]  [line width=0.4] (77+4*\x,61) -- (77+4*\x,68) ;
\draw [color={shamrockgreen}  ,draw opacity=0.8, ->, >=latex, shorten >=0.5pt][fill={shamrockgreen}  ,fill opacity=0.8 ] [line width=0.4]  (77+4*\x,68) -- (77+4*\x,75) ;
\draw [color={shamrockgreen}  ,draw opacity=0.6, ->, >=latex, shorten >=0.5pt][fill={shamrockgreen}  ,fill opacity=0.6 ]  [line width=0.4] (77+4*\x,75) -- (77+4*\x,82) ;
\draw [color={shamrockgreen}  ,draw opacity=0.4, ->, >=latex, shorten >=0.5pt][fill={shamrockgreen}  ,fill opacity=0.4 ] [line width=0.4]  (77+4*\x,82) -- (77+4*\x,89) ;
\draw [color={shamrockgreen}  ,draw opacity=0.2, ->, >=latex, shorten >=0.5pt] [fill={shamrockgreen}  ,fill opacity=0.2 ] [line width=0.4] (77+4*\x,89) -- (77+4*\x,96) ;
}

\draw [color={shamrockgreen}  ,draw opacity=1 ]   (169,30) -- (169,23.42) ;
\draw [shift={(169,23.42)}, rotate = 89.84] [color={shamrockgreen}  ,draw opacity=1 ][line width=0.75]    (0,5.59) -- (0,-5.59)   ;
\draw [color={gray}  ,draw opacity=1 ]  [line width=2] (220.25,30) -- (220.25,90) ;

\draw (286,12) node [anchor=north west][inner sep=0.75pt]  [color={denim}  ,opacity=1 ]  {$\mathbf{v}( x) =\sqrt{2}v(\overline{x})\mathbf{e}_{1}$};
\draw (76,12) node [anchor=north west][inner sep=0.75pt]  [color={byzantium}  ,opacity=1 ]  {$\pi ( x) =\tfrac{\Gamma}{\sqrt{2}}  x_{1}$};
\draw (54.3,62.65) node [anchor=north west][inner sep=0.75pt]  [font=\small,color={rgb, 255:red, 0; green, 0; blue, 0 }  ,opacity=1 ]  {$x_{1}$};
\draw (32.1,26.45) node [anchor=north west][inner sep=0.75pt]  [font=\small,color={rgb, 255:red, 0; green, 0; blue, 0 }  ,opacity=1 ]  {$x_{2} =\overline{x}$};
\draw (214,10) node [anchor=north west][inner sep=0.75pt]  [color={gray}  ,opacity=1 ] [font=\LARGE] {$\Sigma $};
\draw (143.5,12) node [anchor=north west][inner sep=0.75pt]  [color={shamrockgreen}  ,opacity=1 ]  {$\mathbf{E}( x) =E(\overline{x})\mathbf{e}_{2}$};

\end{tikzpicture}\vspace{-1.5mm}

    \caption{Schematic diagram of an infinite pipe $\Omega\coloneqq \mathbb{R}\times \Sigma$ with~\mbox{cross-section}~$\Sigma\subseteq \mathbb{R}^1$:~in~\textcolor{denim}{blue}, the velocity field $\mathbf{v}\colon \Omega\to \mathbb{R}^2$, which depends only on the $\overline{x}$-variable and points~in~the~\mbox{$\mathbb{R}\mathbf{e}_1$-direction}; 
    in \textcolor{byzantium}{purple}, the kinematic pressure $\pi\colon \Omega\to \mathbb{R}$,  which only depends~on~the~\mbox{$x_1$-variable}; in \textcolor{shamrockgreen}{green}, the electric field $\mathbf{E}\colon \Omega\to \mathbb{R}^2$,  which only depends on the $\overline{x}$-variable and points in the $\mathbb{R}\mathbf{e}_2$-direction.}
    \label{fig:strip}
\end{figure}

\newpage
 
 With the ansatz~\eqref{assumption_x_direction}, similar to Section \ref{sec:fully},
 in the 2D steady $p(\cdot)$-Navier--Stokes~\mbox{equations}~\eqref{eq:p-NSE_steady}, the following reductions apply:
 \begin{subequations} 
 \begin{itemize}[noitemsep,topsep=2pt,leftmargin=!,labelwidth=\widthof{$\bullet$}]
    \item[$\bullet$] \emph{(Incompressibility).} 
    The flow is incompressible, \textit{i.e.}, we have that 
\begin{align}
   \textup{div}\,\mathbf{v}= \smash{\textcolor{black}{\sqrt{2}}}\partial_{x_1} v=0\quad \text{ a.e.\ in }\Omega\,;\label{eq:formula.1-2D}
\end{align}
\item[$\bullet$] \emph{(Laminarity).} 
There is no convection and, therefore, the flow is laminar, \textit{i.e.}, we have that 
\begin{align}
    \textup{div}(\mathbf{v}\otimes \mathbf{v})
    =
    \smash{\textcolor{black}{\sqrt{2}}}\left(\begin{array}{c}
         \smash{\partial_{x_1}\vert v\vert^2}  \\
         0 
    \end{array}\right)=\mathbf{0}_2\quad \text{ a.e.\ in }\Omega\,;
    \label{eq:formula.2-2D}
\end{align}
\item[$\bullet$] \emph{($x_1$-independence of strain).} 
The strain depends only on the $\overline{x}$-derivative of $v$ and, thus, the shear-rate $ \vert\mathbf{D}\mathbf{v}\vert =\vert \partial_{\overline{x}} v\vert$ as well as the stress tensor $\mathbf{S}(\cdot,\mathbf{Dv})=\vert \mathbf{D}\mathbf{v}\vert^{p(\cdot)-2}\mathbf{D}\mathbf{v}$~and~its~\mbox{divergence}
\begin{align}
\label{eq:formula.5-2D}
    \textup{div}\,(\vert \mathbf{D}\mathbf{v}\vert^{p(\cdot)-2}\mathbf{D}\mathbf{v})
    =
    \left(\begin{array}{c}
         \textcolor{black}{\frac{1}{\sqrt{2}}}\partial_{\overline{x}}(\vert \partial_{\overline{x}}v\vert^{p(\cdot)-2}\partial_{\overline{x}}v)  \\
         0 
    \end{array}\right)
    \quad \text{ a.e.\ in }\Omega\,.
\end{align}
\end{itemize}
\end{subequations}

In summary, due to the reductions \eqref{eq:formula.1-2D}--\eqref{eq:formula.5-2D},
the 2D steady $p(\cdot)$-Navier--Stokes equations \eqref{eq:p-NSE_steady} reduce to a system seeking 
for two scalar unknowns  $v\colon \Sigma\to \mathbb{R}$ and $\pi\colon \Omega\to \mathbb{R}$ such that
\begin{subequations} \label{eq:p-NSE_steady_reduced}
\begin{alignat}{2}
     -\partial_{\overline{x}}(\vert \partial_{\overline{x}}v\vert^{p(\cdot)-2}\partial_{\overline{x}}v)+\smash{\textcolor{black}{\sqrt{2}}}\partial_{x_1}\pi&=0&&\quad \text{ a.e.\ in }\Omega\,,\label{eq:p-NSE_steady_reduced.1}\\
     \partial_{\overline{x}} \pi&=0&&\quad \text{ a.e.\ in }\Omega\,,\label{eq:p-NSE_steady_reduced.2}\\
     v(\pm r)&=0\,.\label{eq:p-NSE_steady_reduced.3}
\end{alignat}
\end{subequations}
 
 Two observations can be made in the reduced 2D steady $p(\cdot)$-Navier--Stokes equations \eqref{eq:p-NSE_steady_reduced}:
\begin{itemize}[noitemsep,topsep=2pt,leftmargin=!,labelwidth=\widthof{Observation B:},font=\itshape]
    \item[Observation A:]\hypertarget{OA}{} As usual for fully-developed flows, from  \eqref{eq:p-NSE_steady_reduced.2}, we deduce that the kinematic pressure $\pi\colon \Omega\to \mathbb{R}$ is independent of the $\overline{x}$-variable, \textit{i.e.}, we have that $\pi\colon \mathbb{R}\to \mathbb{R}$;
    \item[Observation B:]\hypertarget{OB}{} Since, owing to \eqref{def:E},
    the power-law index $p\colon \Omega\to(1,+\infty)$ is independent~of~the \mbox{$x_1$-variable}, \textit{i.e.},~we~have~that $p\colon \Sigma\to(1,+\infty)$, 
     from \eqref{eq:p-NSE_steady_reduced.1} and Observation \hyperlink{OA}{A}, it follows that there exist constants $\Gamma,\textcolor{black}{c_0^\pi}\in \mathbb{R}$ such that
\begin{subequations}\label{eq:p-NSE_steady_reduced.2.0} 
\begin{alignat}{2}
   -\partial_{\overline{x}}(\vert \partial_{\overline{x}}v\vert^{p(\cdot)-2}\partial_{\overline{x}}v)+\textcolor{black}{\Gamma}&=0&&\quad \text{ a.e.\ in }\Sigma\,,\label{eq:p-NSE_steady_reduced.2.1} \\
    \pi&=\textcolor{black}{\smash{\tfrac{\Gamma}{\sqrt{2}}}} \mathrm{id}_{\mathbb{R}}+\textcolor{black}{c_0^\pi}&&\quad \text{ a.e.\ in }\mathbb{R}\,. \label{eq:p-NSE_steady_reduced.2.2} 
\end{alignat}
\end{subequations}
\end{itemize}
 
If we set $\textcolor{black}{\Gamma=-1}$ in \eqref{eq:p-NSE_steady_reduced.2.1} and \eqref{eq:p-NSE_steady_reduced.2.2}, then imposing $\textcolor{black}{\pi}(0)=0$, 
to enforce the uniqueness of the kinematic pressure, it follows that $\textcolor{black}{c_0^\pi}=0$. 
As a result, in order to satisfy  the ansatz~\eqref{assumption_x_direction}, it is only left to determine the velocity in the $\mathbb{R}\mathbf{e}_1$-direction $v\colon \Sigma\to \mathbb{R}$ solving \eqref{eq:p-NSE_steady_reduced.2.1}~with~$\textcolor{black}{\Gamma=-1}$~and~\eqref{eq:p-NSE_steady_reduced.3},  which we will do for two particular choices of the power-law index:\enlargethispage{5mm}
\begin{itemize}[noitemsep,topsep=2pt,leftmargin=!,labelwidth=\widthof{(b)},font=\itshape]
    \item[(a)] $p\colon \Sigma\to (1,+\infty)$ is piece-wise constant and even;
    \item[(b)] $p\colon \Sigma\to (1,+\infty)$ is piece-wise constant and non-even.
\end{itemize} 
\begin{remark}
    We consider the 2D steady $p(\cdot)$-Navier--Stokes equations \eqref{eq:p-NSE_steady}, which become, with the ansatz~\eqref{assumption_x_direction}, 
    the 1D $p(\cdot)$-Dirichlet equation \eqref{eq:p-NSE_steady_reduced.2.1}, which can
 be solved~explicitly. Note that the presence of jumps of the power-law index $p\colon \Sigma\to (1,+\infty)$ will not make it possible to deal with the classical regularity results. Nevertheless, our solutions will be Lipschitz and smooth out of a finite number of points; even if some geometric properties of $p\colon \Sigma\to (1,+\infty)$~are~needed. In addition, an extension to the  3D steady $p(\cdot)$-Navier--Stokes equations with a power-law index $p\colon\hspace{-0.1em} \Sigma\hspace{-0.1em}\to\hspace{-0.1em} (1,+\infty)$ having a finite number of jumps in only one direction excludes~\mbox{singular}~behaviours: a 1D minimization problem can
be extended to higher dimensional rectangular ducts simply by choosing the power-law index to depend
on one coordinate only.
\end{remark}

\subsection*{(a) Even case}\vspace{-1mm}\hypertarget{subsec:symmetric}{}\enlargethispage{3mm} 
    
    \hspace{5mm}Let  $(\zeta_i)_{i=1,\ldots,N}\hspace{-0.15em}\subseteq\hspace{-0.15em} (-r,0]$, $N\hspace{-0.15em}\in \hspace{-0.15em}\mathbb{N}$,   be such that $-r\hspace{-0.15em}=\hspace{-0.15em}\zeta_0\hspace{-0.15em}<\hspace{-0.15em}\ldots\hspace{-0.15em}<\hspace{-0.15em}\zeta_N\hspace{-0.15em}=\hspace{-0.15em}0$ 
    and  ${(p_i)_{i=1,\ldots,N}\hspace{-0.15em}\subseteq \hspace{-0.15em}(1,+\infty)}$. Then, let the piece-wise constant and even 
    power-law index
    $p\colon \Sigma\to  (1,+\infty)$, for every~$\overline{x}\in I_i\coloneqq (\zeta_{i-1},\zeta_i]$, $i=1,\ldots,N$, be defined by
    \begin{align}\label{eq:symmetric.1}
        p(\pm \overline{x})\coloneqq p_i \,.
    \end{align} 
    Recall that, by \eqref{eq:p-NSE_steady_reduced.2.1} with $\Gamma\hspace{-0.1em}=\hspace{-0.1em}-1$ and \eqref{eq:p-NSE_steady_reduced.3}, we seek  $v\hspace{-0.1em}\in \hspace{-0.1em}W^{1,p(\cdot)}_0(\Sigma)$ with ${\vert \partial_{\overline{x}}v\vert^{p(\cdot)-2}\partial_{\overline{x}}v\hspace{-0.1em}\in\hspace{-0.1em} W^{1,1}(\Sigma)}$ such that\vspace{-0.5mm}
    \begin{subequations}\label{eq:symmetric.2}
    \begin{align}
            -\partial_{\overline{x}}(\vert \partial_{\overline{x}}v\vert^{p(\cdot)-2}\partial_{\overline{x}}v)&=1\quad \text{ a.e.\ in }\Sigma\,,\label{eq:symmetric.2.1}\\
            v(\pm r)&=0\,.\label{eq:symmetric.2.2}\\[-6mm]\notag
    \end{align} 
    \end{subequations}
    To begin with, due to \eqref{eq:symmetric.2.1}, for every $i=1,\ldots,N$, there exists some $a_i\in \mathbb{R}$  such that\vspace{-0.5mm}
    \begin{align}\label{eq:symmetric.3}
        \smash{\vert \partial_{\overline{x}}v\vert^{p(\cdot)-2}\partial_{\overline{x}}v=a_i-\textup{id}_{\mathbb{R}}\quad\text{ in }I_i\,.}\\[-6mm]\notag
    \end{align}
    Next, in order to be able to find a  Hagen--Poiseuille-type solution (\textit{cf}.\ \cite{Hagen1839,Poiseuille1846}), let us~assume~that\vspace{-0.5mm}
    \begin{align}\label{eq:symmetric.4.-1}
        \partial_{\overline{x}}v(\overline{x})\begin{cases}
            >0& \text{ for a.e.\ }\overline{x}\in (-r,0)\,,\\[-0.5mm]
            =0& \text{ for }\overline{x}=0\,,\\[-0.5mm]
            <0& \text{ for a.e.\ }\overline{x}\in (0,r)\,.
        \end{cases}\\[-6mm]\notag
    \end{align}  
    Using assumption \eqref{eq:symmetric.4.-1}, from \eqref{eq:symmetric.3}, we infer that $a_i>\zeta_i$ for all $i=1,\ldots,N-1$~as~well~as~$a_N\ge 0$ and, consequently, for every $i=1,\ldots,N$, \eqref{eq:symmetric.3} equivalently  reads\vspace{-0.5mm}
    \begin{align}\label{eq:symmetric.4}
        \smash{(\partial_{\overline{x}}v)^{p_i-1}=a_i-\textup{id}_{\mathbb{R}}\quad\Leftrightarrow \quad \partial_{\overline{x}}v=(a_i-\textup{id}_{\mathbb{R}})^{\smash{\frac{1}{p_i-1}}}\quad\text{ a.e.\ in }I_i\,.}\\[-6mm]\notag
    \end{align}  
    Then, due to \eqref{eq:symmetric.4}, for every $i=1,\ldots,N$, there exists some $b_i\in \mathbb{R}$ such that\vspace{-0.5mm}
    \begin{align}\label{eq:symmetric.5}
        \smash{v=-\tfrac{1}{(p_i)'}\vert a_i-\textup{id}_{\mathbb{R}}\vert^{(p_i)'}+b_i\,\quad\text{ a.e.\ in }I_i\,.}\\[-6mm]\notag
    \end{align}
    
	Let us next explicitly identify the constants $a_i,b_i\in \mathbb{R}$, $i=1,\ldots,N$:\enlargethispage{7.5mm}

    \textit{$\bullet$ Identification of $a_i\in \mathbb{R}$, $i=1,\ldots,N$.} 
    Due to ${\vert \partial_{\overline{x}}v\vert^{p(\cdot)-2}\partial_{\overline{x}}v=(\partial_{\overline{x}}v)^{p(\cdot)-1}\in W^{1,1}(-r,0)}$ (\textit{cf}.\ \eqref{eq:symmetric.4.-1}), we have that $(\partial_{\overline{x}}v)^{p(\cdot)-1}\in C^0([-r,0])$, 
    which, due to \eqref{eq:symmetric.4},  implies that $a_i=a$~for~all $i=1,\ldots,N$. Then, from  $\partial_{\overline{x}}v(0)=0$ (\textit{cf}.\ \eqref{eq:symmetric.4.-1}), we infer that 
    \begin{align*}
        a=a_i=0\quad\text{ for all }i=1,\ldots,N\,.
    \end{align*}

    \textit{$\bullet$ Identification of $b_i\in \mathbb{R}$, $i=1,\ldots,N$.}
    Due to $v(-r)=0$ (\textit{cf}.\ \eqref{eq:symmetric.2.2}), we~have~that~$b_1=\smash{\frac{1}{(p_1)'}}r^{(p_1)'}$. Then, due to $v\in W^{1,p(\cdot)}(\Sigma)$,  we have that $v\in C^0(\overline{\Sigma})$, 
    so~that using \eqref{eq:symmetric.5}, 
    we can compute $b_i\in \mathbb{R}$, $i=2,\ldots,N$, iteratively via 
       \begin{align*}
           \smash{b_i=
           -\tfrac{1}{(p_{i-1})'}\vert\zeta_{i-1}\vert^{(p_{i-1})'}+b_{i-1}
           +\tfrac{1}{(p_i)'}\vert\zeta_{i-1}\vert^{(p_i)'}\quad \text{ for all }i=2,\ldots,N}\,.
       \end{align*}
    \textcolor{black}{Two representative velocity profiles and power-law indices are displayed in Figure~\ref{fig:explicit_even}.}\vspace{-0.5mm}

    \begin{figure}[H]
    \centering
    \includegraphics[width=0.5\linewidth]{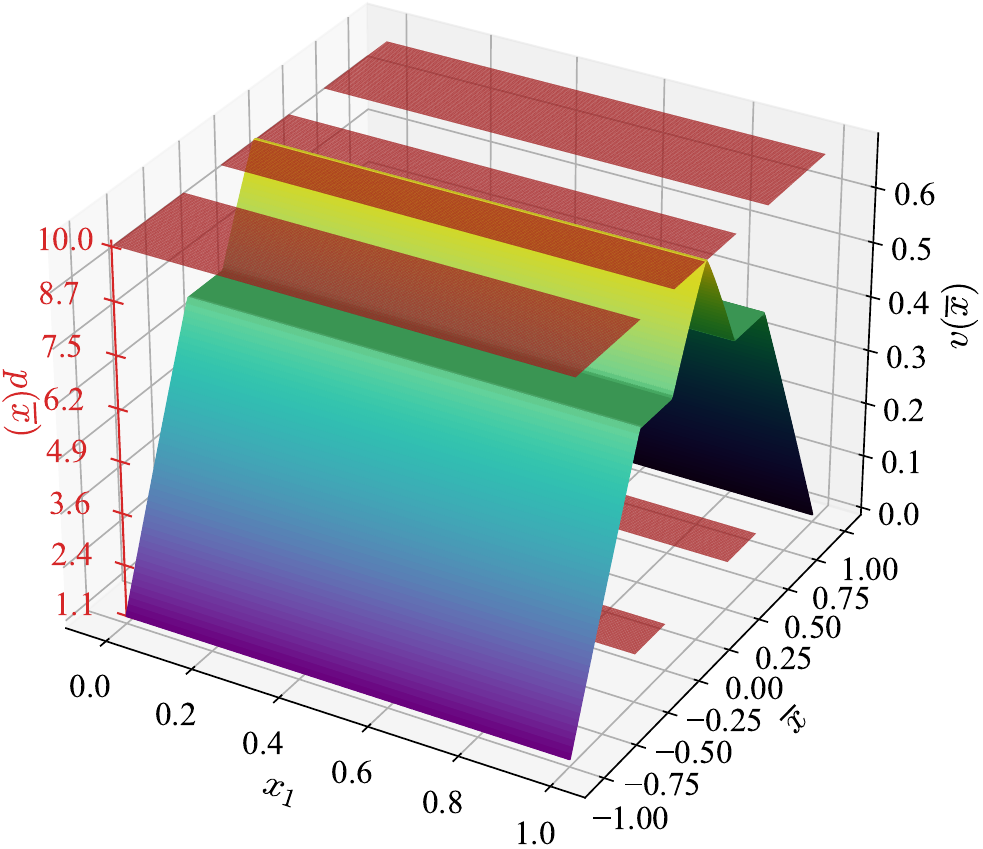}\includegraphics[width=0.5\linewidth]{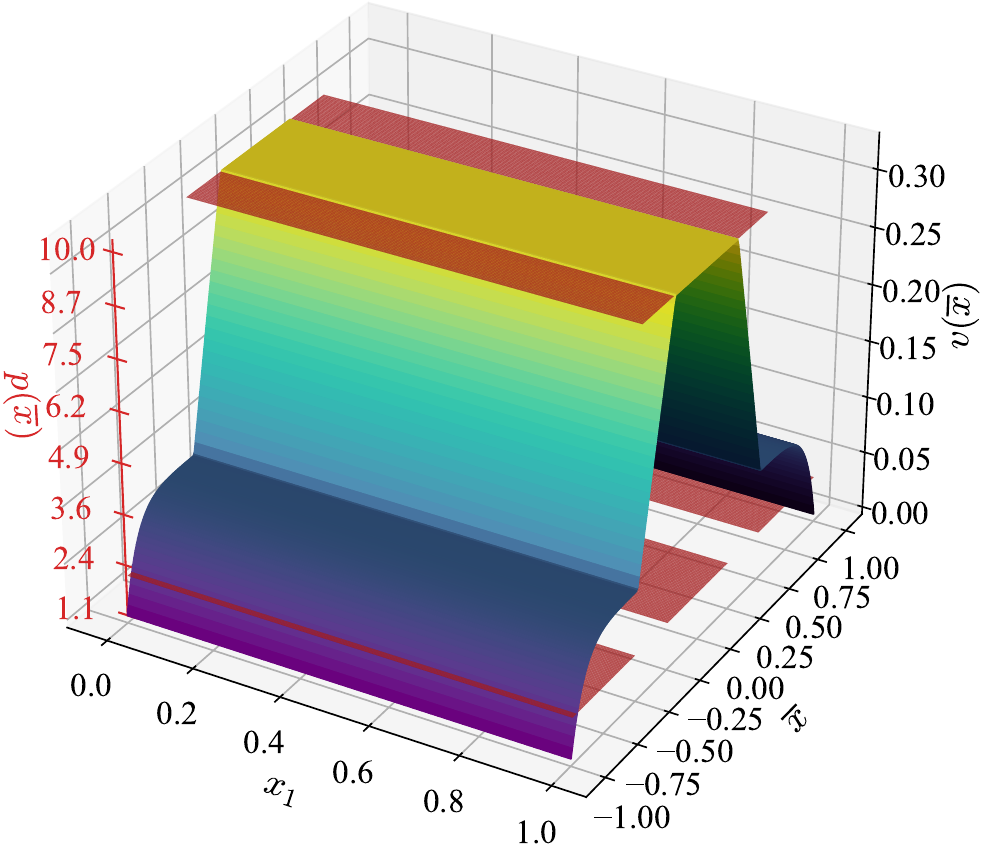}
    
    \caption{\color{black}Surface plots associated with the explicit even solution of the reduced steady
problem \eqref{eq:symmetric.2} with $\Gamma=-1$: the rescaled velocity in the
$\mathbb{R}\mathbf e_1$-direction $v\colon \Sigma\to \mathbb R$ given by
\eqref{eq:symmetric.5} (\textcolor{byzantium}{v}\textcolor{denim}{i}\textcolor{shamrockgreen}{r}\textcolor{Goldenrod}{i}\textcolor{byzantium}{d}\textcolor{denim}{i}\textcolor{shamrockgreen}{s} colormap)
and the piece-wise constant even power-law index
$p\colon \Sigma\to (1,+\infty)$ (\textcolor{black}{red}), each restricted to
$(0,1)\times\Sigma\subseteq \Omega$. Here, $\zeta_0\coloneqq-1$,
$\zeta_1=-\frac{1}{2}$, $\zeta_2=-\frac{1}{4}$, and $\zeta_3=0$:
\textit{left:} $p_1=10$, $p_2=1.1$, and $p_3=10$;
\textit{right:} $p_1=1.1$, $p_2=10$, and $p_3=1.1$.}
\label{fig:explicit_even}
\end{figure}\newpage  

In the above case, we have a piece-wise constant and even 
power-law index $p\colon \Sigma\to (1,+\infty)$, which determines a piece-wise regular and even rescaled velocity component $v\colon\Sigma \to \mathbb{R} $. If we drop the assumption that the power-law index is even, the above construction~might~no~longer~be~\mbox{possible}. Next, we 
consider the case of a piece-wise constant and non-even 
power-law index 
$p\colon \Sigma\to (1,+\infty)$, which has a single point of discontinuity.\vspace{-1mm}\enlargethispage{2.5mm}

    \subsection*{(b) Non-even case}\vspace{-1mm}\hypertarget{subsec:nonsymmetric}{}

    
    \hspace{5mm}Let $\zeta\in \Sigma$ be the prescribed point of discontinuity and $p_1,p_2\in (1,+\infty)$.
    Then, let the piece-wise constant and  non-even power-law index
    $p\colon \Sigma\to  (1,+\infty)$, for every $\overline{x}\in \Sigma$,~be~defined~by\vspace{-0.5mm}
    \begin{align*}
        p(\overline{x})\coloneqq \begin{cases}
            p_1&\text{ if }\overline{x}\le \zeta\,,\\
            p_2&\text{ if }\overline{x}>\zeta\,.
        \end{cases}\\[-6mm]\notag
    \end{align*} 
    Recall that we seek $v\in W^{1,p(\cdot)}_0(\Sigma)$ with $\vert \partial_{\overline{x}}v\vert^{p(\cdot)-2}\partial_{\overline{x}}v\in W^{1,1}(\Sigma)$ solving \eqref{eq:symmetric.2},~which, 
    as before, implies the existence of constants $a_1,a_2,b_1,b_2\hspace{-0.1em}\in\hspace{-0.1em} \mathbb{R}$ such that for every $\overline{x}\hspace{-0.1em}\in\hspace{-0.1em} \Sigma$,~we~have~that\vspace{-0.5mm}
    \begin{align}\label{eq:nonlin_a.1}
        v(\overline{x})=\begin{cases}
            -\frac{1}{(p_1)'}\vert a_1-\overline{x}\vert^{(p_1)'}+b_1&\text{ if }\overline{x}\le \zeta\,,\\
            -\frac{1}{(p_2)'}\vert a_2-\overline{x}\vert^{(p_2)'}+b_2&\text{ if }\overline{x}> \zeta\,.
        \end{cases}\\[-6mm]\notag
    \end{align}
    Next, due to $\vert \partial_{\overline{x}}v\vert^{p(\cdot)-2}\partial_{\overline{x}}v\in W^{1,1}(\Sigma)$, we have that $\vert \partial_{\overline{x}}v\vert^{p(\cdot)-2}\partial_{\overline{x}}v\in C^0(\overline{\Sigma})$ and, consequently, $\vert \partial_{\overline{x}}v(\zeta)\vert^{p_1-2}\partial_{\overline{x}}v(\zeta)\hspace{-0.1em}=\hspace{-0.1em}\vert \partial_{\overline{x}}v(\zeta)\vert^{p_2-2}\partial_{\overline{x}}v(\zeta)$, 
    so that $a\hspace{-0.1em}\coloneqq\hspace{-0.1em} a_1\hspace{-0.1em}=\hspace{-0.1em}a_2$. Then, from $v(\pm r)\hspace{-0.1em}=\hspace{-0.1em}0$,~we~infer~that\vspace{-0.5mm}
    \begin{align*}
        b_1&\coloneqq \smash{\tfrac{1}{(p_1)'}}\vert r+a\vert^{(p_1)'}\,,\\
        b_2&\coloneqq \smash{\tfrac{1}{(p_2)'}}\vert r-a\vert^{(p_2)'}\,.\\[-6mm]\notag
    \end{align*} 
    Moreover, due to $v\in \smash{W^{1,p(\cdot)}_0(\Sigma)}$, we have that $v\in C^0(\overline{\Sigma})$, so that from \eqref{eq:nonlin_a.1}, we infer that\vspace{-0.5mm}
    \begin{align}\label{eq:nonlin_a}
        -\tfrac{1}{(p_1)'}\vert \zeta-a\vert^{(p_1)'}+\tfrac{1}{(p_1)'}\vert r+a\vert^{(p_1)'}= -\tfrac{1}{(p_2)'}\vert \zeta-a\vert^{(p_2)'}+\tfrac{1}{(p_2)'}\vert r-a\vert^{(p_2)'}\,.\\[-6mm]\notag
    \end{align}
    In the parameter regimes considered below, this scalar equation has a unique solution $a=a_1=a_2$,
    which can be found numerically.  

     \textcolor{black}{Two representative velocity profiles and power-law indices are displayed in Figure~\ref{fig:explicit_noneven}.}
    
\begin{figure}[H]
    \centering
    \includegraphics[width=0.5\linewidth]{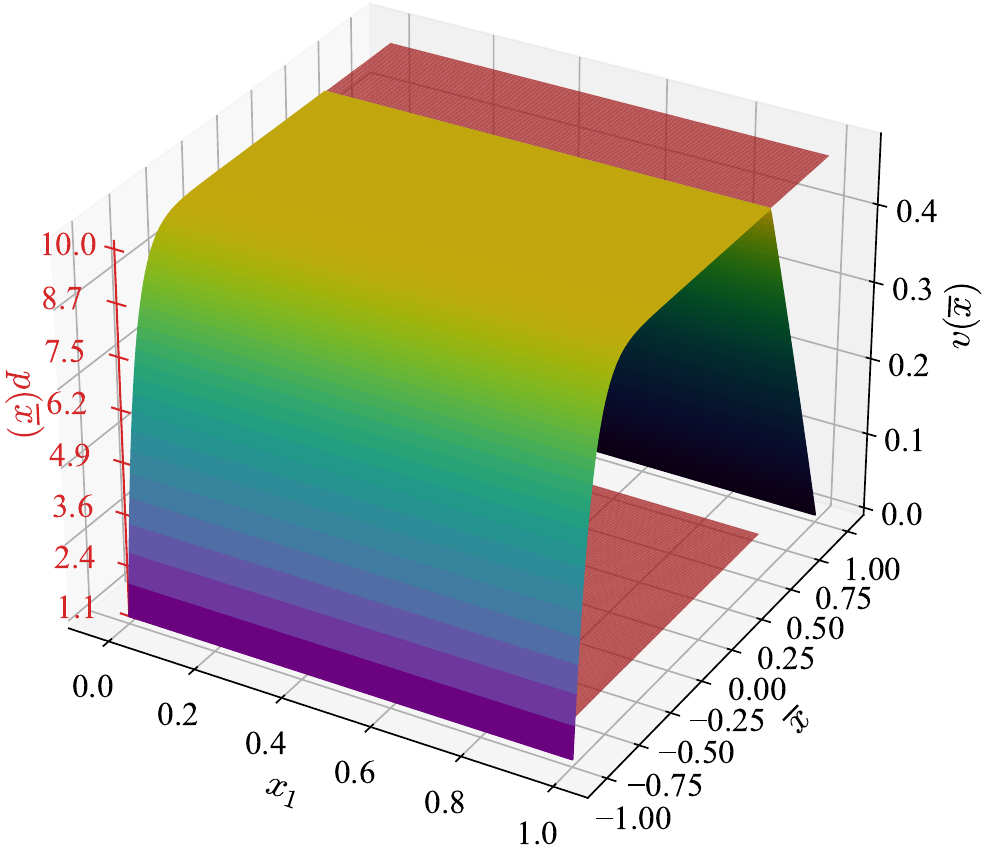}\includegraphics[width=0.5\linewidth]{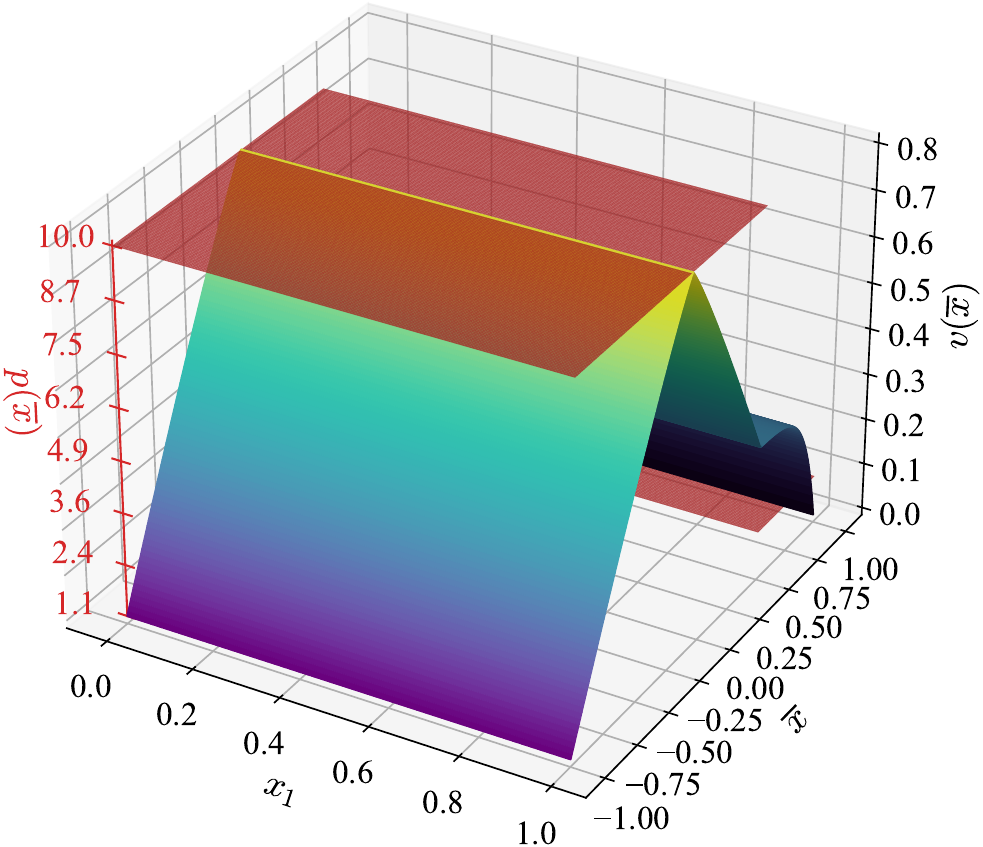}
    \caption{\color{black}Surface plots associated with the explicit non-even solution of the reduced
steady problem \eqref{eq:symmetric.2} with $\Gamma=-1$: the rescaled velocity in the
$\mathbb{R}\mathbf e_1$-direction $v\colon \Sigma\to \mathbb R$ given by
\eqref{eq:nonlin_a.1} (\textcolor{byzantium}{v}\textcolor{denim}{i}\textcolor{shamrockgreen}{r}\textcolor{Goldenrod}{i}\textcolor{byzantium}{d}\textcolor{denim}{i}\textcolor{shamrockgreen}{s} colormap), with the parameter $a$ determined by \eqref{eq:nonlin_a}, 
and the piece-wise constant non-even power-law index
$p\colon \Sigma\to (1,+\infty)$ (\textcolor{black}{red}), each restricted to
$(0,1)\times\Sigma\subseteq \Omega$. Here, $\zeta=\frac{1}{2}$:
\textit{left:} $p_1=1.1$ and $p_2=10$;
\textit{right:} $p_1=10$ and $p_2=1.1$.}
\label{fig:explicit_noneven}
\end{figure}\newpage

\section{Numerical Experiments}\label{sec:experiments}\vspace{-1mm}

    \hspace{5mm}\textcolor{black}{In this section, we illustrate the theoretical findings in the \emph{`assigned-flux'}
    setting (\textit{i.e.}, for prescribed flow rate) by means of three numerical experiments:} 
  \begin{itemize}[noitemsep,topsep=2pt,leftmargin=!,labelwidth=\widthof{2.},font=\itshape]
      \item[1.] \textit{Error decay rates for a non-trivially time-periodic solution:} We measure error decay rates for a non-\hspace{-0.1mm}trivially \hspace{-0.1mm}time-periodic \hspace{-0.1mm}solution \hspace{-0.1mm}in \hspace{-0.1mm}the \hspace{-0.1mm}linear \hspace{-0.1mm}case, \hspace{-0.1mm}\textit{i.e.},  \hspace{-0.1mm}Womersley-type~\hspace{-0.1mm}\mbox{solution}~(\textit{cf}.~\hspace{-0.1mm}\mbox{\cite{Hagen1839,Poiseuille1846}});

    \item[2.] \textit{Error decay rates for trivially time-periodic solutions:} We measure error decay rates~for~trivially time-periodic  solutions 
    in the non-linear case with possibly position-dependent~\mbox{power-law}~\mbox{index};
      
      \item[3.] \textit{Comparison with a direct $d$-dimensional approximation of \eqref{eq:periodic_pNSE}:} We compare the
      $(d-1)$-dimen\-sional approximation of \eqref{eq:periodic_pLaplace} by means of the 
      discrete variational formulation (in the sense of Def.~\ref{scheme:weak_solution}) with an analogous, but direct  $d$-dimensional approximation of \eqref{eq:periodic_pNSE}.
  \end{itemize}
  \hspace{5mm}In doing so, in order to reduce the computational effort, we restrict to the case $d=2$, \textit{i.e.}, the cross-section $\Sigma\subseteq \mathbb{R}^1$ is an interval and the infinite pipe $\Omega=\mathbb{R}\times \Sigma$ is an infinite strip.\linebreak 
All experiments were conducted using the finite element software \texttt{FEniCS} (version~2019.1.0,~\textit{cf}.~\cite{fenics}); \textcolor{black}{in all reduced computations, we employ $\smash{\widehat{V}}_h=\mathbb{P}^1_c(\mathcal T_h)$ and, therefore, $V_h=\mathbb{P}^1_c(\mathcal T_h)\cap
W^{1,1}_0(\Sigma)$.}
All discrete (variational) solutions are computed by means of the {Picard} iteration in Algorithm~\ref{alg:picard-iteration} (with initial guess $\widetilde{v}_h^0=0\in B_h^\tau$, tolerance $\texttt{tol}_{\textup{stop}}=1.0\times 10^{-12}$, maximum iterations $\texttt{K}_{\textup{max}}=100$, and error norm $\|\cdot\|_{V_h}=\|\cdot\|_{\Sigma}$). The series of non-linear systems emerging in the temporally iterative computation of the discrete solution of the discrete initial value problem
\eqref{lem:scheme:well_posedness.0}--\eqref{lem:scheme:well_posedness.1.2-new} is approximated via a semi-implicit discretized $L^2$-gradient flow
(deemed to terminate if a successive iterate difference criterion with absolute tolerance
$\texttt{tol}_{\textup{abs}} \coloneqq 1\times10^{-8}$ is satisfied and~with~a~sparse direct
solver from \texttt{MUMPS} (version~5.5.0,~\textit{cf}.~\cite{mumps}) as linear solver~for~the~linearized~systems).\enlargethispage{14mm}\vspace{-2mm}

  \subsection{Error decay rates for a non-trivially time-periodic solution}\vspace{-0.5mm}

  \hspace{5mm}In this subsection, we measure error decay rates for the approximation~of~a~non-trivially time-periodic solution of the 2D problem \eqref{eq:periodic_pLaplace} by means of the fully-discrete finite-difference/-element discretization \eqref{scheme:weak_solution.1}--\eqref{scheme:weak_solution.3.2}.

      Since constructing manufactured solutions for time-periodic problems, including both  the  2D \textit{`\textcolor{black}{assigned flux}'} problem \eqref{eq:periodic_pLaplace} and  the 2D \textit{`\textcolor{black}{assigned pressure}'} problem \eqref{eq:periodic_pLaplace_2}, is significantly more demanding than for initial value problems, we restrict to the case $p(\cdot)\equiv 2$. More precisely, we assume that the (planar) stress vector $\mathbf{s}\colon \mathbb{R}^1\to \mathbb{R}^1$ is position-independent and, for every $\mathbf{a}\in \mathbb{R}^1$, defined by\vspace{-1.5mm}
  \begin{align*}
      \mathbf{s}(\mathbf{a})\coloneqq \mathbf{a}\,.\\[-6mm]\notag
  \end{align*}
  Moreover, the cross-section is given via $\Sigma\coloneqq (-r,r)$, for some radius $r>0$, and we consider a time-periodic flow rate $\alpha\hspace{-0.1em}\in\hspace{-0.1em} W^{1,2}(I)$, where $I\hspace{-0.1em}\coloneqq\hspace{-0.1em}(0,L)$, for the time period $L\hspace{-0.1em}\coloneqq \hspace{-0.1em}2\pi$,~and~an~integer~${\omega\hspace{-0.1em}\in\hspace{-0.1em}\mathbb{N}}$, so that the unique solution to the 2D \textit{`\textcolor{black}{assigned flux}'} problem \eqref{eq:periodic_pLaplace} is a Womersley-type solution\footnote{%
We prescribe $\Gamma\in L^2(I)$ and use the induced flow rate
$\alpha\coloneqq(v,1)_\Sigma\in W^{1,2}(I)$ in the assigned-flux formulation.} (\textit{cf}.\ \cite{Hagen1839,Poiseuille1846}) and,~for~every~$(t,x)\in I\times\Sigma$,  given via (\textit{cf}.\ Figure \ref{fig:hagen_solutions})\footnote{$\mathbf{i}\coloneqq \sqrt{-1}\in \mathbb{C}$ denotes the \emph{imaginary unit}.}\vspace{-0.5mm}
      \begin{align}\label{eq:solution_hagen}
        \begin{aligned} 
          v(t,x)&\coloneqq \textup{Re}\Big[\tfrac{\mathbf{i}\exp( \mathbf{i}t\omega)}{\omega(1+\exp((1+\mathbf{i})\sqrt{2}r\sqrt{\omega})}
         \times\Big(\exp\big(\tfrac{(1+\mathbf{i})\sqrt{\omega}(r-x)}{\sqrt{2}}\big)\\[-0.5mm]&\qquad\quad+\exp\big(\tfrac{(1+\mathbf{i})\sqrt{\omega}(r+x)}{\sqrt{2}}\big)
         -\exp((1+\mathbf{i})\sqrt{2}r\sqrt{\omega})-1
        \Big)
            \Big]
          \,,\\[-1mm]
          \Gamma(t) &\coloneqq -\cos(\omega t)\,.
          \end{aligned}\\[-6mm]\notag
      \end{align}
        Then, for the choices $r\in \{1.0,5.0,10.0\}$ and $\omega=1$, a series of 
     triangulations $\{\mathcal{T}_{h_i}\}_{i=1,\ldots,11}$~of~$\Sigma$, obtained by uniform refinement starting with the initial triangulation $\mathcal{T}_{h_0}\coloneqq \{[-r,0],[0,r]\}$, and a series of 
    partitions $\{\mathcal{I}_{\tau_i}\}_{i=1,\ldots,11}$ and  $\{\mathcal{I}_{\tau_i}^0\}_{i=1,\ldots,11}$ of $I$ and $(-\tau_i,2\pi)$, $i=1,\ldots,11$,~respectively, with step-sizes $\tau_i\coloneqq 2\pi\times 2^{-i}$, $i=1,\ldots,11$,  employing element-wise affine elements (\textit{i.e.}, $\ell_v=1$ in \eqref{def:fe_space}), we compute the \textit{`natural'} error quantities\vspace{-0.5mm}
      \begin{align}\label{eq:error_hagen}
        \left.\begin{aligned} 
        \texttt{err}_{v,i}^{\smash{\scaleto{L^\infty L^2}{5pt}}}&\coloneqq \|v_{h_i}^{\tau_i}-\mathrm{I}_{\tau_i}^0v\|_{L^\infty(I;L^2(\Sigma))}\,,\\[-0.25mm]
          \texttt{err}_{v,i}^{\smash{\scaleto{L^2 W^{1,2}_0}{7pt}}}&\coloneqq \|\nabla v_{h_i}^{\tau_i}-\nabla \mathrm{I}_{\tau_i}^0v\|_{I\times\Sigma}\,,
          \\[-0.25mm]
           \texttt{err}_{\Gamma,i}^{\smash{\scaleto{L^2}{5pt}}}&\coloneqq \|\Gamma^{\tau_i}-\mathrm{I}_{\tau_i}^0\Gamma\|_{I}\,,
           \end{aligned}\quad\right\}\quad i=1,\ldots,11\,.\\[-6mm]\notag
      \end{align}

        In Figure \ref{fig:hagen_errors}\textit{(right column)}, for $r\in \{1.0,5.0,10.0\}$, 
        for all three
error quantities in \eqref{eq:error_hagen}, \textit{i.e.}, $\smash{\texttt{err}_{v,i}^{\smash{\scaleto{L^\infty L^2}{5pt}}}}$, $\smash{\texttt{err}_{v,i}^{\smash{\scaleto{L^2 W^{1,2}_0}{7pt}}}}$, and $ \texttt{err}_{\Gamma,i}^{\smash{\scaleto{L^2}{5pt}}}$, $i=1,\ldots,11$, we report the first-order
error decay rate $\mathcal{O}(\tau_i+h_i)$, $i=1,\ldots,11$. \textcolor{black}{At the same time, the number of Picard iterations remains essentially stable under
refinement after a possible coarse-resolution regime; this behavior is observed throughout
all numerical experiments in this section.}\enlargethispage{2.5mm}

        \begin{figure}[H]
            \centering
            \includegraphics[width=\linewidth]{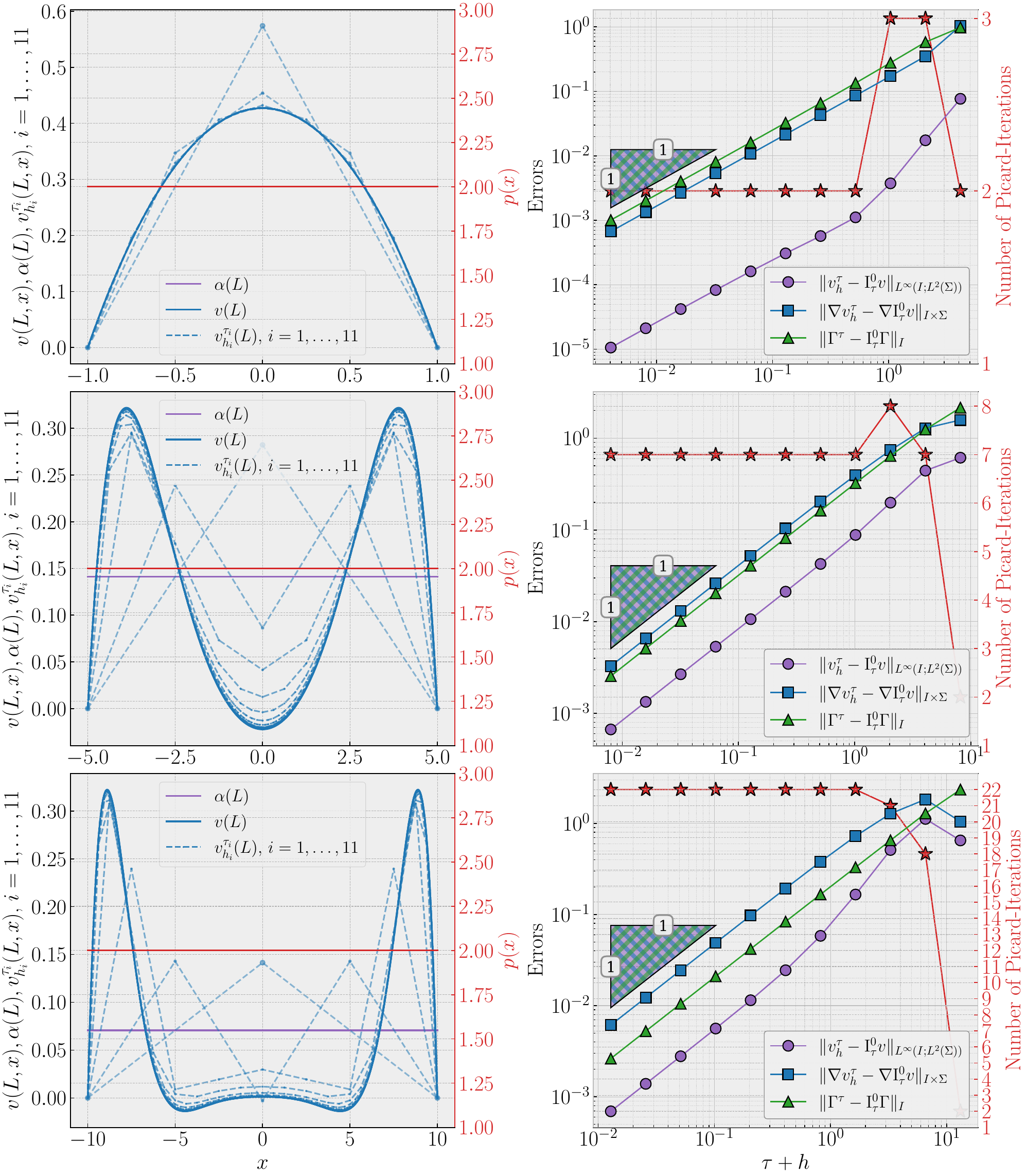}
            \caption{\textit{left column:} line plots of the final/initial flow rate $\alpha(L)=\alpha(0)\in \mathbb{R}$ (\textcolor{byzantium}{purple}), solution $v(L)=v(0)\colon  \Sigma \to \mathbb{R}$ (\textit{cf}.\ \eqref{eq:solution_hagen}) (\textcolor{denim}{blue}), approximations $v_{h_i}^{\tau_i}(L)=v_{h_i}^{\tau_i}(0)\colon \Sigma\to\mathbb{R}$,~${i=1,\ldots,11}$ (\textcolor{denim}{dashed blue}), and power-law index $p\equiv 2$ (\textcolor{black}{red}); \textit{right column:} error plots for the error quantities in \eqref{eq:error_hagen} (\textcolor{byzantium}{purple}/\textcolor{denim}{blue}/\textcolor{shamrockgreen}{green}) and number of Picard iterations (\textcolor{black}{red}) needed in Algorithm~\ref{alg:picard-iteration} to terminate; \textit{top row:} $r=1.0$; \textit{middle row:} $r=5.0$; \textit{bottom row:} $r=10.0$.}
            \label{fig:hagen_errors}
        \end{figure}\newpage

        \begin{figure}[H]
            \centering
            \includegraphics[width=\linewidth]{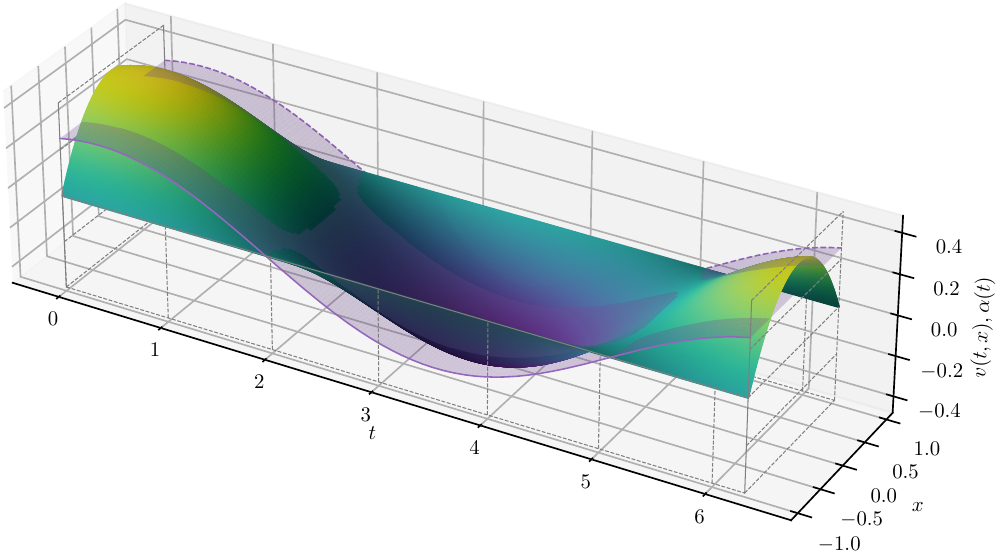}\vspace{-20mm}
            \includegraphics[width=\linewidth]{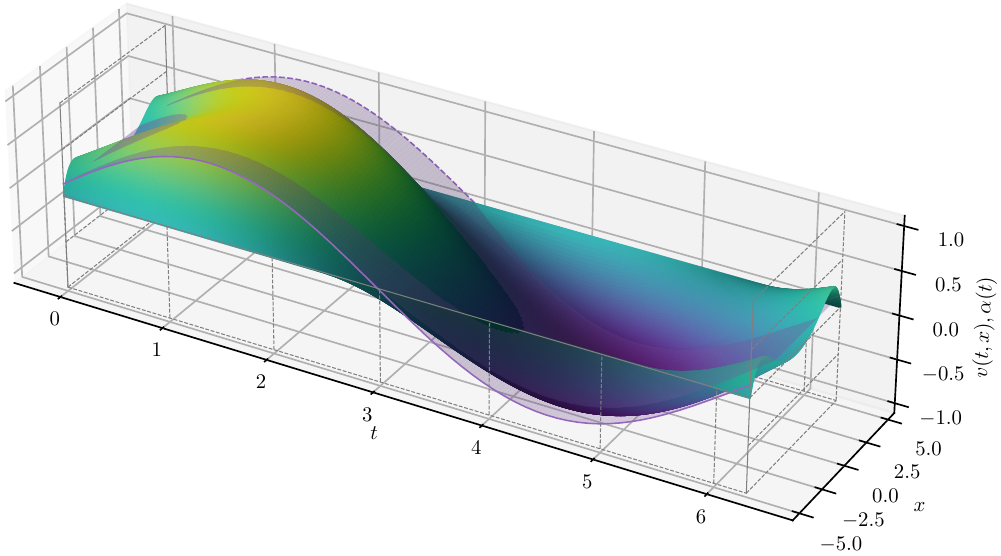}\vspace{-20mm}
            \includegraphics[width=\linewidth]{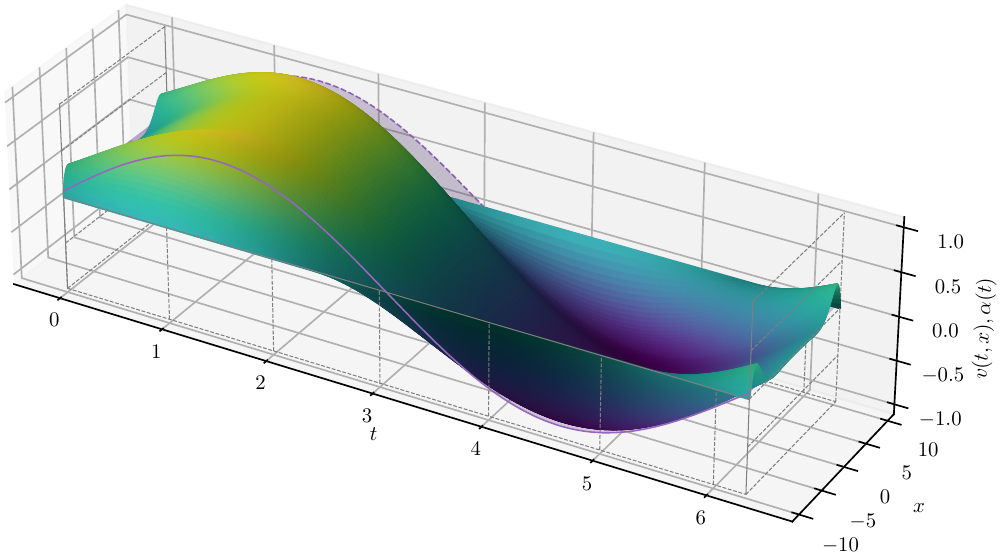}
            \caption{Surface plots of the (constantly in space extended) prescribed $2\pi$-time-periodic flow rate $\alpha\colon I\to \mathbb{R}$ (\textcolor{byzantium}{purple}) and the $2\pi$-time-periodic Womersley-type solution $v\colon  I\times \Sigma\to\mathbb{R}$~(\textit{cf}.~\eqref{eq:solution_hagen}) (\textcolor{byzantium}{v}\textcolor{denim}{i}\textcolor{shamrockgreen}{r}\textcolor{Goldenrod}{i}\textcolor{byzantium}{d}\textcolor{denim}{i}\textcolor{shamrockgreen}{s} colormap): \textit{top:} $r=1.0$; \textit{middle:} $r=5.0$; \textit{bottom:} $r=10.0$.}
            \label{fig:hagen_solutions}
        \end{figure} \newpage

    \subsection{Error decay rates for trivially time-periodic solutions}\enlargethispage{7.5mm}

    \hspace{5mm}The explicit solutions constructed in Section \ref{sec:Exact-Solutions} allow us to consider~at~least~\mbox{time-independent} --and, thus, trivially time-periodic-- solutions also in the case $p\not\equiv 2$ and, in particular,  $p\in \mathcal{P}^\infty(\Sigma)$. More precisely, we assume that the (planar) stress vector $\mathbf{s}\colon \Sigma\times\mathbb{R}^1\to \mathbb{R}^1$ is (possibly) position-dependent and, for a.e.\ $x\in \Sigma$ and every $\mathbf{a}\in \mathbb{R}^1$, defined by\vspace{-0.5mm}
    \begin{align*}
        \mathbf{s}(x,\mathbf{a})\coloneqq \vert \mathbf{a}\vert^{p(x)-2}\mathbf{a}\,.\\[-6mm]\notag
    \end{align*}
    Moreover, the cross-section is given via $\Sigma\coloneqq (-1,1)$ 
    and 
     we consider~a~constant flow rate $\alpha\in \mathbb{R}$, where $I\coloneqq (0,L)$, for the time period $L\coloneqq 1$.~Then,~we~\mbox{distinguish}~three~cases:

    \begin{itemize}[noitemsep,topsep=2pt,leftmargin=!,labelwidth=\widthof{\textit{(Non-even
case)}},font=\itshape]
        \item[(Constant case).] We choose $p\equiv \textup{const}\in  \{\frac{3}{2},\frac{5}{2}\}$ and $\alpha=\frac{3}{4}$ if $p=\frac{5}{2}$,~${\alpha =\frac{1}{2}}$ if ${p=\frac{3}{2}}$, so that the unique trivially time-periodic solution to \eqref{eq:periodic_pLaplace}, for every $x\in\Sigma$, is given via\vspace{-0.5mm}
        \begin{align}\label{def:constant}
            \smash{v(x)\coloneqq \tfrac{1}{p'}\{1-\vert x\vert^{p'}\}}\,;\\[-6mm]\notag
        \end{align} 
        \item[(Even case).] We choose $p\in \mathcal{P}^\infty(\Sigma)$ as in Subsection \ref{sec:Exact-Solutions}(\hyperlink{subsec:symmetric}{a}) with $N=2$, $\zeta_1=-\frac{1}{2}$, ${p_1=\frac{3}{2}}$,~${p_2=\frac{5}{2}}$, and $\alpha\approx 0.586868$, so that the unique trivially time-periodic solution~to~\eqref{eq:periodic_pLaplace}, 
        as in Subsection~\ref{sec:Exact-Solutions}(\hyperlink{subsec:symmetric}{a}), for every $x\in \Sigma$, is given via\vspace{-0.5mm}
        \begin{align}\label{def:sym}
         \hspace{-2.5mm}  v(x)\coloneqq \begin{cases}
                \frac{1}{(p_1)'}\{1-\vert x\vert^{(p_1)'}\}&\textup{ if }\vert x\vert\ge \vert\zeta_1\vert\,,  \\
                \frac{1}{(p_2)'}\{\vert \zeta_1\vert^{(p_2)'}-\vert x\vert^{(p_2)'}\}+\frac{1}{(p_1)'}\{1-\vert\zeta_1\vert^{(p_1)'}\}&\text{ if } \vert x\vert< \vert\zeta_1\vert \,;
            \end{cases}\\[-6mm]\notag
        \end{align}
        
        \item[(Non-even case).] We choose $p\in \mathcal{P}^\infty(\Sigma)$ as in Subsection \ref{sec:Exact-Solutions}(\hyperlink{subsec:nonsymmetric}{b}) with $\zeta=\frac{1}{2}$, $p_1=\frac{5}{2}$, $p_2=\frac{3}{2}$, and $\alpha\approx 0.684009$, so that the unique trivially time-periodic solution~to~\eqref{eq:periodic_pLaplace}, as in Subsection~\ref{sec:Exact-Solutions}(\hyperlink{subsec:nonsymmetric}{b}), \textit{i.e.}, for $a\approx-0.049547$ solving \eqref{eq:nonlin_a}, is given via\vspace{-0.5mm}
        \begin{align}\label{def:nonsym}
            v(x)&\coloneqq \begin{cases}
                \frac{1}{(p_1)'}\{\vert 1+a\vert^{(p_1)'}-\vert a- x\vert^{(p_1)'}\}&\text{ if }x\leq \zeta\,,\\
                \frac{1}{(p_2)'}\{\vert 1-a\vert^{(p_2)'}-\vert a-x\vert^{(p_2)'}\}&\text{ if }x>\zeta\,.
            \end{cases}\\[-6mm]\notag
        \end{align} 
    \end{itemize}
    In all three cases, we have that $\Gamma\equiv -1$.

    Then, for a series of triangulations $\{\mathcal{T}_{h_i}\}_{i=1,\ldots,9}$ of $\Sigma$ obtained by uniform refinement starting with the initial triangulation $\mathcal{T}_{h_0}\coloneqq \{[-1,0],[0,1]\}$,~and~a~series of 
    partitions $\{\mathcal{I}_{\tau_i}\}_{i=1,\ldots,9}$ and  $\{\mathcal{I}_{\tau_i}^0\}_{i=1,\ldots,9}$ of $I$ and $(-\tau_i,1)$, $i=1,\ldots,9$,~respectively, with step-sizes $\tau_i\coloneqq 2^{-i}$, $i=1,\ldots,9$, employing element-wise affine elements (\textit{i.e.}, $\ell_v=1$ in \eqref{def:fe_space}), we compute the~\textit{`natural'}~error quantities\vspace{-0.5mm}
      \begin{align}\label{eq:error_stationary}
        \left.\begin{aligned} 
        \texttt{err}_{v,i}^{\smash{\scaleto{L^\infty L^2}{5pt}}}&\coloneqq \|v_{h_i}^{\tau_i}-\mathrm{I}_{\tau_i}^0v\|_{L^\infty(I;L^2(\Sigma))}\,,\\
          \texttt{err}_{v,i}^{\smash{\scaleto{L^2 \mathbf{f}(\cdot,W^{\smash{1,p(\cdot)}}_0)}{7pt}}}&\coloneqq \|\mathbf{f}(\cdot,\nabla v_{h_i}^{\tau_i})-\mathbf{f}(\cdot,\nabla \mathrm{I}_{\tau_i}^0v)\|_{I\times\Sigma}\,,
          \\[-0.25mm]
           \texttt{err}_{\Gamma,i}^{\smash{\scaleto{(\varphi_{\vert \nabla v\vert})^*}{5pt}}}&\coloneqq \|
            (\varphi_{\vert \nabla v\vert})^*(\cdot,\vert  \Gamma^{\tau_i}-\mathrm{I}_{\tau_i}^0\Gamma\vert)
          \|_{1,I\times \textcolor{black}{\Sigma}}^{\textcolor{black}{1/2}}\,,
           \end{aligned}\quad\right\}\quad i=1,\ldots,9\,,\\[-6mm]\notag
      \end{align}
      where $\mathbf{f}\colon \Sigma\times \mathbb{R}^1\to \mathbb{R}^1$ and $(\varphi_{a})^*\colon \Sigma\times [0,+\infty)\to [0,+\infty)$, for a.e.\  $x\in \Sigma$, every $\mathbf{a}\in \mathbb{R}^1$, and every $a,t\ge 0$, respectively, are defined by $\mathbf{f}(x,\mathbf{a})\coloneqq \vert \mathbf{a}\vert^{\smash{\frac{p(x)-2}{2}}}\mathbf{a}$ and $(\varphi_{a})^*(x,t)\coloneqq (a^{p(x)-1}+t)^{\smash{p'(x)-2}}t^2$.

    We make the following observations in the three cases mentioned above:
    \begin{itemize}[noitemsep,topsep=2pt,leftmargin=!,labelwidth=\widthof{$\bullet$}]\color{black}
            \item[$\bullet$] \emph{Observations in the constant case.} 
        In Figure \ref{fig:stationary_const}\textit{(right column)}, for $p\in \{2.5,1.5\}$,~for~the~errors
        $\smash{\texttt{err}_{v,i}^{\smash{\scaleto{L^2 \mathbf{f}(\cdot,W^{\smash{1,p(\cdot)}}_0)}{7pt}}}}$, $i=1,\ldots,9$,
        we report the  quasi-optimal error decay rate $\mathcal{O}(\tau_i+h_i)$,~${i=1,\ldots,9}$, while for the errors 
        $\smash{\texttt{err}_{v,i}^{\smash{\scaleto{L^\infty L^2}{5pt}}}}$, $i=1,\ldots,9$, and $\smash{\texttt{err}_{\Gamma,i}^{\scaleto{(\varphi_{\vert \nabla v\vert})^*}{5pt}}}$, $i=1,\ldots,9$, we report the increased error decay rate $\mathcal{O}((\tau_i+h_i)^2)$, $i=1,\ldots,9$, which might be traced back to a superconvergence due to the time-independent~flow rate;

        \item[$\bullet$] \emph{Observations in the even case.} 
        In Figure \ref{fig:stationary_sym}\textit{(right)}, 
        for the errors $\smash{\texttt{err}_{v,i}^{\smash{\scaleto{L^2 \mathbf{f}(\cdot,W^{\smash{1,p(\cdot)}}_0)}{7pt}}}}$, $i=1,\ldots,9$,         
         we report the quasi-optimal error decay rate $\mathcal{O}(\tau_i+h_i)$, $i=1,\ldots,9$, while for the errors $\smash{\texttt{err}_{v,i}^{\smash{\scaleto{L^\infty L^2}{5pt}}}}$, $i=1,\ldots,9$, 
        and $\smash{\texttt{err}_{\Gamma,i}^{\scaleto{(\varphi_{\vert \nabla v\vert})^*}{5pt}}}$, $i=1,\ldots,9$, we report the increased error decay~rate~$\mathcal{O}((\tau_i+h_i)^2)$,~$i=1,\ldots,9$;

        \item[$\bullet$] \emph{Observations in the non-even case.} 
        In Figure \ref{fig:stationary_nonsym}\textit{(right)}, for all three error quantities in
        \eqref{eq:error_stationary}, we report
the quasi-optimal error decay rate
$\mathcal{O}(\tau_i+h_i)$, $i=1,\ldots,9$.
    \end{itemize}

    \begin{figure}[H]
        \centering
       \includegraphics[width=\linewidth]{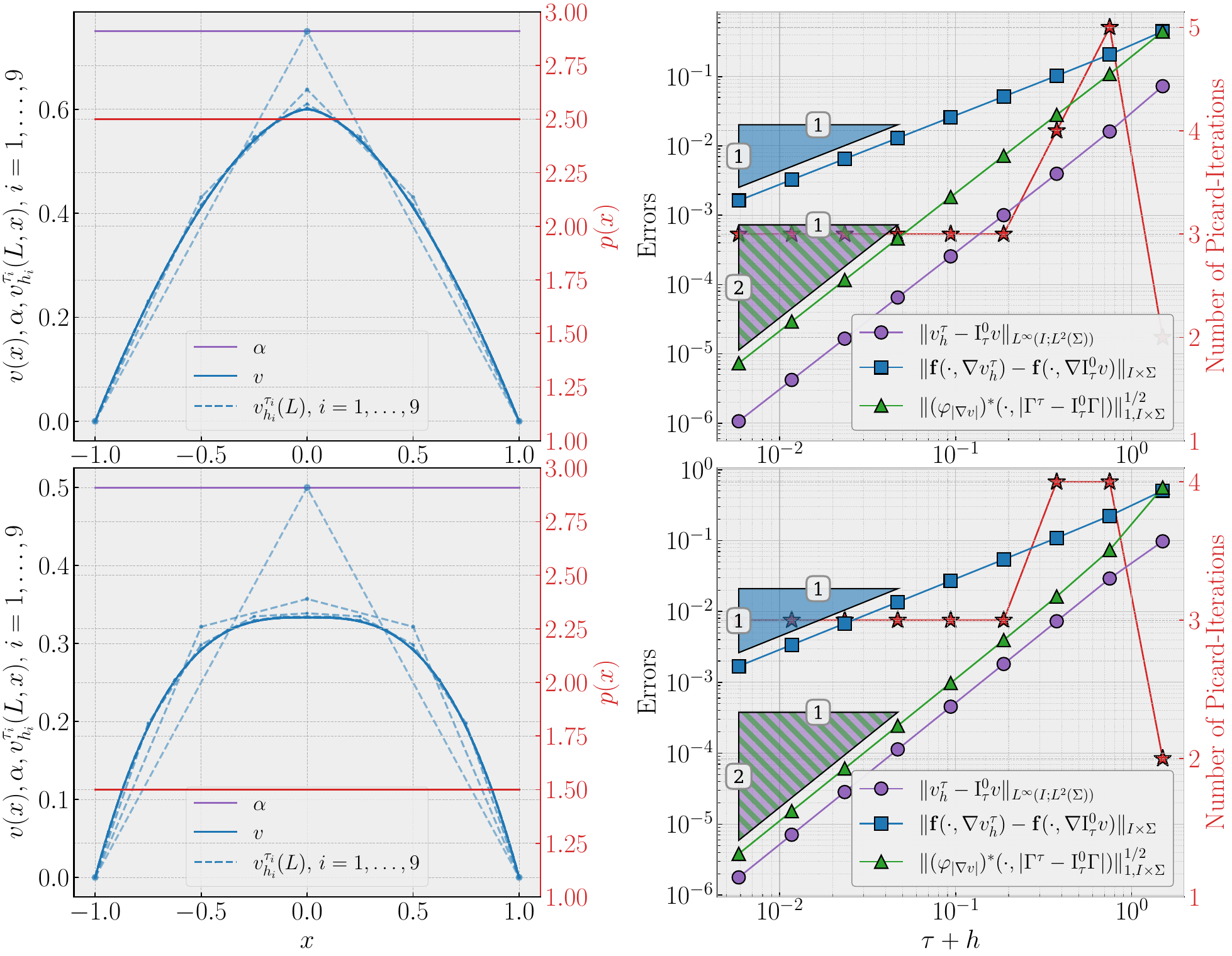} 
        \caption{\textit{left column:} line plots of the constant flow rate $\alpha\in \mathbb{R}$ (\textcolor{byzantium}{purple}),~solution~${v\colon \Sigma\to \mathbb{R}}$ (\textit{cf}.~\eqref{def:constant})     (\textcolor{denim}{blue}), \hspace{-0.15mm}approximations \hspace{-0.15mm}$v_{h_i}^{\tau_i}(L)\colon \hspace{-0.175em}\Sigma\hspace{-0.2em}\to \hspace{-0.2em}\mathbb{R}$,~\hspace{-0.15mm}${i\hspace{-0.175em}=\hspace{-0.175em}1,\ldots,9}$,~\hspace{-0.15mm}(\textcolor{denim}{dashed~blue}),~\hspace{-0.15mm}and~\hspace{-0.15mm}\mbox{power-law}~\hspace{-0.15mm}\mbox{index}~\hspace{-0.15mm}$p$ (\textcolor{black}{red}); \textit{right column:} error plots for the error quantities~in~\eqref{eq:error_stationary}~(\textcolor{byzantium}{purple}/\textcolor{denim}{blue}/\textcolor{shamrockgreen}{green}) and number of Picard iterations (\textcolor{black}{red}) needed in Algorithm \ref{alg:picard-iteration} to terminate; \textit{top row:} $p=2.5$;~\textit{bottom~row:}~$p=1.5$.}
        \label{fig:stationary_const}
    \end{figure}\vspace{-5mm}\enlargethispage{2.5mm}

    \begin{figure}[H]
        \centering
       \includegraphics[width=\linewidth]{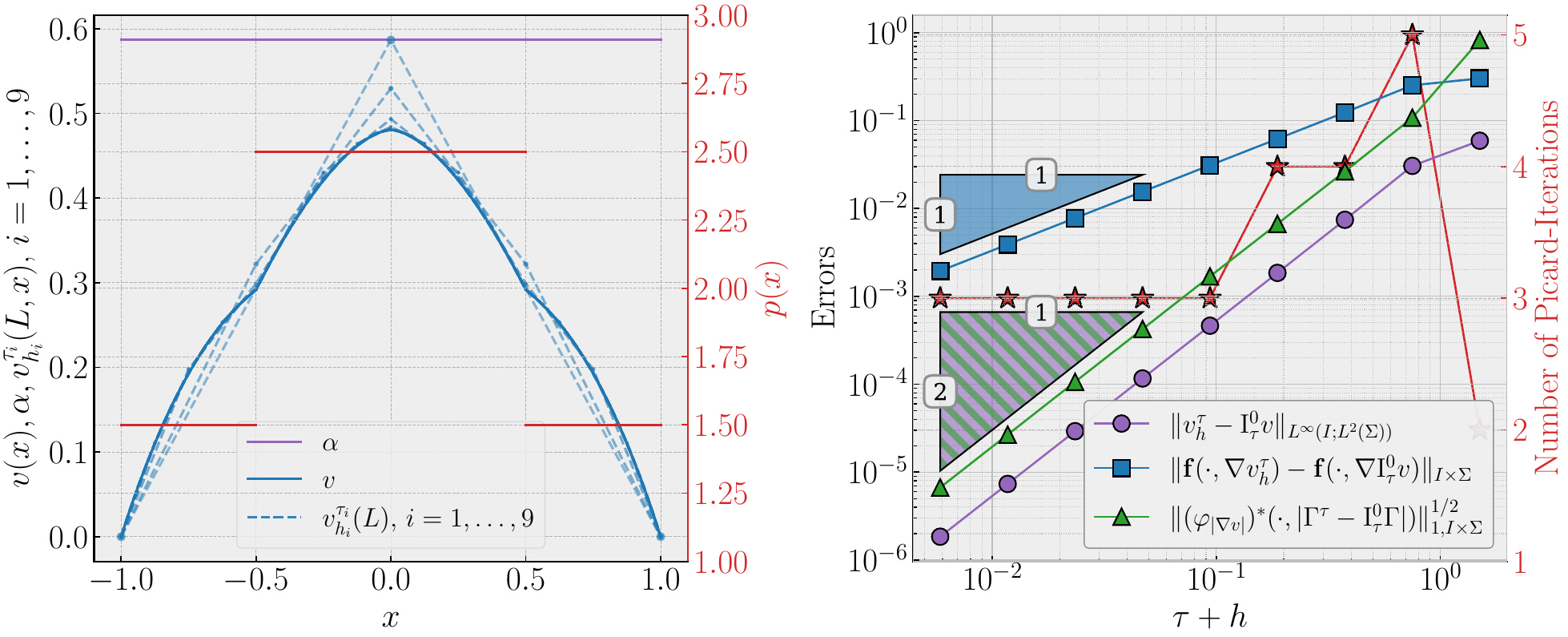}
        \caption{\textit{left:} line plots of the constant flow rate  $\alpha\in \mathbb{R}$ (\textcolor{byzantium}{purple}),  solution $v\colon \Sigma\to \mathbb{R}$~(\textit{cf}.~\eqref{def:sym}) (\textcolor{denim}{blue}), approximations $v_{h_i}^{\tau_i}(L)\colon \Sigma\to \mathbb{R}$, $i=1,\ldots,9$ (\textcolor{denim}{dashed blue}), and power-law index $p\colon \Sigma\to (1,+\infty)$ (\textcolor{black}{red}); \textit{right:} error plots for the error quantities in \eqref{eq:error_stationary} (\textcolor{byzantium}{purple}/\textcolor{denim}{blue}/\textcolor{shamrockgreen}{green}) and number of Picard iterations (\textcolor{black}{red}) needed in Algorithm \ref{alg:picard-iteration} to terminate.}
        \label{fig:stationary_sym}
    \end{figure} 

    \begin{figure}[H]
        \centering
       \includegraphics[width=\linewidth]{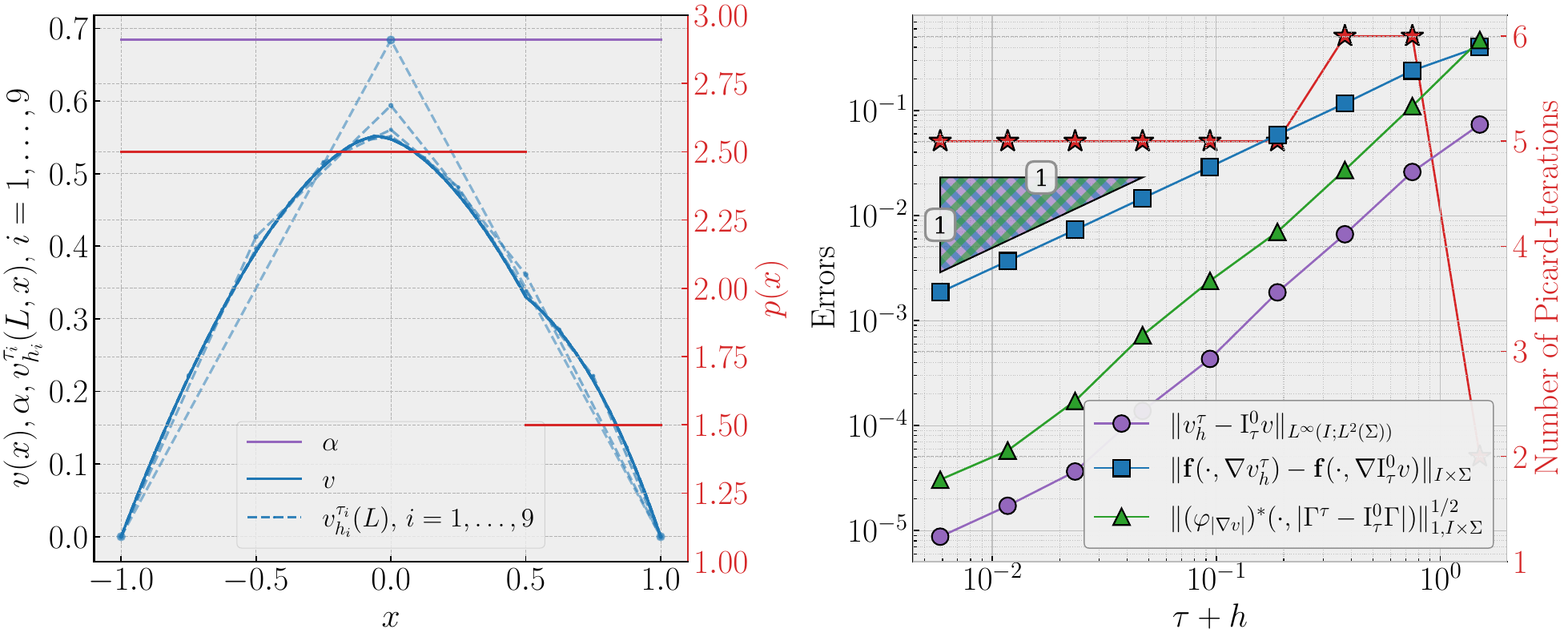}
        \caption{\textit{left:} line plots of the constant flow rate $\alpha\in \mathbb{R}$ (\textcolor{byzantium}{purple}),  solution $v\colon \Sigma\to \mathbb{R}$~(\textit{cf}.~\eqref{def:nonsym}) (\textcolor{denim}{blue}), approximations $v_{h_i}^{\tau_i}(L)\colon \Sigma\to \mathbb{R}$, $i=1,\ldots,9$ (\textcolor{denim}{dashed blue}), and power-law index $p\colon \Sigma\to (1,+\infty)$ (\textcolor{black}{red}); \textit{right:} error plots for the error quantities in \eqref{eq:error_stationary} (\textcolor{byzantium}{purple}/\textcolor{denim}{blue}/\textcolor{shamrockgreen}{green}) and number of Picard iterations (\textcolor{black}{red}) needed in Algorithm \ref{alg:picard-iteration} to terminate.}
        \label{fig:stationary_nonsym}
    \end{figure}
 
    \subsection{Comparison with a direct $d$-dimensional approximation of \eqref{eq:periodic_pNSE}}\enlargethispage{7.5mm}

    \hspace{5mm}In \hspace{-0.1mm}order \hspace{-0.1mm}to \hspace{-0.1mm}compare \hspace{-0.1mm}experimentally \hspace{-0.1mm}the \hspace{-0.1mm}reduced \hspace{-0.1mm}1D \hspace{-0.1mm}problem \hspace{-0.1mm}\eqref{eq:periodic_pLaplace} \hspace{-0.1mm}with \hspace{-0.1mm}the~\hspace{-0.1mm}full~\hspace{-0.1mm}2D~\hspace{-0.1mm}\mbox{problem}~\hspace{-0.1mm}\eqref{eq:periodic_pNSE}, we truncate the infinite pipe to the strip
$\Omega\coloneqq (0,x_{\max})\times \Sigma$, where $x_{\max}\hspace{-0.1em}>\hspace{-0.1em}0$ and $\Sigma\hspace{-0.1em}\coloneqq\hspace{-0.1em} (-r,r)$,~${r\hspace{-0.1em}>\hspace{-0.1em}0}$, and interpret the 
system as a variational truncation whose artificial cross-sections are treated through the weak formulation.
For $I\coloneqq(0,L)$, $L>0$, and a given $L$-time-periodic~flow~rate $\alpha\colon I\to\mathbb{R}$, we seek a velocity field $\mathbf{v}\colon I\times\Omega\to\mathbb{R}^2$ and a kinematic~pressure~$\pi\colon I\times\Omega\to\mathbb{R}$~such~that\vspace{-4.5mm}
    \begin{subequations}\label{eq:periodic_pNSE_truncated}
    \begin{alignat}{2}
    \partial_t \mathbf{v}- \textup{div}\,\mathbf{S}(\cdot,\mathbf{D} \mathbf{v})+\textup{div}(\mathbf{v}\otimes \mathbf{v})+\nabla \pi&=\mathbf{0}_2&&\quad \text{ in }I\times \Omega\,,\\
    \textup{div}\,\mathbf{v}&=0&&\quad\text{ in } I\times\Omega\,,\\
    (\mathbf{v},\mathbf{n}_{\Sigma_k})_{\Sigma_k}&=\alpha&&\quad\text{ in } I\,,\;k=1,2\,,\\
    \mathbf{v}(\cdot,\pm r)&=\mathbf{0}_2 &&\quad\text{ on } I\times(0,x_{\max})\,,\\
    \mathbf{v}(0)=\mathbf{v}(L)\,,\;\pi(0)&=\pi(L)&&\quad\text{ in }\Omega\,,
\end{alignat}
\end{subequations}
where the \emph{inflow} and \emph{outflow cross-sections} are given via
    \begin{align}\label{eq:inflow_outflow}
        \Sigma_k \coloneqq\{x_{\max}(k-1)\}\times \Sigma\,,\quad k=1,2\,,
    \end{align}
    respectively, with unit-length vector fields $\mathbf{n}_{\Sigma_k}\coloneqq \mathbf{e}_1\colon \Sigma_k\to \mathbb{S}^1$, $k=1,2$, (\textit{cf}.\ Figure \ref{fig:2Dto1D}).\vspace{-1mm}

    \begin{figure}[H]
        \centering

 
\tikzset{
pattern size/.store in=\mcSize, 
pattern size = 5pt,
pattern thickness/.store in=\mcThickness, 
pattern thickness = 0.3pt,
pattern radius/.store in=\mcRadius, 
pattern radius = 1pt}
\makeatletter
\pgfutil@ifundefined{pgf@pattern@name@_q6gi73xbs}{
\pgfdeclarepatternformonly[\mcThickness,\mcSize]{_q6gi73xbs}
{\pgfqpoint{0pt}{0pt}}
{\pgfpoint{\mcSize+\mcThickness}{\mcSize+\mcThickness}}
{\pgfpoint{\mcSize}{\mcSize}}
{
\pgfsetcolor{\tikz@pattern@color}
\pgfsetlinewidth{\mcThickness}
\pgfpathmoveto{\pgfqpoint{0pt}{0pt}}
\pgfpathlineto{\pgfpoint{\mcSize+\mcThickness}{\mcSize+\mcThickness}}
\pgfusepath{stroke}
}}
\makeatother

 
\tikzset{
pattern size/.store in=\mcSize, 
pattern size = 5pt,
pattern thickness/.store in=\mcThickness, 
pattern thickness = 0.3pt,
pattern radius/.store in=\mcRadius, 
pattern radius = 1pt}
\makeatletter
\pgfutil@ifundefined{pgf@pattern@name@_c5fhjx29z}{
\pgfdeclarepatternformonly[\mcThickness,\mcSize]{_c5fhjx29z}
{\pgfqpoint{0pt}{0pt}}
{\pgfpoint{\mcSize+\mcThickness}{\mcSize+\mcThickness}}
{\pgfpoint{\mcSize}{\mcSize}}
{
\pgfsetcolor{\tikz@pattern@color}
\pgfsetlinewidth{\mcThickness}
\pgfpathmoveto{\pgfqpoint{0pt}{0pt}}
\pgfpathlineto{\pgfpoint{\mcSize+\mcThickness}{\mcSize+\mcThickness}}
\pgfusepath{stroke}
}}
\makeatother

 
\tikzset{
pattern size/.store in=\mcSize, 
pattern size = 5pt,
pattern thickness/.store in=\mcThickness, 
pattern thickness = 0.3pt,
pattern radius/.store in=\mcRadius, 
pattern radius = 1pt}
\makeatletter
\pgfutil@ifundefined{pgf@pattern@name@_xf3lt236y}{
\pgfdeclarepatternformonly[\mcThickness,\mcSize]{_xf3lt236y}
{\pgfqpoint{0pt}{0pt}}
{\pgfpoint{\mcSize+\mcThickness}{\mcSize+\mcThickness}}
{\pgfpoint{\mcSize}{\mcSize}}
{
\pgfsetcolor{\tikz@pattern@color}
\pgfsetlinewidth{\mcThickness}
\pgfpathmoveto{\pgfqpoint{0pt}{0pt}}
\pgfpathlineto{\pgfpoint{\mcSize+\mcThickness}{\mcSize+\mcThickness}}
\pgfusepath{stroke}
}}
\makeatother
\tikzset{every picture/.style={line width=0.75pt}} 

\begin{tikzpicture}[x=1.125pt,y=1.125pt,yscale=-1,xscale=1]

\draw [color={Red}  ,draw opacity=1 ]   (28.23,60.25) -- (55.5,60.43) ;
\draw [shift={(58.5,60.45)}, rotate = 180.37] [fill={Red}  ,fill opacity=1 ][line width=0.08]  [draw opacity=0] (5.36,-2.57) -- (0,0) -- (5.36,2.57) -- (3.56,0) -- cycle    ;
\draw [color={green}  ,draw opacity=1 ]   (28.23,60.25) -- (28.36,32.75) ;
\draw [shift={(28.38,29.75)}, rotate = 90.28] [fill={green}  ,fill opacity=1 ][line width=0.08]  [draw opacity=0] (5.36,-2.57) -- (0,0) -- (5.36,2.57) -- (3.56,0) -- cycle    ;
\draw  [fill={rgb, 255:red, 0; green, 0; blue, 0 }  ,fill opacity=1 ] (29.02,60.25) .. controls (29.02,59.81) and (28.67,59.45) .. (28.23,59.45) .. controls (27.79,59.45) and (27.43,59.81) .. (27.43,60.25) .. controls (27.43,60.69) and (27.79,61.05) .. (28.23,61.05) .. controls (28.67,61.05) and (29.02,60.69) .. (29.02,60.25) -- cycle ;
\draw  [fill={denim}  ,fill opacity=0.1 ] (74.54,90.4) -- (74.54,30.4) -- (367.38,29.99) -- (367.38,89.99) -- (74.54,90.4) -- cycle ;
\draw [color={gray}  ,draw opacity=1 ]   (74.54,30.4) -- (74.52,24.07) ;
\draw [shift={(74.52,24.07)}, rotate = 89.84] [color={gray}  ,draw opacity=1 ][line width=0.75]    (0,5.59) -- (0,-5.59)   ;
\draw [color={gray}  ,draw opacity=1 ]   (367.5,30.15) -- (367.5,23.82) ;
\draw [shift={(367.5,23.82)}, rotate = 89.84] [color={gray}  ,draw opacity=1 ][line width=0.75]    (0,5.59) -- (0,-5.59)   ;
\draw  [color={byzantium}  ,draw opacity=1 ] (74.88,90.25) .. controls (74.89,94.92) and (77.22,97.25) .. (81.89,97.24) -- (211.26,97.01) .. controls (217.93,97) and (221.26,99.32) .. (221.27,103.99) .. controls (221.26,99.32) and (224.59,96.99) .. (231.26,96.98)(228.26,96.98) -- (360.64,96.75) .. controls (365.31,96.74) and (367.63,94.41) .. (367.62,89.74) ;
\draw  [color={byzantium}  ,draw opacity=1 ] (367.38,89.5) .. controls (367.38,84.83) and (365.05,82.5) .. (360.38,82.5) -- (340.88,82.5) .. controls (334.21,82.5) and (330.88,80.17) .. (330.88,75.5) .. controls (330.88,80.17) and (327.55,82.5) .. (320.88,82.5)(323.88,82.5) -- (301.38,82.5) .. controls (296.71,82.5) and (294.38,84.83) .. (294.38,89.5) ;
\draw  [color={byzantium}  ,draw opacity=1 ] (147.63,90) .. controls (147.6,85.33) and (145.25,83.02) .. (140.58,83.05) -- (121.2,83.18) .. controls (114.53,83.23) and (111.18,80.92) .. (111.15,76.25) .. controls (111.18,80.92) and (107.87,83.27) .. (101.2,83.32)(104.2,83.3) -- (81.83,83.45) .. controls (77.16,83.48) and (74.84,85.83) .. (74.87,90.5) ;
\draw  [draw opacity=0][pattern=_c5fhjx29z,pattern size=3.75pt,pattern thickness=0.75pt,pattern radius=0pt, pattern color={rgb, 255:red, 0; green, 0; blue, 0}] (367.38,30.34) -- (375.13,30.34) -- (375.13,89.99) -- (367.38,89.99) -- cycle ;
\draw  [draw opacity=0][pattern=_xf3lt236y,pattern size=3.75pt,pattern thickness=0.75pt,pattern radius=0pt, pattern color={rgb, 255:red, 0; green, 0; blue, 0}] (66.79,30.4) -- (74.54,30.4) -- (74.54,90.05) -- (66.79,90.05) -- cycle ;
\draw [color={gray}  ,draw opacity=1 ][line width=1.5]    (74.54,30.4) -- (74.54,90.4) ;
\draw [color={gray}  ,draw opacity=1 ][line width=1.5]    (367.38,29.99) -- (367.38,89.99) ;
\draw  [color={shamrockgreen}  ,draw opacity=1, line width=1.5 ][pattern=_q6gi73xbs,pattern size=6pt,pattern thickness=0.75pt,pattern radius=0pt, pattern color={shamrockgreen}] (148.38,30.25) -- (294.21,30.25) -- (294.21,90) -- (148.38,90) -- cycle ;
\draw [color={gray}  ,draw opacity=1 ]   (74.54,60.4) -- (97.63,60.49) ;
\draw [shift={(100.63,60.5)}, rotate = 180.22] [fill={gray}  ,fill opacity=1 ][line width=0.08]  [draw opacity=0] (5.36,-2.57) -- (0,0) -- (5.36,2.57) -- cycle    ;
\draw [color={gray}  ,draw opacity=1 ]   (367.38,59.99) -- (390.46,60.08) ;
\draw [shift={(393.46,60.09)}, rotate = 180.22] [fill={gray}  ,fill opacity=1 ][line width=0.08]  [draw opacity=0] (5.36,-2.57) -- (0,0) -- (5.36,2.57) -- cycle    ;

\draw (54.3,63) node [anchor=north west][inner sep=0.75pt]  [font=\small,color={rgb, 255:red, 0; green, 0; blue, 0 }  ,opacity=1 ]  {$x_{1}$};
\draw (31.5,26.5) node [anchor=north west][inner sep=0.75pt]  [font=\small,color={rgb, 255:red, 0; green, 0; blue, 0 }  ,opacity=1 ]  {$x_{2} =\overline{x}$};
\draw (57.5,14) node [anchor=north west][inner sep=0.75pt]  [color={gray}  ,opacity=1 ] {$\Sigma _{1} =\{0\} \times \Sigma $};
\draw (316,14) node [anchor=north west][inner sep=0.75pt]  [color={gray}  ,opacity=1 ]  {$\{x_{\max}\} \times \Sigma =\Sigma _{2}$};
\draw (213.25,105) node [anchor=north west][inner sep=0.75pt]  [color={byzantium}  ,opacity=1 ]  {$x_{\max}$};
\draw (101,64) node [anchor=north west][inner sep=0.75pt]  [color={byzantium}  ,opacity=1 ]  {$\frac{x_{\max}}{4}$};
\draw (321,64) node [anchor=north west][inner sep=0.75pt]  [color={byzantium}  ,opacity=1 ]  {$\frac{x_{\max}}{4}$};
\draw  [draw opacity=0][fill={white}  ,fill opacity=0.75 ]  (220, 60) circle [x radius= 7.5, y radius= 7.5]   ;
\draw (216,56.5) node [anchor=north west][inner sep=0.75pt]  [color={shamrockgreen}  ,opacity=1 ]  [font =\large] {$\omega $};
\draw (90,48.5) node [anchor=north west][inner sep=0.75pt]  [color={gray}  ,opacity=1 ]  {$\mathbf{n}_{\Sigma _{1}}$};
\draw (382,48.5) node [anchor=north west][inner sep=0.75pt]  [color={gray}  ,opacity=1 ]  {$\mathbf{n}_{\Sigma _{2}}$};
\draw (297,32.5) node [anchor=north west][inner sep=0.75pt]  [color={denim}  ,opacity=1 ]  {$\Omega \hspace{-0.1em}=\hspace{-0.1em} ( 0,x_{\max})\hspace{-0.1em}\times\hspace{-0.1em} \Sigma $};

\end{tikzpicture}
        \caption{Schematic diagram of the strip $\Omega\coloneqq (0,x_{\max})\times \Sigma$ (\textcolor{denim}{blue}) of finite length $x_{\max}>0$, with cross-section $\Sigma\hspace{-0.1em}\coloneqq \hspace{-0.1em}(-r,r)$, $r\hspace{-0.1em}>\hspace{-0.1em}0$, inflow/outflow cross-sections $\Sigma_k \hspace{-0.1em}\coloneqq\hspace{-0.1em}\{x_{\max}(k-1)\}\times \Sigma$,~$ k\hspace{-0.1em}=\hspace{-0.1em}1,2$, (\textcolor{gray}{gray}) with unit-length vector fields $\mathbf{n}_{\Sigma_k}\coloneqq \mathbf{e}_1\colon \Sigma_k\to \mathbb{S}^{d-1}$, $k=1,2$, respectively, and  truncated strip $\omega\coloneqq (\frac{x_{\max}}{4},\frac{3x_{\max}}{4})\times \Sigma$ (\textcolor{shamrockgreen}{dashed green}).}
        \label{fig:2Dto1D}
    \end{figure}

    \newpage
    In the full 2D problem \eqref{eq:periodic_pNSE_truncated}, for a given $\overline{x}$-dependent power-law index $p\in \mathcal{P}^{\infty}(\Sigma)$, 
    let~the~stress tensor $\smash{\mathbf{S}\colon \Sigma\times \mathbb{R}^{2\times 2}_{\textup{sym}}\to\mathbb{R}^{2\times 2}_{\textup{sym}}}$, for a.e.\ $\overline{x}\in \Sigma$ and every $\mathbf{A}\in \mathbb{R}^{2\times 2}_{\textup{sym}}$, be defined by\vspace{-0.5mm}
    \begin{align*}
        \smash{\mathbf{S}(\overline{x},\mathbf{A})\coloneqq \vert  \mathbf{A}\vert^{p(\overline{x})-2}\mathbf{A}\,.}
    \end{align*}
    so that, following the reasoning in Section \ref{sec:fully}, the associated (planar) stress vector $\smash{\mathbf{s}\colon \Sigma\times \mathbb{R}^1\to\mathbb{R}^1}$, for a.e.\ $\overline{x}\in \Sigma$ and every $\mathbf{a}\in \mathbb{R}^1$, is given via\vspace{-0.5mm}
    \begin{align*}
        \smash{\mathbf{s}(\overline{x},\mathbf{a})\coloneqq (\tfrac{1}{2}\vert  \mathbf{a}\vert^2)^{\frac{p(\overline{x})-2}{2}}\tfrac{1}{2}\mathbf{a}=2^{-\frac{p(\overline{x})}{2}}\vert\mathbf{a}\vert^{p(\overline{x})-2}\mathbf{a}\,.}
    \end{align*}
    
    If the finite length $x_{\max}>0$ of the strip $\Omega\coloneqq (0,x_{\max})\times \Sigma$ is chosen sufficiently large,~we~expect the velocity field $\mathbf{v}\colon I\times \Omega\to \mathbb{R}^2$ and the  kinematic pressure $\pi\colon I\times \Omega\to \mathbb{R}$,~solving~\eqref{eq:periodic_pNSE_truncated}, to behave according to the definition of a fully-developed flow (\textit{cf}.\ \eqref{eq:restrictions}~with~\eqref{def:x_bar}), at least~in~a region $\omega\subseteq \Omega$ that is sufficiently far away from the inflow~and~outflow~cross-sections~(\textit{cf}.~Figure~\ref{fig:2Dto1D}).\enlargethispage{1.5mm}

    More precisely, for our numerical experiments, we choose strip length $x_{\max}=20$, the cross-section radius $r=\frac{1}{2}$, the time period $L=2\pi$, the $2\pi$-periodic flow-rate $\alpha\coloneqq \cos\colon I\to \mathbb{R}$,~and the power-law index $p\coloneqq 2+\mathrm{id}_{\mathbb{R}}\colon \Sigma\to(1,+\infty)$. As region $\omega\subseteq \Omega$, in which we expect the velocity field and the  kinematic pressure, solving \eqref{eq:periodic_pNSE_truncated}, to behave according to the definition of a fully-developed flow (\textit{cf}.\ \eqref{eq:restrictions} with \eqref{def:x_bar}), we choose $\omega\coloneqq (\frac{x_{\max}}{4},\frac{3x_{\max}}{4})\times \Sigma=(5,15)\times (-\frac{1}{2},\frac{1}{2})$.

     Then, for a series of 
     triangulations $\{\mathcal{T}_{h_i}\}_{i=1,\ldots,8}$
    and $\{\boldsymbol{\mathcal{T}}_{\!h_i}\}_{i=1,\ldots,8}$
     of $\Sigma$ and $\Omega$,~\mbox{respectively}, each obtained by uniform refinement starting with the initial triangulations $\mathcal{T}_{h_0}\coloneqq\{[-\frac{1}{2},0],[0,\frac{1}{2}]\}$ and
     $\boldsymbol{\mathcal{T}}_{\!h_0}\coloneqq \{\textup{conv}\{-\frac{1}{2}\mathbf{e}_2,20\mathbf{e}_1-\frac{1}{2}\mathbf{e}_2,20\mathbf{e}_1+\frac{1}{2}\mathbf{e}_2\},\textup{conv}\{-\frac{1}{2}\mathbf{e}_2,\frac{1}{2}\mathbf{e}_2,20\mathbf{e}_1+\frac{1}{2}\mathbf{e}_2\}\}$,~\mbox{respectively}, and a series of 
    partitions $\{\mathcal{I}_{\tau_i}\}_{i=1,\ldots,8}$ and  $\{\mathcal{I}_{\tau_i}^0\}_{i=1,\ldots,8}$ of $I$ and $(-\tau_i,2\pi)$, $i=1,\ldots,8$,~\mbox{respectively}, with step-sizes $\tau_i\coloneqq 2\pi\times 2^{-i}$, $i=1,\ldots,8$, employing element-wise quadratic elements (\textit{i.e.}, $\ell_v=2$ in \eqref{def:fe_space}) and the (lowest-order) {Taylor}--{Hood}~element (\textit{cf}.~\cite{TaylorHood1973}), \textit{i.e.},\vspace{-0.75mm} 
    \begin{align*}
        \mathbf{V}_{h_i}&\coloneqq \textcolor{black}{(\mathbb{P}^2_c(\boldsymbol{\mathcal{T}}_{\!h_i}))^2\cap \{\mathbf{v}\in (W^{1,1}(\Omega))^2\mid \mathbf{v}=\mathbf{0}_2\text{ a.e.\ on }(0,x_{\mathrm{max}})\times \partial\Sigma\}}\,,\\
        Q_{h_i} &\coloneqq \textcolor{black}{\{\xi_h\in \mathbb{P}^1_c(\boldsymbol{\mathcal{T}}_{\!h_i})\mid (\xi_h,1)_{\Omega}=0\}}\,,
    \end{align*}
    we compute $(\mathbf{v}_{h_i}^{\tau_i},\pi_{h_i}^{\tau_i},\lambda_1^{\tau_i},\lambda_2^{\tau_i})\in \mathbb{P}^0(\mathcal{I}_{\tau_i}^0;\mathbf{V}_{h_i})\times \mathbb{P}^0(\mathcal{I}_{\tau_i};Q_{h_i})\times (\mathbb{P}^0(\mathcal{I}_{\tau_i}))^2$ such that\vspace{-0.75mm}
    \begin{align*}
\mathbf{v}_{h_i}^{\tau_i}(0)=\mathbf{v}_{h_i}^{\tau_i}(L)\quad\text{ a.e.\ in }\Omega\,, 
    \end{align*}
    and for every $(\boldsymbol{\phi}_{h_i}^{\tau_i},\xi_{h_i}^{\tau_i},\eta^{\tau_i}_1,\eta^{\tau_i}_2)\in \mathbb{P}^0(\mathcal{I}_{\tau_i};\mathbf{V}_{h_i})\times \mathbb{P}^0(\mathcal{I}_{\tau_i};Q_{h_i})\times (\mathbb{P}^0(\mathcal{I}_{\tau_i}))^2$, there holds\vspace{-0.75mm}
    \begin{subequations}\label{eq:2Dscheme}
    \begin{align}
        \left.\begin{aligned}
            (\mathrm{d}_{\tau_i} \mathbf{v}_{h_i}^{\tau_i},\boldsymbol{\phi}_{h_i}^{\tau_i})_{I\times \Omega}+(\mathbf{S}(\cdot,\mathbf{D}\mathbf{v}_{h_i}^{\tau_i}),\mathbf{D}\boldsymbol{\phi}_{h_i}^{\tau_i})_{I\times \Omega}&\\
     \textstyle+\sum_{k=1}^2{(\lambda_k^{\tau_i},\boldsymbol{\phi}_{h_i}^{\tau_i}\cdot \mathbf{n}_{\Sigma_k})_{I\times \Sigma_{k}}} -(\pi_{h_i}^{\tau_i},\textup{div}\,\boldsymbol{\phi}_{h_i}^{\tau_i})_{I\times \Omega}\\ -\tfrac{1}{2}(\mathbf{v}_{h_i}^{\tau_i}\otimes \mathbf{v}_{h_i}^{\tau_i},\nabla\boldsymbol{\phi}_{h_i}^{\tau_i})_{I\times \Omega}
        +\tfrac{1}{2}(\boldsymbol{\phi}_{h_i}^{\tau_i}\otimes \mathbf{v}_{h_i}^{\tau_i},\nabla\mathbf{v}_{h_i}^{\tau_i})_{I\times \Omega}
     \end{aligned}\right\}&=0\,,\\
(\textup{div}\,\mathbf{v}_{h_i}^{\tau_i},\xi_{h_i}^{\tau_i})_{I\times \Omega}&=0\,,\\
        (\mathbf{v}_{h_i}^{\tau_i}\cdot \mathbf{n}_{\Sigma_k},\eta_k^{\tau_i})_{I\times \Sigma_k}&=(\alpha^\tau,\eta_k^{\tau_i})_{I}\,,\; k=1,2\,.
    \end{align} 
    \end{subequations}

    Since an explicit solution for \eqref{eq:periodic_pNSE_truncated} is not available, in order to compare the~1D~approximation~of \eqref{eq:periodic_pNSE} by means of the discrete variational formulation (in the sense of Def.\ \ref{scheme:weak_solution}) with the direct 2D approximation \eqref{eq:2Dscheme}, for $i=1,\ldots,8$, we compute the error quantities\vspace{-0.75mm}
      \begin{align}\label{eq:error_2Dto1D}
        \left.\begin{aligned} 
        \texttt{err}_{\mathbf{v},i}^{\smash{\scaleto{L^\infty L^2}{5pt}}}&\coloneqq \|\mathbf{v}_{h_i}^{\tau_i}-v_{h_i}^{\tau_i}\mathbf{e}_1\|_{L^\infty(I;(L^2(\omega))^2)}\,,\\
          \texttt{err}_{\mathbf{v},i}^{\smash{\scaleto{L^2 \mathbf{F}(\cdot,W^{1,p(\cdot)}_0)}{7pt}}}&\coloneqq \|\mathbf{F}(\cdot,\mathbf{D}\mathbf{v}_{h_i}^{\tau_i})-\mathbf{F}(\cdot,\mathbf{D}(v_{h_i}^{\tau_i}\mathbf{e}_1))\|_{I\times\omega}\,,
          \\
           \texttt{err}_{\pi,i}^{\smash{\scaleto{(\varphi_{\vert\mathbf{D}\mathbf{v}\vert})^*}{7pt}}}&\coloneqq \|(\varphi_{\smash{\vert \mathbf{D}\mathbf{v}_{h_i}^{\tau_i}\vert+\vert \mathbf{D}(v_{h_i}^{\tau_i}\mathbf{e}_1)\vert}})^*(\cdot,\vert\smash{\overline{\pi^{\tau_i}_{h_i}-\Gamma^{\tau_i}x_1}^{\omega}}\vert)\|_{1,I\times\omega}^{1/2}\,,
           \end{aligned}\quad\right\}\quad i=1,\ldots,8\,,
      \end{align}
      where $\mathbf{F}\colon \Omega\times \mathbb{R}^{2\times 2}_{\textup{sym}}\to \mathbb{R}^{2\times 2}_{\textup{sym}}$ is defined by $\mathbf{F}(x,\mathbf{A})\coloneqq \vert \mathbf{A}\vert^{\smash{\frac{p(x)-2}{2}}}\mathbf{A}$ for a.e.\  $x\in \Omega$ and all $\mathbf{A}\in \mathbb{R}^{2\times 2}$ and $\overline{\pi^{\tau_i}_{h_i}-\Gamma^{\tau_i}x_1}^{\omega}\coloneqq \pi^{\tau_i}_{h_i}-\Gamma^{\tau_i}x_1-\smash{\frac{1}{\vert\omega\vert}}(\pi^{\tau_i}_{h_i}-\Gamma^{\tau_i}x_1,1)_{\omega}$.\enlargethispage{5mm}

      In Figure \ref{fig:2Dto1D_errors}\textit{(right)}, for the errors $\texttt{err}_{\mathbf{v},i}^{\smash{\scaleto{L^2 \mathbf{F}(\cdot,W^{1,p(\cdot)}_0)}{7pt}}}$, $i=1,\ldots,8$, and $\texttt{err}_{\pi,i}^{\smash{\scaleto{(\varphi_{\vert\mathbf{D}\mathbf{v}\vert})^*}{7pt}}}$, $i=1,\ldots,8$, we report the error decay rate $\mathcal{O}((\tau_i+h_i)^2)$, $i=1,\ldots,8$, and for the errors $\texttt{err}_{\mathbf{v},i}^{\smash{\scaleto{L^\infty L^2}{5pt}}}$, $i=1,\ldots,8$, we report the error decay rate $\mathcal{O}((\tau_i+h_i)^3)$,  $i=1,\ldots,8$. 

In Figure \ref{fig:2Dto1D_errors2}, the absolute errors between the 1D and 2D approximations of the velocity field and the kinematic pressure at time $t=\pi$ and for the refinement step $i=8$~are~depicted.

    \begin{figure}[H]
        \centering
        \includegraphics[width=\linewidth]{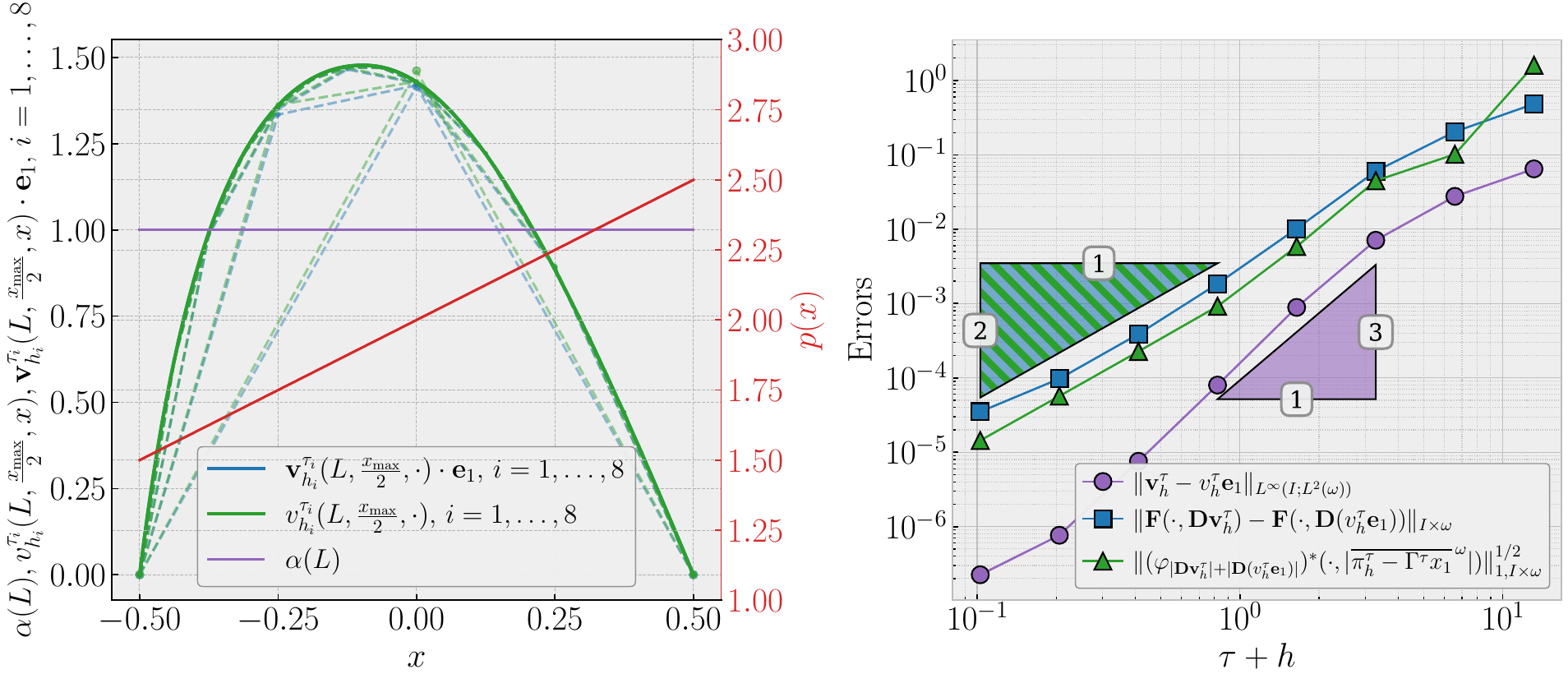}\vspace{-2.5mm}
        \caption{\hspace{-0.15mm}\textit{left:} \hspace{-0.15mm}line \hspace{-0.15mm}plots \hspace{-0.15mm}of \hspace{-0.15mm}the \hspace{-0.15mm}final/initial~\hspace{-0.15mm}flow~\hspace{-0.15mm}rate~\hspace{-0.15mm}${\alpha(L)\hspace{-0.175em}=\hspace{-0.175em}\alpha(0)\hspace{-0.175em}\in\hspace{-0.175em} \mathbb{R}}$~\hspace{-0.15mm}(\mbox{\textcolor{byzantium}{purple}}),~\hspace{-0.15mm}1D~\hspace{-0.15mm}\mbox{approximations} $v_{h_i}^{\tau_i}(L,\frac{x_{\max}}{2},\cdot)=v_{h_i}^{\tau_i}(0,\frac{x_{\max}}{2},\cdot)\colon \Sigma\to \mathbb{R}$, $i=1,\ldots,8$, (\textcolor{denim}{dashed blue}), 2D approximations $\mathbf{v}_{h_i}^{\tau_i}(L,\frac{x_{\max}}{2},\cdot)\cdot \mathbf{e}_1=\mathbf{v}_{h_i}^{\tau_i}(0,\frac{x_{\max}}{2},\cdot)\cdot\mathbf{e}_1\colon \Sigma\to \mathbb{R}$, $i=1,\ldots,8$,  (\textcolor{shamrockgreen}{dashed green}),~and power-law index $p\colon \Sigma\to (1,+\infty)$ (\textcolor{black}{red}); \textit{right:} error plots for the error quantities~in~\eqref{eq:error_2Dto1D} (\textcolor{byzantium}{purple}/\textcolor{denim}{blue}/\textcolor{shamrockgreen}{green}).}
        \label{fig:2Dto1D_errors}
    \end{figure}\vspace{-6.5mm}\enlargethispage{8.5mm}

    \begin{figure}[H]
        \centering
        \includegraphics[width=0.875\linewidth]{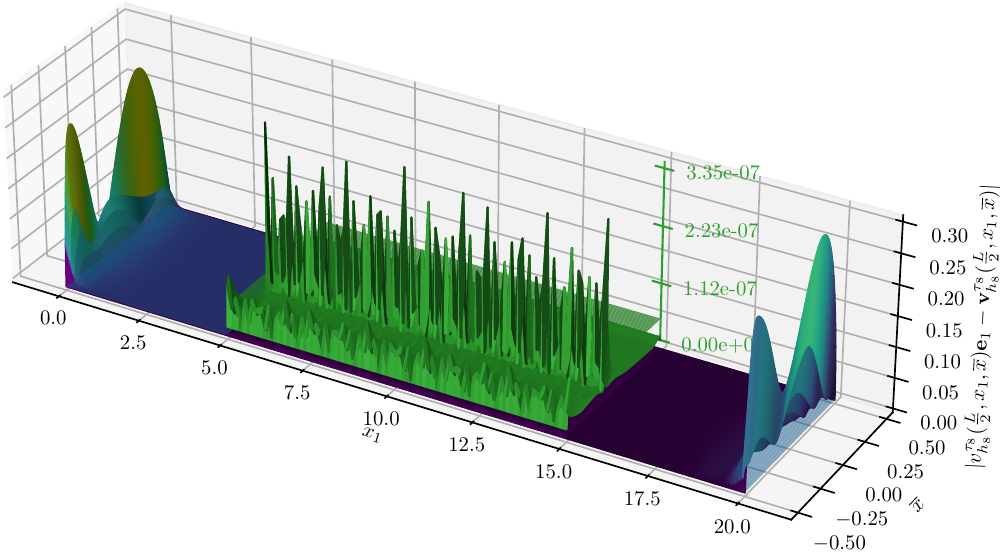}\vspace{-20mm}
         \includegraphics[width=0.875\linewidth]{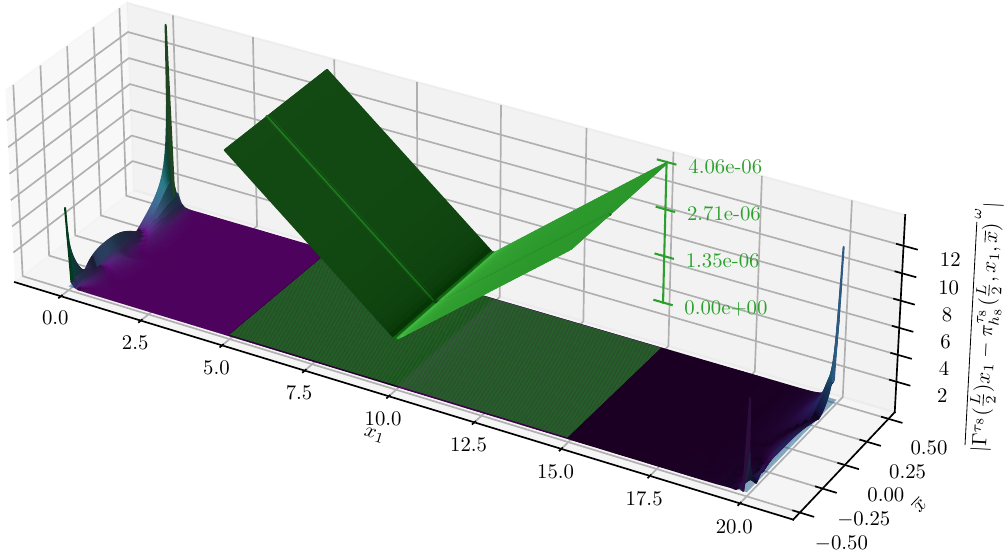}\vspace{-1.5mm}
        \caption{surface plots of the absolute errors (\textcolor{byzantium}{v}\textcolor{denim}{i}\textcolor{shamrockgreen}{r}\textcolor{Goldenrod}{i}\textcolor{byzantium}{d}\textcolor{denim}{i}\textcolor{shamrockgreen}{s}) between the 1D and 2D approximations, where the \textcolor{shamrockgreen}{green}  area represents the portion that is taken into account in the error analysis (\textit{cf}.~\eqref{eq:error_2Dto1D}): \textit{top:} absolute error between $\mathbf{v}_{h_8}^{\tau_8}(\frac{L}{2})\colon \Omega\to \mathbb{R}^2$ and $v_{h_8}^{\tau_8}(\frac{L}{2})\mathbf{e}_1\colon \Omega\to \mathbb{R}^2$; \textit{bottom:}~absolute error between $\smash{\pi_{h_8}^{\tau_8}(\frac{L}{2})}\colon \Omega\to \mathbb{R}$ and $((x_1,\overline{x})\mapsto\smash{\Gamma^{\tau_8}(\frac{L}{2})x_1)}\colon \Omega\to \mathbb{R}$.} 
        \label{fig:2Dto1D_errors2}
    \end{figure}\newpage

    \section*{Conclusions}

\hspace{5mm}We studied simplified smart fluids with variable \textcolor{black}{exponent} power-law. More precisely, we investigated a setting in which it is possible to establish the existence of pulsatile fully-developed solutions, with assigned time-periodic flow rate or pressure drop. The findings generalize those known for constant power-law indices and are obtained through a fully-constructive numerical approach, which is later tested through numerical experiments. The considered geometric setting also makes it possible to determine solutions explicitly --at least in~special~cases~(\textit{i.e.},~for~the~steady~problem, which is a particular time-periodic case). These explicit solutions are the natural counterpart of the Hagen--Poiseuille flow and can be considered as limiting solutions for  long enough~straight~pipes. At the same time, they provide natural guesses for initial data in inflow problems.~The~numerical\linebreak experiments confirm the convergence of the fully-discrete finite-difference/-element discretization to the  solution of the continuous problem; where particular computational effort is made to obtain a time-periodic discrete solution, by applying a Picard iteration on~the~initial/final~\mbox{datum}. The reported experimental results confirm the rapid convergence to pulsatile solutions, which are generalized versions of the classical Womersley solutions of the unsteady Navier--Stokes~equations.  Apart from that, we investigated the convergence of a fully-discrete finite-difference/-element discretization of the \textit{`full'} problem (\textit{i.e.}, (possibly) with convection and a  velocity field non-trivial in directions orthogonal to the axis $\mathbb{R}\mathbf{a}$) towards the fully-developed one, showing --as expected-- better convergence for the velocity field than for the kinematic pressure, and also higher accuracy in the regions away from the entrance-exit cross-sections of the pipe. Therefore, besides the study of Womersley-type flows for smart fluids,  we also constructed~a~robust test case for simulations of electro-rheological fluids.

    \section*{\textcolor{black}{Code availability}}
    
    \hspace{5mm}\textcolor{black}{The legacy-\texttt{FEniCS} implementation and numerical data used for the numerical~\mbox{experiments}~in Section~\ref{sec:experiments}  are archived on \texttt{Zenodo}
(version~v1.0.0, DOI: \href{https://doi.org/10.5281/zenodo.21179010}{10.5281/zenodo.21179010});
see also~\cite{berselli_kaltenbach_2026_pulsatile_smart_fluids_code}.
The development repository is available at
\href{https://github.com/alexkaltenbach/pulsatile-smart-fluids-code}{https://github.com/alexkaltenbach/pulsatile-smart-fluids-code}.}
    
	{\setlength{\bibsep}{0pt plus 0.0ex}\small
		
		\bibliographystyle{aomplain}
		\bibliography{references}
		
	}
    
    \appendix

    \section{Approximation and boundedness properties of \texorpdfstring{$\chi_h$}{chi-h}}\enlargethispage{7mm}\vspace{-1mm}

    \hspace{5mm}In this short appendix, we prove the approximation and stability properties \eqref{eq:chih_approx_stab.1}--\eqref{eq:chih_approx_stab.2} of the function $\chi_h\in V_h$, defined by \eqref{def:chih}. To this end, first note that from Assumption~\ref{ass:PiV}, it follows that there exists  a constant $c_{\Pi}(1)>0$ (\textit{cf}.\ \cite[Thm.~4.6]{DieningRuzicka2007})~such~that for every $w\in W^{1,1}(\Sigma)$, there holds\vspace{-0.75mm}
        \begin{align}\label{lem:chih_bound.3}
            \|w-\Pi_hw\|_{1,\Sigma}+h\|\nabla \Pi_h w\|_{1,\Sigma}\leq c_{\Pi}(1) h\|\nabla w\|_{1,\Sigma}\,.\\[-5.75mm]\notag
        \end{align} 
    From the $L^1(\Sigma)$-approximation and $W^{1,1}(\Sigma)$-stability property of $\Pi_h$ (\textit{cf}.\ \eqref{lem:chih_bound.3}), for every $n>0$, it follows that there exists a constant $c_{\Pi}(p,n)>0$ (\textit{cf}.\ \cite[Lem.\ 3.5, Cor. 3.6]{BreitDieningSchwarzacher2015}) such that for every $w\in W^{1,p(\cdot)}(\Sigma)$, there holds\vspace{-0.5mm}
    \begin{align}\label{lem:chih_bound.1}
            \rho_{p(\cdot),\Sigma}(w-\Pi_hw)+\rho_{p(\cdot),\Sigma}(h\nabla \Pi_hw)&\leq c_{\Pi}(p,n)\bigl(h^n+\rho_{p(\cdot),\Sigma}(h\nabla w)\bigr)\,.\\[-6mm]\notag
    \end{align}
    By means of \eqref{lem:chih_bound.3} and \eqref{lem:chih_bound.1}, we can derive the claimed approximation and stability properties \eqref{eq:chih_approx_stab.1}--\eqref{eq:chih_approx_stab.2} of the function $\chi_h\in V_h$, defined by \eqref{def:chih}.

    \begin{lemma}\label{lem:chih_bound}
        Let $\chi\in W^{1,p(\cdot)}_0(\Sigma)$ with $(\chi,1)_{\Sigma}=1$ and $\chi_h\in V_h$, defined by \eqref{def:chih}. Moreover, assume that $h>0$ is sufficiently small, so that\vspace{-0.5mm}
        \begin{align}\label{lem:chih_bound.-1}
            \smash{c_{\Pi}(1)h\|\nabla \chi\|_{1,\Sigma}\leq \tfrac{1}{2}}\,.\\[-6mm]\notag
        \end{align}
       Then, there holds\vspace{-1mm}
        \begin{subequations}\label{lem:chih_bound.0}
    \begin{align}\label{lem:chih_bound.0.1}
        \rho_{p(\cdot),\Sigma}(\chi-\chi_h)&\lesssim h^n+\rho_{p(\cdot),\Sigma}(h\|\nabla \chi\|_{1,\Sigma}\chi)+\rho_{p(\cdot),\Sigma}(h\nabla \chi)\,,\\
         \rho_{p(\cdot),\Sigma}(\nabla \chi_h)&\lesssim h^{n-1}+\rho_{p(\cdot),\Sigma}(\|\nabla \chi\|_{1,\Sigma}\chi)+\rho_{p(\cdot),\Sigma}(\nabla \chi)\,,\label{lem:chih_bound.0.1.2}\\
        \|\chi_h\|_{\Sigma}&\lesssim \|\nabla \chi\|_{p(\cdot),\Sigma}\,,\label{lem:chih_bound.0.2}
    \end{align}
    \end{subequations}
    where the implicit constant in $\lesssim$ depends on $n$, $p$, and the choice of the finite element~space~$V_h$.
    \end{lemma}

    \begin{proof}
        \emph{ad \eqref{lem:chih_bound.0.1}.} 
        By the $L^1(\Sigma)$-approximation property of $\Pi_h$ (\textit{cf}.\ \eqref{lem:chih_bound.3}) and \eqref{lem:chih_bound.-1},~we~have~that
        \begin{align}\label{lem:chih_bound.3.new}
            \vert (\chi-\Pi_h\chi,1)_{\Sigma}\vert \leq c_{\Pi}(1)h\|\nabla \chi\|_{1,\Sigma}\leq \tfrac{1}{2}\,,
        \end{align} 
       which~implies~that
        \begin{align}\label{lem:chih_bound.4}
            \begin{aligned}
            1=\vert(\chi,1)_{\Sigma}\vert &\leq\vert (\chi-\Pi_h\chi,1)_{\Sigma}\vert+\vert (\Pi_h\chi,1)_{\Sigma}\vert 
            \\&\leq \tfrac{1}{2}+\vert (\Pi_h\chi,1)_{\Sigma}\vert
            \end{aligned} 
        \end{align} 
        If we combine \eqref{lem:chih_bound.3} and \eqref{lem:chih_bound.4},  we arrive at $\smash{\tfrac{\vert(\chi-\Pi_h\chi,1)_{\Sigma}\vert}{\vert(\Pi_h\chi,1)_{\Sigma}\vert}\leq 1}$, which, due to $(\chi,1)_{\Sigma}=1$,~yields~that
        \begin{align}\label{lem:chih_bound.6}
            \begin{aligned} 
           \rho_{p(\cdot),\Sigma}(\Pi_h\chi-\chi_h)&=\rho_{p(\cdot),\Sigma}\big(\Pi_h\chi\tfrac{(\chi-\Pi_h\chi,1)_{\Sigma}}{(\Pi_h\chi,1)_{\Sigma}}\big)
           \\&\leq \rho_{p(\cdot),\Sigma}(c_{\Pi}(1)\,h\|\nabla \chi\|_{1,\Sigma}\Pi_h\chi)\,.
           \end{aligned}
        \end{align}
        Similarly, due to $\smash{\tfrac{1}{\vert(\Pi_h\chi,1)_{\Sigma}\vert}\leq 2}$,
        we have that
        \begin{align}\label{lem:chih_bound.7}
            \begin{aligned} 
            \rho_{p(\cdot),\Sigma}(\nabla\chi_h)&=  \rho_{p(\cdot),\Sigma}\big(\tfrac{1}{\vert(\Pi_h\chi,1)_{\Sigma}\vert}\nabla\Pi_h\chi\big) 
            \\&\leq \rho_{p(\cdot),\Sigma}(2\nabla\Pi_h\chi)\,.\end{aligned}
        \end{align} 
        In summary, combining \eqref{lem:chih_bound.1}, \eqref{lem:chih_bound.6}, and \eqref{lem:chih_bound.7}, we arrive at the claimed approximation and stability estimates \eqref{lem:chih_bound.0.1} and \eqref{lem:chih_bound.0.1.2}.

        \emph{ad \eqref{lem:chih_bound.0.2}.} Due to $\smash{\tfrac{1}{\vert(\Pi_h\chi,1)_{\Sigma}\vert}\leq 2}$, the embedding $\smash{W^{1,p^-}_0(\Sigma)}\hookrightarrow L^2(\Sigma)$, the $\smash{W^{1,p^-}_0(\Sigma)}$-stability of $\Pi_h$ (\textit{cf}.\ \cite[Cor.\ 4.8]{DieningRuzicka2007}), and the embedding $L^{p(\cdot)}(\Sigma)\hookrightarrow L^{p^-}(\Sigma)$ (\textit{cf}.\ \cite[Cor.\ 3.3.4]{DHHR2011}),~we~have~that
    \begin{align*}
         \|\chi_h\|_{\Sigma}&=\big\|\tfrac{1}{\vert(\Pi_h\chi,1)_{\Sigma}\vert}\Pi_h\chi\big\|_{\Sigma}\\&\leq2\|\Pi_h \chi\|_{\Sigma}\\&\lesssim\|\nabla\Pi_h \chi\|_{p^-,\Sigma}
        \\&\lesssim \|\nabla \chi\|_{p^-,\Sigma}
       \\&\lesssim \|\nabla \chi\|_{p(\cdot),\Sigma}\,,
    \end{align*}
    which is the claimed stability estimate \eqref{lem:chih_bound.0.2}.
    \end{proof}
	
\end{document}